\documentclass[11pt]{article}
\usepackage{bbm}
\usepackage{geometry}
\usepackage{color}
\usepackage{amsfonts}
\usepackage{algorithm}
\usepackage{algorithmic}
\usepackage{latexsym, amssymb, amsmath, amscd, amsthm, amsxtra}
\usepackage{mathtools}
\usepackage{enumerate}
\usepackage[all]{xy}
\usepackage{mathrsfs}
\usepackage{fancyhdr}
\usepackage{listings}
\usepackage{hyperref}
\usepackage{enumitem}

\def\11{\mathbbm{1}}
\def\Gc{\mathcal{G}}

\def\Es{\mathsf E}
\def\ER{Erd\H{o}s-R\'enyi\ }
\def\Ac{\mathcal A}

\def\Vs{\mathsf V}
\def\Gs{\mathsf G}
\def\vs{\mathsf v}

\def\Rs{\mathsf R}
\def\Eb{\mathbb E}

\newcommand{\Pb}{\mathbb P}

\newtheorem{thm}{Theorem}[section]

\newtheorem{proposition}[thm]{Proposition}

\newtheorem{lemma}[thm]{Lemma}
\newtheorem{cor}[thm]{Corollary}

\theoremstyle{definition}
\newtheorem{remark}{Remark}[section]

\numberwithin{equation}{section}

\newenvironment{definition}[1][Definition]{\begin{trivlist}
		\item[\hskip \labelsep {\bfseries #1}]}{\end{trivlist}}



\hypersetup{colorlinks=true,linkcolor=blue}

\begin{document}
	
	\title{The Algorithmic Phase Transition of Random Graph Alignment Problem}
	
	\author{Hang Du\\MIT\and Shuyang Gong\\Peking University\and Rundong Huang\\Peking University}
	
	\maketitle

\begin{abstract}
    We study the graph alignment problem over two independent Erd\H{o}s-R\'enyi random graphs on $n$ vertices, with edge density $p$ falling into two regimes separated by the critical window around $p_c:=\sqrt{\log n/n}$. Our result reveals an algorithmic phase transition for this random optimization problem: polynomial-time approximation schemes exist in the sparse regime, while statistical-computational gap emerges in the dense regime. Additionally, we establish a sharp transition on the performance of online algorithms for this problem when $p$ is in the dense regime, resulting in a $\sqrt{8/9}$ multiplicative constant factor gap between achievable solutions and optimal solutions. 
\end{abstract}

\section{Introduction and main results}\label{sec1}
The \emph{Graph Alignment Problem} (GAP) for two simple graphs with the same number of vertices involves finding a vertex bijection which maximizes the size of overlap of these two graphs. Specifically, given $G(V,E)$ and $\mathsf G(\mathsf V,\mathsf E)$ with $|V|=|\mathsf V|$, the goal is to find a bijection $\pi:V\to\mathsf V$ that maximizes the expression
	\[
	\sum_{v_i\neq v_j}\mathbf{1}_{(v_i,v_j)\in E}\mathbf{1}_{(\pi(v_i),\pi(v_j))\in \mathsf E}\,.
	\] 
Closely related to the \emph{Maximum Common Subgraph Problem} (MCS), GAP is an important but challenging  combinatorial optimization problem, which plays essential roles in various applied fields, such as computational biology \cite{SXB08,VCP15}, social networking \cite{NS08,NS09}, computer vision \cite{BBM05,CSS07} and natural language processing \cite{HNM05}. The study of efficient algorithms for solving GAP exactly or approximately has been conducted extensively through the past decades.

Unfortunately, as a special case of the \emph{Quadratic Assignment Problem} (QAP) \cite{PRW94,BCPP98}, the exact solution of GAP is known to be NP-hard, meaning that every problem in NP can be reduced to it in polynomial time. Moreover, it was shown in \cite{MMS10} that approximating QAP within a factor $2^{\log^{1-\varepsilon}(n)}$ for any $\varepsilon>0$ is also NP-hard, so finding near-optimal solutions for GAP efficiently also seems to be out of reach in general.

Due to the worst-case hardness result, analyses of GAP in existing literature are typically restricted to specific classes of instances, such as sparse graphs (e.g. \cite{KHK11,KP11}) or correlated \ER graphs (e.g. \cite{PG11, YG13, LFP14, KHG15, FQRM+16, SGE17, BCL19, DMWX21, FMWX22a, FMWX22b, BSH19, CKMP19, DCKG19, MX20, GM20, MRT23, GMS22+, MWXY23+, DL23+}, which we will discuss in details in Section~\ref{subsec-BG}). In this paper we consider GAP over typical instances for a pair of \emph{independent} \ER graphs with the same edge density, which we refer to as the random GAP. Our primary objective is to find near-optimal solutions within a multiplicative constant factor. 

Formally we fix $n\in \mathbb{N}, p\in(0,1)$, denote $V=\{v_1,\dots,v_n\},\Vs=\{\mathsf{v}_1,\dots,\mathsf{v}_n\}$ and let $G,\Gs$ be independent samples of \ER graphs on $V,\Vs$ with edge density $p$ (whose distribution we denote as $\mathbf G(n,p)$), respectively. For $1\le i\neq j\le n$, we let
	\begin{equation*}
		G_{i,j}=G_{j,i}=\mathbf{1}_{(v_i,v_j)\in E},\quad \Gs_{i,j}=\Gs_{j,i}=\mathbf{1}_{(\mathsf{v}_i,\mathsf{v}_j)\in \Es}\,.
	\end{equation*}
	Then, $\{G_{i,j}\}_{1\le i<j\le n},\{\Gs_{i,j}\}_{1\le i<j\le n}$ are independent Bernoulli variables with parameter $p$ by definition. %We denote  the symmetric group by $\operatorname{S}_n$. 
 For any permutation $\pi\in \operatorname{S}_n$, %it naturally corresponding to a bijection $\Pi$ from $V$ to $\Vs$ by setting $\Pi(v_i)=\vs_{\pi(i)}$ for any $1\le i\le n$. With this notation, 
	define 
	\[
	\operatorname{O}(\pi)=\sum_{1\le i<j\le n}G_{i,j}\Gs_{\pi(i),\pi(j)}
	\]
	as the overlap of $G,\Gs$ along with the vertex correspondence $v_i\mapsto \vs_{\pi(i)}$. We will study the optimization problem of $\operatorname{O}(\pi)$ under typical realizations of the random inputs $G$ and $\Gs$. %To this end, sometimes we will actually work on certain fixed realization of $\Gs=(\Vs,\Es)$ which satisfies certain \emph{$p$-regularity condition} (which holds for an \ER graph $\mathbf G(n,p)$ with high probability; see Definition and Proposition below), and simply exploits the randomness of the sole graph $G$. From this perspective, 
	we note that for any $\pi\in \operatorname{S}_n$, the expectation of $\operatorname{O}(\pi)$ is given by $E_{n,p}=\binom{n}{2}p^2$. We will see later that when $p$ is large enough, $E_{n,p}$ is indeed the leading order term of $\max_{\pi\in\operatorname{S}_n} \operatorname O(\pi)$, so it would be reasonable to consider the centered version of the family of variables
	\begin{equation}\label{eq-centered-version}
		\{\widetilde{\operatorname{O}}(\pi)\}_{\pi\in \operatorname{S}_n}\stackrel{\operatorname{def}}{=}\{\operatorname{O}(\pi)-E_{n,p}\}_{\pi\in \operatorname{S}_n}\,.
	\end{equation}

	For positive functions $f,g:\mathbb N\to \mathbb R^+$, we use standard notations like $f=o(g),O(g),\Omega(g)$ to mean that $f/g$ is converging to $0$, bounded above from $\infty$, bounded below from $0$, respectively. %Also we use $f\ll g$ (resp. $f\gg g$) to mean $f/g$ tends to $0$ (resp. $\infty$) as $n\to\infty$, and $f\asymp g$ stands for that $cf\le g\le Cf$ holds for some universal constants $0<c<C$.
 %(maybe changed to) 
 Additionally, we use the notation $f\ll g$ (resp. $f\gg g$) to indicate that $f/g$ tends to $0$ (resp. $\infty$) as $n\to\infty$. We also use $f\asymp g$ to denote that $cf\le g\le Cf$ for some universal constants $0<c<C$. In this paper, we focus on two regimes of the parameter $p$ that are separated by the critical window around $p_c=\sqrt{\log n/n}$. Formally, we introduce the following conventions:
	\begin{itemize}
		\item we say $p$ is in the \textbf{sparse regime}, if $\log n/n\le p\ll p_c$;
		\item we say $p$ is in the \textbf{dense regime}, if $p_c\ll p\ll 1$.
	\end{itemize}

    Our first result establishes the asymptotics of the maximum centered overlap.

\begin{thm}\label{thm-info}
	Let $(G,\Gs)\sim \mathbf G(n,p)^{\otimes 2}$, then the following hold.\\
\noindent \emph{(i)} When $p$ is in the sparse regime, for $S_{n,p}{=}\frac{n\log n}{\log \big({\log n}/{np^2}\big)}$, we have
			\begin{equation}\label{eq-sparse-info}
				\max_{\pi\in\operatorname{S}_n}\widetilde{\operatorname{O}}(\pi)\Big\slash S_{n,p}\stackrel{\text{in probability}}{\rightarrow} 1\,, \text{as }n\to \infty\,.
            \end{equation}
			\noindent \emph{(ii)} When $p$ is in the dense regime, for $D_{n,p}=\sqrt{n^3p^2\log n}$, we have
			\begin{equation}\label{eq-dense-info}
				\max_{\pi\in\operatorname{S}_n}\widetilde{\operatorname{O}}(\pi)\Big\slash D_{n,p}\stackrel{\text{in probability}}{\rightarrow} 1\,, \text{as }n\to\infty\,.
			\end{equation}
%	In particular, \eqref{eq-sparse-info} and \eqref{eq-dense-info} holds for $G,\Gs$ independently sampled from $\mathbf G(n,p)$.
\end{thm}
\begin{remark}
	A simple computation gives that
		\[
		p\ll p_c\Rightarrow S_{n,p}\gg E_{n,p}\quad\text{and}\quad
		p\gg p_c\Rightarrow D_{n,p}\ll E_{n,p}\,.
		\]
		From these relations, it follows that the typical asymptotic behavior of $\max_{\pi\in \operatorname{S}_n}\operatorname{O}(\pi)$ is governed by $S_{n,p}$ in the sparse regime. While in the dense regime, it is captured by the expectation $E_{n,p}$. Moreover, a straightforward application of concentration inequalities suggests that in the dense regime, all the $\operatorname O(\pi)$ values concentrate around $E_{n,p}$. This explains why we focus on the centered version $\widetilde{\operatorname{O}}(\pi)$. Otherwise, the problem would become trivial in the dense regime, both in terms of informational and computational aspects.
	\end{remark}
We now shift our attention to the algorithmic aspect of this problem, and we start by presenting the precise definition of graph alignment algorithms.
\begin{definition}[Graph alignment algorithms]
	%	For a (potentially randomized) \emph{graph alignment algorithm}, we mean an algorithm $\mathcal A=\Ac(\Omega,f,\{\mathbb{P}_{G,\Gs}\})$, where $\Omega$ is some abstract sample space, $f$ is a deterministic function from $\operatorname{G}_n\times\operatorname{G}_n\times \Omega$ to $\operatorname{S}_n$,  $\{\mathbb{P}_{G,\Gs}\}$ is a family of probability measures on $\Omega$ indexed by $\operatorname{G}_n\times \operatorname{G}_n$, and $\mathcal A$ is defined as follows: for an input $(G,\Gs)$, the output $\mathcal A(G,\Gs)=f(G,\Gs,\omega)\in \operatorname{S}_n$ with $\omega$ sampled from $\mathbb{P}_{G,\Gs}$. 
		%	Denote the set of graph alignment algorithm as $\operatorname{GAA}$. 
  Denote $\mathfrak{G}_n$ as the collection of all simple graphs on $n$ vertices. A \emph{graph alignment algorithm} is a (potentially randomized) algorithm denoted by $\mathcal{A}=\mathcal(\Omega,f,{\mathbb{P}_{G,\Gs}})$. Here, $\Omega$ is an abstract sample space, $f$ is a deterministic function from $\mathfrak{G}_n\times\mathfrak{G}_n\times \Omega$ to $\operatorname{S}_n$, and ${\mathbb{P}_{G,\Gs}}$ is a family of probability measures on $\Omega$ indexed by $\mathfrak{G}_n\times \mathfrak{G}_n$. Given an input $(G,\Gs)$, the algorithm first samples the internal randomness $\omega$ from $\mathbb{P}_{G,\Gs}$ and then outputs $\mathcal{A}(G,\Gs)=f(G,\Gs,\omega)\in \operatorname{S}_n$.

The set of graph alignment algorithms will be denoted as $\operatorname{GAA}$.
	\end{definition} 
We aim to determine the power limit of \emph{efficient algorithms} in $\operatorname{GAA}$ for typical instances of $(G,\Gs)\sim \mathbf G(n,p)^{\otimes 2}$. For $G,\Gs$ independently sampled from $\mathbf{G}(n,p)$, let $\operatorname{S}_\beta(G,\Gs)$ be the set of permutations $\pi$ which satisfy $\widetilde{\operatorname{O}}(\pi)\ge\beta S_{n,p}$ (resp. $\widetilde{\operatorname{O}}(\pi)\ge \beta D_{n,p}$) for $p$ in the sparse regime (resp. the dense regime). Namely, $\operatorname{S}_\beta(G,\Gs)$ is the set of asymptotically $\beta$-optimal solutions to the random GAP. 
	 %Specifically, we seek to identify the values of $\beta$ for which there exists a polynomial-time algorithm $\Ac\in \operatorname{GGA}$ such that the probability $\Pb[\mathcal{A}(G,\Gs)\in \operatorname{S}_\beta(G,\Gs)]$ is non-vanishing or close to one. Here, the probability is taken over the random input $(G,\Gs)\sim \mathbf{G}(n,p)^{\otimes 2}$, together with the additional randomness in the algorithm itself.
	%In light of Theorem~\ref{thm-info}, we will informally call $p\ll p_c$ as the \emph{sparse regime}, $p_c\ll p\ll 1$ as the \emph{dense regime} and $\Omega(1)\le p\le 1-\Omega(1)$ as the \emph{super dense regi
		Our next result states that for $p$ in the sparse regime,  polynomial-time algorithms may successfully find solutions in $\operatorname{S}_\beta(G,\Gs)$ with high probability for $\beta$ arbitrarily close to $1$.
\begin{thm}\label{thm-PTAS}When $p$ is in the sparse regime, there is a polynomial-time algorithm $\mathcal A\in \operatorname{GAA}$ (depending on $n,p,\varepsilon$) such that as $n\to \infty$, 
		\[
		\mathbb P[\mathcal A(G,\Gs)\in \operatorname{S}_{1-\varepsilon}(G,\Gs)]=1-o(1)\,,
		\]
		where $\mathbb P$ is taken over $(G,\Gs)\sim \mathbf G(n,p)^{\otimes 2}$ as well as the internal randomness of $\mathcal A$.
\end{thm}
\begin{remark}\label{rmk-connection-to-DDG24}
In \cite{DDG22}, the authors developed a polynomial-time approximation scheme for the random GAP when the parameter $p$ satisfies $p=n^{-\alpha+o(1)}$ for some constant $\alpha\in (1/2,1]$ with $p\geq \log n/n$ (the maximal overlap therein is asymptotically $\frac{n}{2\alpha-1}$, which aligns with the result of Theorem~\ref{thm-info}-(i)). Therefore, the analysis in this paper is presented only for the edge case $p=n^{-1/2+o(1)}$, $p\ll p_c$.

Unsurprisingly, the algorithms presented in this paper are closely related to those appeared in \cite{DDG22}, and we now elaborate more on this connection. Consider a naive greedy algorithm that sequentially determines $\pi(1),\pi(2),\dots,\pi(n)$. For each $k$, $\pi(k)$ is selected from the remaining unmatched vertices to maximize $\sum_{i<k}G_{i,k}\Gs_{\pi(i),\pi(k)}$. In [16], the authors first noticed that this naive algorithm indeed achieves near-optimal performance for certain specific values of $\alpha$. Inspired by this, the authors extended the simple greedy matching idea to design more sophisticated matching algorithms for general $\alpha\in (1/2,1]$.

Following the same spirit, our efficient algorithms for finding near optimal matchings in the regime $p=n^{-1/2+o(1)},p\ll p_c$ are essentially minor variants of the naive greedy matching algorithm (see Section~\ref{subsec-algo-framework} for the precise definition). %Hence in this regime, intriguingly, such an easy algorithm is information-theoretically near-optimal.
Additionally, the idea extends into the dense regime $p\gg p_c$, where we will analyze the same algorithms to prove Theorem~\ref{thm-algo}-(i) below. %This suggests that in the dense regime, although the greedy algorithm is suboptimal, its performance reaches the computational limit of random GAP (see Theorem~\ref{thm-algo}-(ii) and the discussions therein for more details). 

\end{remark}

According to Theorem~\ref{thm-PTAS}, there is no statistical-computational gap in the sparse regime. However, in the dense regime, algorithmic barrier seems to emerge. We present evidence for such a claim by establishing the so-called \emph{overlap-gap property} in the dense regime, and by illustrating the failure of stable algorithms via this property. We state the result informally as below, and the precise version will be given in Section~\ref{subsec-two-OGP}.
	\begin{thm}[\textbf{informal}]\label{thm-OGP}
		There exists a universal constant $\beta_0<1$ such that for any $p$ in the dense regime, the solution space $\operatorname{S}_{\beta_0}(G,\Gs)$ satisfies 2-OGP with high probability (with respect to $(G,\Gs)\sim \mathbf G(n,p)^{\otimes 2}$). As a result, no stable algorithm in $\operatorname{GAA}$ can find solutions in $\operatorname{S}_{\beta_0}(G,\Gs)$ with (sufficiently) high probability in the dense regime.
	\end{thm}
	
 Building upon the existence of the statistical computational gap, our aim is to determine the precise threshold of optimal computationally achievable solutions, commonly referred to as the \emph{computational threshold}. While lower bound on this threshold can be obtained by constructing and analyzing specific efficient algorithms, establishing a non-trivial upper bound for the entire class of polynomial-time algorithms poses a major challenge. However, there are strong indications and widely-used approaches that demonstrate algorithmic limitations beyond certain thresholds within specific sub-classes. In this paper, we narrow our focus to a specific class of algorithms called \emph{online algorithms}, which are defined as follows:
	
	\begin{definition}[Online algorithms]\label{def-online-algo}
	%	We define an algorithm $\mathcal A\in \operatorname{GAA}$ as an online algorithm if the output $\pi^*$ of $\mathcal A$ satisfies the following conditions:
 An algorithm $\mathcal{A}\in \operatorname{GAA}$ is considered as an online algorithm if $\mathcal A$ outputs a permutation $\pi^*$ following the rules that
 %\red{An algorithm $\mathcal{A}\in \operatorname{GAA}$ is considered an online algorithm if its output $\pi^*$ satisfies the following conditions:}
		\begin{itemize}
			\item The values of $\pi^*(1),\pi^*(2),\dots,\pi^*(n)$ are determined sequentially.
			\item For each $1\le k\le n$, once $\pi^*(1),\dots,\pi^*(k-1)$ are determined, $\pi^*(k)$ is a random variable (with internal randomness of the algorithm itself) that takes a value in $[n]\setminus \{\pi^*(1),\dots,\pi^*(k-1)\}$ with a distribution $\Pb_k$ determined by $\{G_{ij}\}_{1\le i<j\le k}$, $\{\Gs_{ij}\}_{1\le i<j\le n}$, and $\pi^*(1),\dots,\pi^*(k-1)$.
		\end{itemize} 
		The set of online algorithms will be denoted as $\operatorname{OGAA}$.
	\end{definition}
	Intuitively, one might think of $G$ as a graph that is constructed online, vertex by vertex, while $\Gs$ is a pre-specified offline graph. From this perspective, the above definition implies that an online algorithm must determine $\pi^*(k)$ using only the information before the $k$-th vertex in $G$ is constructed.
 
    We focus on the power limit of online algorithms and demonstrate a computational phase transition in the dense regime at $$\beta_c=\int_0^1\sqrt{2x}\operatorname{d}\!x=\frac{2\sqrt{2x^3}}{3}\Big|^1_0=\sqrt{\frac 89}\,.$$

	\begin{thm}\label{thm-algo}
	For $p$ in the dense regime and any $\varepsilon>0$, the following hold.\\
\emph{(i)} Assume further that $p\le 1/(\log n)^{4}$, then there exists $\Ac^*\in \operatorname{OGGA}$ which runs in $O(n^3)$ times, such that as $n\to\infty$, 
			\[
			\Pb\big[\Ac^*(G,\Gs)\in \operatorname{S}_{\beta_c-\varepsilon}(G,\Gs)\big]=1-o(1)\,,
			\]
			i.e. $\mathcal A^*$ finds $(\beta_c-\varepsilon)$-optimal solutions for typical instances of $(G,\Gs)\sim\mathbf G(n,p)^{\otimes 2}$.
\\
\emph{(ii)} There exists $c=c(\varepsilon)>0$ such that for any $\mathcal{A}\in \operatorname{OGAA}$, 
			\[
			\Pb\Big[\mathbb{Q}\big[\mathcal A(G,\Gs)\in \operatorname{S}_{\beta_c+\varepsilon}(G,\Gs)\big]\ge \exp(-cn\log n)\Big]=o(1)\,,
			\]
			where $\mathbb Q$ denotes the internal randomness of the algorithm $\mathcal A$. In other words, no online algorithm can find $(\beta_c+\varepsilon)$-optimal solutions with probability exceeding $\exp\big(-\Omega(n\log n)\big)$ under typical instances of $(G,\Gs)\sim \mathbf G(n,p)^{\otimes 2}$. %n.given that $\Gs$ satisfies $p$-regularity conditio
	\end{thm}
	\begin{remark}\label{rmk-thm-algo}
Several remarks about Theorem~\ref{thm-algo} are provided below.
\\\noindent
(i) As mentioned earlier in Remark~\ref{rmk-connection-to-DDG24}, the algorithm $\mathcal A^*$ is a minor variant of the naive greedy matching algorithm. Therefore, Theorem~\ref{thm-algo} implies that this simple algorithm achieves the asymptotically optimal performance among online algorithms. Additionally, the assumption $p\le 1/(\log n)^{4}$ is made for technical reasons. Indeed, we expect that a similar result holds for any $p_c\ll p\ll 1$ (see Section~\ref{subsec-algo-framework} for a heuristic derivation and Remark~\ref{rmk-p=o(1)} for more technical discussions).
\\\noindent
(ii) The upper bound of success probability $\exp\big(-\Omega(n\log n)\big)$ in Theorem~\ref{thm-algo}-(ii) is exponentially tight up to a multiplicative constant. This is because even a trivial guessing would find solutions in $\operatorname{S}_{\beta_c+\varepsilon}(G,\Gs)$ with probability at least $1/n!>\exp(-n\log n)$ under typical instances.
\\\noindent
(iii) In Theorem~\ref{thm-algo}-(ii), we present the hardness result specifically for online algorithms, as the proof in this case is relatively straightforward and intuitive. Nevertheless, the same hardness result extends to a broader class of algorithms, termed greedy local iterative algorithms. Informally, these algorithms operate in a local and iterative manner, where each iteration maximally exploits the advantages of the current step, potentially at the cost of introducing adverse effects on subsequent steps (see Definition~\ref{def-GLIA} and Proposition~\ref{prop-extended-hardness} for precise details). For instance, the algorithms in \cite{DDG22}, while not online, still fall within the framework of greedy local iterative algorithms.

It is worth noting that for many null-type random optimization problems (as opposed to planted problems), greedy local iterative algorithms are believed to capture the full power of efficient algorithms. A prominent example is the max-clique problem in $G \sim \mathcal{G}(n, 1/2)$, where a simple greedy algorithm is widely believed to be computationally optimal. Other examples where greedy local iterative algorithms achieve state-of-the-art computational performance include the largest submatrix problem~\cite{GL18} and the random discrepancy minimization problem~\cite{GKPX23+}. Based on this perspective, we expect that $\beta_c$ indeed represents the true computational threshold for the random GAP in the dense regime.
\\\noindent
(iv) It is plausible that the framework introduced in \cite{HS21} can be applied to exclude the possibility of any sufficiently stable algorithm in $\operatorname{GAA}$ finding solutions beyond the threshold $\beta_c$ with non-vanishing probability. Moreover, building on the ideas from \cite{HS21}, a more recent work \cite{HSS25} characterizes the stable algorithm threshold for the random perceptron model in terms of a stochastic control problem. However, the stochastic control problem possesses a very abstract form, making it difficult to directly compute the closed form threshold. It is likely that a similar stochastic control problem formulation can be obtained for the random GAP, but pinpointing the precise threshold for stable algorithms (particularly, determining whether it equals to $\beta_c$) remains an intriguing avenue for future research.
\end{remark}

 Based on the aforementioned theorems, we establish sharp informational thresholds in both regimes, as well as an algorithmic phase transition as $p$ changes from one regime to the other. To provide a comprehensive overview, let us briefly discuss the problem when $p$ falls into other regimes.
Firstly, when $p$ is below the threshold $\log n/n$, the asymptotic value of the typical maximal overlap is given by $cn$ for some constant $c=c(p)\le 1$. However, both graphs are likely to shatter into components with various sizes in this regime, posing difficulty for determining the exact value of $c$ as well as understanding the algorithmic picture of the random GAP.
Secondly, it would be interesting to consider the regime when $p=\Omega(1)$ and we believe this regime shares a more or less similar picture with the dense regime. However, our work does not aim to cover this specific regime as it does differ in specific details from the dense regime, both in terms of informational and computational perspectives (see the proof of Proposition \ref{prop-two-point-est} and Proposition \ref{prop-forbidden-structure}). %Our work does not aim to cover this specific regime.
Thirdly, arguably the most intriguing regime is when p is within the critical window and interpolates between the sparse and dense regimes, where multiple phase transitions (in terms of informational and computational results) occur. Unfortunately, our current methods are tailored specifically for the sparse and dense regimes, and they do not establish a sharp threshold for $p$ within the critical window. Further insights beyond the scope of this paper are necessary to fully understand how these phase transitions take place in this regime.
 %in the sparse regime ensures the connectivity of $G$ and $\Gs$ with high probability. This assumption is natural when considering the Graph Alignment Problem (GAP), as without it, one could simply match between the components. 

	\subsection{Backgrounds and related works}\label{subsec-BG}
	Our motivation for considering the random GAP is twofold. On the one hand, the GAP for a pair of independent \ER graphs arises naturally in the study of correlated random graph models, which is an extensively studied topic in the field of combinatorial statistics in recent years. A deeper understanding of either the informational threshold or the computational threshold for random GAP may also shed lights on the correlated model itself. 
	On the other hand, there is inherent intrigue in gaining insights into the computational complexity of such a natural and important problem. Although the computational intractability for the worst-case scenario is already known (primarily due to certain highly structured ``bad" instances), the situation for the average-case setting remains unclear. 
	We next delve into these two aspects that motivate our current work in greater detail.
	
	\noindent\textbf{Correlated random graph model.}
	The \emph{correlated random graph model} refers to a pair of correlated \ER random graphs with the same number of vertices, where the correlation between them is determined by a \emph{latent} vertex bijection. The study of the correlated model primarily focuses on recovering the hidden correspondence based solely on the topological structures of these two graphs, as well as on the closely related correlation detection problem. Substantial progress has been made in recent years, including information-theoretic analysis \cite{CK16, CK17, HM20, WXY20, WXY21, DD22a, DD22b} and proposals for various efficient algorithms \cite{PG11, YG13, LFP14, KHG15, FQRM+16, SGE17, BCL19, DL22+, DMWX21, FMWX22a, FMWX22b, BSH19, CKMP19, DCKG19, MX20, GM20, MRT23, MWXY21+, GMS22+, MWXY23+, DL23+}, etc. Currently, the community has achieved a comprehensive understanding of the informational thresholds for the correlated random graph model. However, a substantial statistical-computational gap remains and the precise computational threshold has yet to be determined.
	
	The initial exploration of the GAP over two independent instances of \ER graphs stemmed from the study of correlation detection for a pair of random graphs. This can be formulated as a hypothesis testing problem, where, under the null hypothesis, the two graphs are independently sampled, while under the alternative hypothesis, the pair is sampled from the correlated law. The authors of \cite{WXY20} were the first to investigate this problem, and they introduced the maximal overlap of these two graphs as the testing statistic, which is where the concept of random GAP arises. However, the authors only derived an informational negative result using a straightforward first moment method, which was sufficient for their purposes. Subsequently, in \cite[Section 7, Open question Q1]{GML20}, the authors formally formulated the random GAP and sought to determine the exact informational thresholds for different regimes of $p$. One of the main objectives of this paper is to provide an answer to this question.
	\\
	\noindent\textbf{Computational complexity of random optimization problems.}
\emph{Random optimization problems} often refer to solving optimization problems when the input instances are generated randomly. These problems arise in various fields, including computer science, operations research, and statistical physics. Due to the difficulty caused by non-convexity and high--dimensionality, the computational complexity of random optimization problems is currently an active but challenging research area. As mentioned earlier, while worst-case hardness results for these problems are well-established, the problem of average-case complexity (captured by the computational difficulty associated with solving the problem over randomly generated data) can be much more intricate.
	
	As seen for many random optimization problems, statistical-computational gaps may exist, indicating that efficient algorithms encounter barriers below the optimal threshold. Notable examples include the hidden clique problem and the closely related problem of densest submatrix detection (for an overview, see \cite{wu_xu_2021}). That being said, there is also the scenario where a polynomial-time algorithm is available for finding near-optimal solutions for typical instances of random optimization problems; in this case polynomial-time algorithms must exploit specific structural properties thanks to the randomness since these problems are NP-hard in the 
worst case.  A successful example of designing a polynomial-time approximation scheme is given in \cite{Subag21}, where the author constructed a greedy algorithm that finds near ground states of certain Gaussian processes on the $n$-dimensional sphere with high probability. Inspired by this work, \cite{Montanari19} presented an efficient approximation scheme for finding near ground state of the Sherrington-Kirkpatrick model (as well as a broader class of spin glass models), and the algorithms can also be tailored to find near-optimal solutions for the max-cut problem in dense \ER graphs. It is worth noting that all the efficient approximation algorithms mentioned above rely on a specific property of the underlying stochastic models called \emph{full replica symmetry breaking (FRSB)}, and in fact it is tempting to conjecture FRSB as an indication of low computational complexity \cite{GMZ22}.
	
	Over the past few decades, various frameworks have been proposed to provide insights into the computational hardness of optimization problems with random inputs (for surveys, see \cite{ZK16, BPW18, RSS19, Garmarnik21}). Notably, the overlap-gap property, initially introduced in \cite{GS14} and popularized in \cite{GMZ22, GKPX23+, GZ19+, HS21, RV17, Wein22}, serves as a geometric barrier that provides solid evidence for the failure of stable algorithms to find solutions beyond a certain threshold \cite{Garmarnik21}. In the last ten years, several generalizations of the initial overlap-gap property have been discovered, such as the multi-OGP \cite{RV17} and the ladder-OGP \cite{Wein22}, which establish computational hardness results close to the believed computational thresholds. Recently, a seminal work \cite{HS21} introduced yet another variant of the overlap-gap property, called the \emph{branching-OGP}, which was used to provide tight hardness results for stable algorithms for mean-field spin glasses. More recently, the branching-OGP framework has also been applied to determine the algorithmic threshold for random perceptron models \cite{HSS25}. The current work is also significantly inspired by the branching-OGP framework. 

	\subsection{Proof overview}

The proof of the main results can be divided into four parts: informational negative/positive results and computational positive/hardness results.

We obtain the informational negative result in Section~\ref{subsec-info-upper} for both the sparse and dense regimes, via a simple union bound (a.k.a. the first moment method). For the informational positive result in the dense regime, we employ a truncated second-moment method argument in Section~\ref{subsec-info-lower}. It is important to note that a straightforward application of the Paley-Zygmund inequality only yields a vanishing lower bound of probability. However, inspired by \cite{CG22}, we enhance the lower bound to $1-o(1)$ by combining with Talagrand's concentration inequality.

In the case where $p$ is in the sparse regime, surprisingly, we have not discovered any non-constructive proof of the informational positive result. Instead, in Section~\ref{sec-PTAS}, we construct specific algorithms to find near-optimal solutions for the random GAP. By analyzing these algorithms in Section~\ref{subsec-analysis-sparse}, we establish an algorithmic positive result in the sparse regime, which naturally implies an informational positive result and happens to match the informational negative result (note that the same story occurs in \cite{DDG22} too). This completes the proof of Theorem~\ref{thm-info} and thereby concludes Theorem~\ref{thm-PTAS}. Additionally, in Section~\ref{subsec-algo-analysis-dense}, we analyze the performance of the same algorithms in the dense regime and show that these algorithms can find near $\beta_c$-optimal solutions, leading to the first item in Theorem~\ref{thm-algo}.

In Section~\ref{sec-2-OGP}, we focus on the computational hardness results in the dense regime and derive Theorem~\ref{thm-OGP} as well as the second item in Theorem~\ref{thm-algo}. We establish the existence of algorithmic barriers by taking advantage of various types of the overlap-gap property. While the argument for the failure of stable algorithms via $2$-OGP in Section~\ref{subsec-two-OGP} is rather standard, our proof of the failure of online algorithms in Section~\ref{subsec-branching-OGP} is a novel combination of the branching-OGP framework presented by \cite{HS21} and a resampling technique inspired by \cite{GKPX23+}.
\\
\ \\
\noindent\textbf{Acknowledgements.} We would like to express our sincere gratitude to Jian Ding for his heartfelt guidance and support during the early stages of this project, as well as for providing us with many helpful suggestions for this manuscript. 
We are also grateful to Nike Sun for introducing us to the crucial literature reference \cite{HS21}. Furthermore, we thank Brice Huang, Zhangsong Li, and Mark Sellke for having engaging discussions with us on the algorithmic aspects of this problem. H. Du and S. Gong are partially supported by NSFC Key Program Project No.12231002.
\section{The informational thresholds}
This section is devoted to the proof of Theorem~\ref{thm-info}, where the negative and positive results are established separately in the next two subsections.
\label{sec-proof-info}
\subsection{Informational negative results}\label{subsec-info-upper}
In this short subsection we prove the upper bound for Theorem~\ref{thm-info}. First we recall the Chernoff bound for binomial variables.
\begin{proposition}[Chernoff bound]\label{prop-chernoff}
	For any $N\in \mathbb{N},P\in (0,1)$ and $\delta>0$, let $X\sim\mathbf{B}(N,P)$, it holds that
	\begin{equation}
        \label{eq-chernoff-bound-upper}
		\Pb[X\ge (1+\delta)NP]\le \exp\Big(-NP\big((1+\delta)\log(1+\delta)-\delta\big)\Big)\,,
	\end{equation}
 and 
 \begin{equation}\label{eq-chernoff-bound-lower}
     \Pb[X\le(1-\delta)NP]\le \exp\left(-\frac{\delta^2NP}{2}\right)\,.
 \end{equation}
In particular, we have for any $K\ge 0$,
\begin{equation}\label{eq-chernoff-bound}
	\Pb[|X-NP|\ge K]\le 2\exp\left(-\frac{K^2}{2(NP+K)}\right)\,.
\end{equation}
\end{proposition}
The proof of \eqref{eq-chernoff-bound-upper} and \eqref{eq-chernoff-bound-lower} can be found in, for example, \cite[Appendix C, Lemma 11]{WXY20}. Subsequently, \eqref{eq-chernoff-bound} follows from the easy-checked fact that
\[
(1+\delta)\log (1+\delta)-\delta\ge \frac{\delta^2}{2(1+\delta)}\,.
\]
The remaining proof is just a straightforward application of the union bound.

\begin{proof}[Proof of the upper bound for Theorem~\ref{thm-info}]
%We write $G=(V,E)$ and it is clear that $|E|$ follows a binomial distribution $\mathbf B(\binom{n}{2},p)$. From \eqref{eq-chernoff-bound} we can readily conclude that \begin{equation}\label{eq-E-is-not-large}
%\Pb\Big[|E|\le \binom{n}{2}p+\sqrt{n^2p\log n}\Big]=1-o(1)\,.
%\end{equation}
%We then fix a realization of $G$ with this property. 
For any fixed permutation $\pi \in \operatorname{S}_n$, it is evident that $\operatorname{O}(\pi) \sim \mathbf{B}(\binom{n}{2}, p^2)$. We will show that for any fixed $\varepsilon > 0$, the probability of such a binomial variable deviating from its expectation by more than $(1+\varepsilon)S_{n,p}$ or $(1+\varepsilon)D_{n,p}$ is much less than $1/n!$. This accomplishes our goal through a simple union bound.
 
When $p$ is in the sparse regime, we have $|E|p<n^2p^2\ll S_{n,p}$. Hence from \eqref{eq-chernoff-bound-upper},
	\begin{equation*}
	\begin{split}
	&\Pb\left[\mathbf{B}\Big(\binom{n}{2},p^2\Big)\ge (1+\varepsilon)S_{n,p}+E_{n,p}\right]\\
	{\le}&\exp \left(-\big(E_{n,p}+(1+\varepsilon)S_{n,p}\big)\log \left(1+\frac{(1+\varepsilon)S_{n,p}}{n^2p^2}\right)+(1+\varepsilon)S_{n,p}+E_{n,p}\right)\\
	\le&\exp \left(-\big(1+\varepsilon+o(1)\big)S_{n,p}\log \frac{S_{n,p}}{n^2p^2}\right)=\exp\left(-S_{n,p}\cdot(1+\varepsilon+o(1)\log \left(\frac{\log n}{np^2}\right)\right)\\
 =&\exp\Big(-\big(1+\varepsilon+o(1)\big)n\log n\Big)\,,
    \end{split}
    \end{equation*}
    which is much less than $1/n!$, as desired. When $p$ is in the dense regime, %we have $|E|p-E_{n,p}\ll D_{n,p}\ll E_{n,p}$. Therefore,
    we get from \eqref{eq-chernoff-bound} that
	\begin{equation*}
	\begin{aligned}
	&\ \Pb\left[\mathbf{B}\Big(\binom{n}{2},p^2\Big)\ge (1+\varepsilon)D_{n,p}+E_{n,p}\right]
\le 2
	\exp \left(-\frac{\big(1+\varepsilon+o(1)\big)^2D_{n,p}^2}{2E_{n,p}+2(1+\varepsilon)D_{n,p}}\right)\,,
	\end{aligned}
    \end{equation*}
    which equals to $\exp\Big(-\big[(1+\varepsilon)^2+o(1)\big]n\log n\Big)\ll 1/n!$, as desired.
\end{proof}

\subsection{Informational positive results}\label{subsec-info-lower}
 We now focus on the positive results stated in Theorem~\ref{thm-info}. The result of the sparse regime will be proved in Section~\ref{subsec-analysis-sparse}, and this subsection is dedicated to proving the informational positive result for $p$ in the dense regime. So throughout this subsection, we always assume that $p_c\ll p\ll 1$.
 %Interestingly, for the scenario where $p$ falls within the sparse regime, we have not found any \emph{existence proof} of the lower bounds. Instead, we give a \emph{constructive proof} in Section~\ref{subsec-analysis-sparse} by presenting specific (polynomial-time) algorithms that provide solutions achieving the corresponding lower bounds. In this way we establish the informational threshold in Theorem~\ref{thm-info} for the sparse regime, thereby concluding Theorem~\ref{thm-PTAS}. The remainder of this section is dedicated to proving the informational lower bound for $p$ in the dense regime. For this purpose, we will employ the truncated second moment method with a combination of concentration inequality, which is a highly non-constructive proof technique. Throughout this subsection, we assume that $p_c\ll p\ll 1$.

 A frequently employed strategy in this paper involves conditioning on a suitable realization of either graph $G$ or $\Gs$, while relying solely on the randomness of the other graph. In simpler terms, we will treat either $G$ or $\Gs$ as a deterministic graph possessing a desired property. To this end, we will introduce a concept named as \emph{admissibility}, which is related to the regularity of a graph across various aspects. This property will play a crucial role in establishing the informational positive result and will also contribute to the computational hardness analysis discussed in Section~\ref{sec-2-OGP}.
 
 Before introducing the definition of admissibility, we emphasize that this property is specifically designed for \ER graphs $\mathbf G(n,p)$ when $p$ falls into the dense regime. We begin with several notations and observations. Denote $\operatorname{U}$ as the set of unordered pairs $\{(i,j):i, j \in [n], i\neq j\}$.
Fix $G$ with edge set $E\subset \operatorname{U}$. For any $\pi\in \operatorname{S}_n$, consider the set  
\[
\operatorname{OL}(G,\pi)=\{(i,j)\in \operatorname{U}:G_{i,j}=G_{\pi(i),\pi(j)}=1\}\,.
\]
then for any $\pi_1,\pi_2\in \operatorname{S}_n$, let $\pi_{1,2}=\pi_1^{-1}\circ \pi_2$, we claim that it holds
\[
\operatorname{Cov}\big(\widetilde{\operatorname{O}}(\pi_1),\widetilde{\operatorname{O}}(\pi_2)\big)=p(1-p)\times |\operatorname{OL}(G,\pi_{1,2})|\,,
\]
where the covariance is taken with respect to $\Gs$.
This is because we can express the covariance $\operatorname{Cov}\big(\widetilde{\operatorname{O}}(\pi_1),\widetilde{\operatorname{O}}(\pi_2)\big)=\operatorname{Cov}\big(\operatorname{O}(\pi_1),\operatorname{O}(\pi_2)\big)$ as
\begin{align*}
\sum_{\substack{(i_1,j_1)\in E\\{(i_2,j_2)\in E}}}\operatorname{Cov}\left(\Gs_{\pi_{1}(i_1),\pi_{1}(j_1)},\Gs_{\pi_{2}(i_2),\pi_{2}(j_2)}\right)=\sum_{\substack{(i_1,j_1)\in E\\(i_2,j_2)\in E}}p(1-p)\mathbf{1}_{(\pi_1(i_1),\pi_1(j_1))=(\pi_2(i_2),\pi_2(j_2))}
\end{align*}
and the number of pairs $(i_1,j_1),(i_2,j_2)$ such that $(\pi^{-1}_{1}(i_1),\pi_{1}^{-1}(j_1))=(\pi_{2}^{-1}(i_2),\pi_{2}^{-1}(j_2))$ is exactly $|\operatorname{OL}(G,\pi_{1,2})|.$

For each permutation $\pi\in \operatorname{S}_n$, let ${F}(\pi)$ denote the number of fixed points and $T(\pi)$ denote the number of transpositions in $\pi$. It is worth noting that in the case of $G\sim \mathbf{G}(n,p)$, an unordered pair $(i,j)\in \operatorname{U}$ appears in $\operatorname{OL}(G,\pi)$ with a probability of either $p$ or $p^2$, depending on whether $(i,j)=(\pi(i),\pi(j))$ or not. In the former case, both $i$ and $j$ are fixed points, or $(i,j)$ represents a transposition. This leads to the following relation
\begin{equation}\label{eq-E-OL}
    \mathbb{E}|\operatorname{OL}(G,\pi)|=\left[\binom{F(\pi)}{2}+T(\pi)\right](p-p^2)+\binom{n}{2}p^2\,.
\end{equation}

\begin{definition}\label{def-admissible}
    For any $p$ in the dense regime, we say a graph $G$ on $n$ vertices is $p$-admissible (we simply call it admissible when the parameter $p$ is clear in the context), if the following items hold:
    \begin{itemize}
        \item The total number of edges in $G$ satisfies
		\begin{equation}
			\label{eq-fixed-graph-edge-concentration}
			\left||E(G)|-\binom{n}{2}p\right|\le 2\sqrt{n^2p\log n}\,.
		\end{equation}
  \item For any induced subgraph $H$ of $G$ with $k$ vertices, it holds that 
  \begin{equation}\label{eq-control-of-E(H)}
      \left||E(H)|-\binom{k}{2}p\right|\le \frac{n^2p}{\log n^{1/4}}\,.
  \end{equation}
	\item  for any $\pi\in \operatorname{S}_n$ it holds
	\begin{equation}
		\label{eq-OL-concentration}
\big||\operatorname{OL}(G,\pi)|-\mathbb{E}|\operatorname{OL}(G,\pi)|\big|\le 2\sqrt{|\operatorname{F(\pi)}|np\log n}+3\sqrt{2n^3p^2\log n}\,.
\end{equation}
In particular, $|\operatorname{OL}(G,\pi)|/\mathbb E|\operatorname{OL}(G,\pi)|=1+o(1)$ uniformly for all $\pi\in \operatorname{S}_n$.
    \end{itemize}
\end{definition}
\begin{lemma}\label{lem-good-event-on-G}
For any $p$ in the dense regime, an \ER graph $G\sim \mathbf G(n,p)$ is $p$-admissible with probability $1-o(1/n^2)$.%, where the probability is taken over the random graph $G\sim \mathbf{G}(n,p)$. 
\end{lemma}
The lemma can be essentially derived using a union bound, and we will provide its proof in the appendix. From this point forward in this section, we will consider a fixed value of $p$ within the dense regime, along with a $p$-admissible graph $G$. Therefore, throughout the remaining portion of this section, the notation $\Pb$ will refer to the probability distribution of a single graph $\Gs$ drawn from $\mathbf G(n,p)$.

The following proposition establishes a concentration result for the maximum of $\widetilde{\operatorname{O}}(\pi)$, which will serve as a key component in proving the lower bound stated in Theorem~\ref{thm-info}.

\begin{proposition}\label{prop-concentration-of-max-O-pi}
For any constant $\varepsilon>0$, there exists a constant $\gamma=\gamma(\varepsilon)>0$, such that for any admissible graph $G$,
	\[
\Pb\left[\big|\max_{\pi\in \operatorname{S}_n}\widetilde{\operatorname{O}}(\pi)-\mathbb{E}\max_{\pi\in \operatorname{S}_n}\widetilde{\operatorname{O}}(\pi)\big|\ge \varepsilon D_{n,p}\right]\le \exp\left(-\gamma n\log n\right)\,,
	\]
 where the probability and expectation is taken over $\Gs\sim \mathbf G(n,p)$.
\end{proposition}
The result stated in Proposition~\ref{prop-concentration-of-max-O-pi} can be obtained through a standard application of Talagrand's concentration inequality. The crucial observation is that the maximum $\max_{\pi\in \operatorname{S}_n}{\operatorname{O}}(\pi)$ is Lipschitz and $s$-certifiable, with its median asymptotically equalling to $n^2p^2/2$. In order to provide a comprehensive explanation, we will include a summary of the relevant terminology and the results of Talagrand's concentration inequality in the appendix. Additionally, we will present a complete proof of Proposition~\ref{prop-concentration-of-max-O-pi} therein.

For any $\varepsilon>0$, we define $X_\varepsilon$ as the number of permutations $\pi\in \operatorname{S}_n$ such that $\widetilde{\operatorname{O}}(\pi)\ge (1-\varepsilon)D_{n,p}$. Another main ingredient for the proof is the following estimate.
\begin{proposition}\label{prop-second-moment-est}
 For any admissible graph $G$ and any constant $\varepsilon\in (0,1)$, it holds that $$\mathbb{E}X_\varepsilon^2=\exp\big(o(n\log n)\big)(\mathbb{E}X_\varepsilon)^2\,,$$
 where the expectation is taken with respect to $\Gs\sim \mathbf G(n,p)$.
\end{proposition}
With Proposition~\ref{prop-second-moment-est} in hand, we can establish the lower bound of Theorem~\ref{thm-info} by combining the Paley-Zygmund inequality with the aforementioned concentration result.

\begin{cor}
	For any $\varepsilon>0$ we have as $n\to \infty$,
	\[
	\Pb\left[\max_{\pi\in \operatorname{S}_n}\widetilde{\operatorname{O}}(\pi)\ge (1-3\varepsilon)D_{n,p}\right]=1-o(1)\,,
	\]
  where the probability is taken over $(G,\Gs)\sim \mathbf G(n,p)^{\otimes 2}$.
\end{cor}
\begin{proof}
Since $\Pb[G\text{ is admissible}]=1-o(1)$ we may condition on the realization of $G$ and assume it is admissible. From Proposition~\ref{prop-second-moment-est} and the Paley-Zygmund inequality, we have for any $\varepsilon
>0$, it holds that%(here the probability is only taken over $\Gs\sim \mathbf G(n,p)$)
\[
\Pb\left[\max_{\pi\in \operatorname{S}_n}\widetilde{\operatorname{O}}(\pi)\ge (1-\varepsilon)D_{n,p}\right]=\Pb[X_{\varepsilon}>0]\ge \frac{(\mathbb{E}X_{\varepsilon})^2}{\mathbb{E}X_{\varepsilon}^2}=\exp\big(-o(n\log n)\big).
\]
Combining with Proposition~\ref{prop-concentration-of-max-O-pi}, this implies $\mathbb{E}\max_{\pi\in \operatorname{S}_n}\widetilde{\operatorname{O}}(\pi)\ge (1-2\varepsilon)D_{n,p}$ holds for $n$ large enough. Then by applying the result of Proposition~\ref{prop-concentration-of-max-O-pi} again we get as $n\to \infty$,
$$
\Pb\left[\max_{\pi\in \operatorname{S}_n}\widetilde{\operatorname{O}}(\pi)\ge (1-3\varepsilon)D_{n,p}\right]=1-o(1)\,.
$$
Since this is true for any $\varepsilon>0$, we reach the desired conclusion.
\end{proof}
Now we turn to the proof of Proposition~\ref{prop-second-moment-est}. We start by the following two-point estimation.
\begin{proposition}\label{prop-two-point-est}
	For any admissible graph $G$, any $\varepsilon\in (0,1)$ and any $\pi_1,\pi_2\in \operatorname{S}_n$, write $\pi_{1,2}=\pi_1^{-1}\circ \pi_2$, then it holds
	\begin{equation}\label{eq-two-point-fun}
		\begin{aligned}
&\Pb\big[\widetilde{\operatorname{O}}(\pi_1)\ge (1-\varepsilon)D_{n,p},\widetilde{\operatorname{O}}(\pi_2)\ge (1-\varepsilon)D_{n,p}\big]\\
		\le& \exp\left(-\frac{2(1-\varepsilon)^2n\log n}{1+(F(\pi_{1,2})/n)^2}+o(n\log n)\right)\,.
		\end{aligned}
	\end{equation}
\end{proposition}
\begin{proof}
    Recall that $\operatorname{U}$ is the set of unordered pairs and $\operatorname{OL}(G,\pi_{1,2})$ is the subset of $\operatorname{U}$ that 
	\[
    \{(i,j)\in U:G_{i,j}=G_{\pi_{1,2}(i),\pi_{1,2}(j)}=1\}=\{(i,j)\in \operatorname{U}:G_{\pi^{-1}_1(i),\pi^{-1}_1(j)}=G_{\pi^{-1}_2(i),\pi_2^{-1}(j)}=1\}\,.
    \]
    We write $L=|\operatorname{OL}(G,\pi_{1,2})|$ for simplicity, and we denote 
	\begin{equation*}
		\begin{aligned}
	    S_0=\sum_{(i,j)\in \operatorname{OL}(G,\pi_{1,2})}G_{\pi^{-1}_1(i),\pi^{-1}_1(j)}\Gs_{i,j},\\
	     S_1=\sum_{(i,j)\in \operatorname{U}\setminus\operatorname{OL}(G,\pi_{1,2})}G_{\pi^{-1}_1(i),\pi^{-1}_1(j)}\Gs_{i,j},\\
	     S_2=\sum_{(i,j)\in \operatorname{U}\setminus\operatorname{OL}(G,\pi_{1,2})}G_{\pi^{-1}_2(i),\pi^{-1}_2(j)}\Gs_{i,j}.\\
        \end{aligned}
    \end{equation*}
	It is clear that $S_0\sim \mathbf B(L,p)$ and $S_1,S_2\sim \mathbf B(E-L,p)$ are independent binomial variables, and 
	\begin{equation*}
		\begin{aligned}
	       \operatorname{O}(\pi_1)=\sum_{(i,j)\in \operatorname{U}}G_{\pi^{-1}_1(i),\pi^{-1}_1(j)}\Gs_{i,j}=S_0+S_1,\\
		   \operatorname{O}(\pi_2)=\sum_{(i,j)\in \operatorname{U}}G_{\pi^{-1}_2(i),\pi^{-1}_2(j)}\Gs_{i,j}=S_0+S_2\,.
        \end{aligned}
    \end{equation*}
Denote $M=\lceil(1-\varepsilon)D_{n,p}\rceil$, then the probability $\Pb\big[\widetilde{\operatorname{O}}(\pi_1)\ge M,\widetilde{\operatorname{O}}(\pi_2)\ge M]$ equals to
\begin{equation}\label{eq-three-terms}
	\begin{aligned}
&\ \Pb\big[S_0+S_1 \ge Ep+M,S_0+S_2\ge Ep+M \big]\\
		=&\ \sum_k\Pb[S_0=Lp+k]\cdot\big(\Pb[S_1\ge(E-L)p+M-k)]\big)^2\\
		\le&\ \Pb[S_0\ge Lp+M]+\big(\Pb[S_1\ge (E-L)p+M]\big)^2\\
  &\ +(M+1)\max_{0\le k\le M}\Pb[S_0\ge Lp+k]\cdot\big(\Pb[S_1\ge (E-L)p+M-k]\big)^2\,.
  \end{aligned}
\end{equation}
We consider the three terms in \eqref{eq-three-terms} separately and claim that each one of them is bounded by the right-hand side of \eqref{eq-two-point-fun}. By the third item in admissibility, we have $L/E = (F(\pi_{1,2})/n)^2 + o(1)$. Applying the Chernoff bound \eqref{eq-chernoff-bound}, we see that the first term of the expression is bounded by
\[
\exp\left(-\frac{(1-\varepsilon)^2D_{n,p}^2}{2(Lp+D_{n,p})}\right)\le \begin{cases}
    \exp\left(-2(1-\varepsilon)^2n\log n+o(n\log n)\right),\quad&\text{if }F(\pi_{1,2})/n\le 1/2\,,\\
    \exp\left(-\frac{(1-\varepsilon)^2 n\log n}{(F(\pi_{1,2})/n)^2}+o(n\log n)\right),\quad&\text{if }F(\pi_{1,2})/n>1/2\,,
\end{cases}
\]
which is always bounded by $\exp\left(-\frac{2(1-\varepsilon)^2n\log n}{1+(F(\pi_{1,2})/n)^2}+o(n\log n)\right)$, as desired. Similar arguments suggest that this is true for the second term in \eqref{eq-three-terms}.
Regarding the last term, by applying \eqref{eq-chernoff-bound} again, we see that it is bounded by $M+1=\exp\big(o(n\log n)\big)$ times
\begin{equation*}
\begin{aligned}
\max_{0\le k\le M}&\ \exp\left(-\frac{k^2}{2(Lp+k)}-\frac{(M-k)^2}{(E-L)p+M-k}\right)
		\le\max_{0\le k\le M}\exp\left(-\frac{M^2}{Ep+Lp+M+k}\right)\\
  		=&\ \exp\left(-\frac{2(1-\varepsilon)^2n^3p^2 \log n}{Ep+Lp+2M}\right)=\exp\left(-\frac{2(1-\varepsilon)^2 n\log n}{1+{L}/{E}+o(1)}\right)\\
    =&\ \exp\left(-\frac{2(1-\varepsilon)^2n\log n}{1+(F(\pi_{1,2})/n)^2}+o(n\log n)\right)\,,
	\end{aligned}
\end{equation*}
where in the first inequality we used Cauchy-Schwarz inequality. 
This completes the proof of the claim and thus \eqref{eq-two-point-fun} follows.
\end{proof}
\begin{proof}[Proof of Proposition~\ref{prop-second-moment-est}]
 By definition we have 
 \[
 \mathbb E X_{\varepsilon}^2=\sum_{\pi_1,\pi_2\in \operatorname{S}_n}\Pb\big[\widetilde{\operatorname{O}}(\pi_1)\ge (1-\varepsilon)D_{n,p},\widetilde{\operatorname{O}}(\pi_2)\ge (1-\varepsilon)D_{n,p}\big]\,.
 \]
 From Proposition~\ref{prop-two-point-est} we get this is bounded by 
 \begin{align*}
&\ \ \ \ \exp\big(o(n\log n)\big)\times\sum_{\pi_1,\pi_2\in \operatorname{S}_n}\exp\left(-\frac{2(1-\varepsilon)^2n\log n}{1+(F(\pi_{1,2})/n)^2}\right)\\
&=\exp\big([1+o(1)]n\log n\big)\sum_{\pi\in\operatorname{S}_n}\exp\left(-\frac{2(1-\varepsilon)^2n\log n}{1+(F(\pi)/n)^2)}\right)\\
&=\exp\big([1+o(1)]n\log n\big)\times\sum_{k=0}^n\exp\left(-\frac{2(1-\varepsilon)^2n\log n}{1+(k/n)^2}\right)\times |\{\pi\in \operatorname{S}_n:F(\pi)=k\}|\\
&\le\exp\big([2+o(1)]n\log n\big)\times\sum_{k=0}^{n}\exp\left(-\left[\frac{2(1-\varepsilon)^2}{1+(k/n)^2}+k/n\right]n\log n\right)\\
&\le\exp\left([2-2(1-\varepsilon)^2+o(1)]n\log n\right)\,,
 \end{align*}
 where in the first inequality we used a well-known fact that $|\{\pi\in \operatorname{S}_n:F(\pi)=k\}|$ is bounded by $\exp\big((n-k)\log n+o(n\log n)\big)$ and the second inequality follows from the observation that for any $\gamma\in [0,1]$, it holds
 $$\gamma[1-2(1-\varepsilon)^2\gamma+\gamma^2]\ge 0\iff \frac{2(1-\varepsilon)^2}{1+\gamma^2}+\gamma\ge 2(1-\varepsilon)^2\,.$$ Finally, we conclude the proof of Proposition~\ref{prop-second-moment-est} by noting that from Lemma~\ref{lem:Upper bound for binomial tail},
\[\mathbb{E}X_\varepsilon=n!\times\Pb[\widetilde{\operatorname{O}}(\pi)\ge (1-\varepsilon)D_{n,p}]\ge\exp\big([1-(1-\varepsilon)^2+o(1)]n\log n\big)\,.\qedhere\]
\end{proof}

\section{Algorithmic positive results via greedy algorithms}\label{sec-PTAS}\label{sec-proof-PTAS}
	
We now focus on the algorithmic aspect of the random GAP. This section develops matching algorithms satisfying the properties in Theorem~\ref{thm-PTAS} and Theorem~\ref{thm-algo}-(i). While \cite{DDG22} provides a polynomial-time approximation scheme for random GAP when \( p = n^{-\alpha+o(1)}, \alpha \in (1/2, 1] \) with $p\ge \log n/n$ (see Remark~\ref{rmk-connection-to-DDG24}), we address the remaining cases: \( p = n^{-1/2+o(1)} \) with \( p \ll p_c \), and \( p_c \ll p \leq 1/(\log n)^4 \).

%\textcolor{blue}{(I moved the blue-colored text to Remark 2.)It is important to note that when $p=n^{-\alpha+o(1)}$, where $1/2<\alpha<1$ is a constant, the value of $S_{n,p}$ is asymptotically  $n/(2\alpha-1)$, and a polynomial-time approximation scheme has already been presented in \cite{DDG22}. Therefore, our focus will be on the remaining cases, specifically when $p=n^{-1/2+o(1)}$ with $p\ll p_c$, and $p_c\ll p\leq 1/(\log n)^4$.}
As we will see, the algorithms for the sparse and dense regimes share the same framework, and their analyses are similar in structure but differ in detail. Section~\ref{subsec-algo-framework} outlines the general algorithmic framework and analysis structure. Sections~\ref{subsec-analysis-sparse} and \ref{subsec-algo-analysis-dense} focus on \( p = n^{-1/2+o(1)} \) with \( p \ll p_c \) and \( p_c \ll p \leq 1/(\log n)^4 \), respectively, completing the proofs of Theorem~\ref{thm-PTAS} and Theorem~\ref{thm-algo}-(i).

	\subsection{Framework of the algorithm}\label{subsec-algo-framework}
Fix a constant $\eta\in (0,1/2)$. We present the following greedy algorithm $\mathcal A_\eta\in \operatorname{GAA}$.

    \begin{algorithm}
		\label{Greedy Matching Algo}
		\caption{Greedy Matching Algorithm $\mathcal A_\eta$}
		\begin{algorithmic}[1]
			\STATE \textbf{Input}: $G,\Gs\in \operatorname{G}_n$ with adjacency matrices $\{G_{i,j}\},\{\mathsf G_{i,j}\}$.
			\STATE \textbf{Initialize}: $\pi^*(i)=i$ for $1\leq i\leq \eta n$, $\Rs_{\lfloor \eta n\rfloor}=\{\lfloor \eta n\rfloor +1,\dots, n\}$.
			\FOR{$\lfloor \eta n\rfloor+1\le s\le \lfloor(1-\eta)n\rfloor$}
   \STATE Uniformly sample $\pi^*(s)$ from $\arg\max_{r\in \mathsf R_{s-1}}\sum_{j<s}G_{j,s}\Gs_{\pi^*(j),r}$. %\COMMENT{Note that $\pi^*(s)$ is almost surely unique; this is the reason why we introduce the perturbations.}
   \STATE Set $\Rs_s=\Rs_{s-1}\setminus \{\pi^*(s)\}$.
			\ENDFOR
      \STATE Complete $\pi^*$ to a permutation such that $\pi^*(\lfloor(1-\eta)n\rfloor+1)<\cdots<\pi^*(n)$.
			\STATE \textbf{Output}: $\pi^*$. 
		\end{algorithmic}
\end{algorithm}

It is clear that $\mathcal A_\eta$ is an online algorithm which runs in $O(n^3)$ times for any $\eta>0$. The primary goal of this section is to show that for sufficiently small constant $\eta>0$, the algorithm $\mathcal A_\eta$ indeed outputs near-optimal solutions in the case $p=n^{-1/2+o(1)}$ with $p\ll p_c$, and near $\beta_c$-optimal solutions in the case $p_c\ll p\le 1/(\log n)^4$. Precisely, we prove the following two propositions regarding the sparse and dense regimes.

\begin{proposition}\label{prop-sparse-PTAS}
    For $p=n^{-1/2+o(1)}$, $p\ll p_c$ and any sufficiently small constant $\eta>0$, it holds that as $n\to \infty$,
    \[
    \Pb[\mathcal A_\eta(G,\Gs)\in \operatorname{S}_{1-4\eta}(G,\Gs)]=1-o(1)\,,
    \]
    where $\Pb$ is taken over $(G,\Gs)\sim \mathbf G(n,p)$ together with the internal randomness of $\mathcal A_\eta$.
\end{proposition}
\begin{proposition}\label{prop-A-eta-is-ideal-online}
    For $p_c\ll p\le 1/(\log n)^4$ and any sufficiently small constant $\eta>0$, it holds that as $n\to \infty$,
    \[
    \Pb[\mathcal A_\eta(G,\Gs)\in \operatorname{S}_{\beta_c-20\eta}(G,\Gs)]=1-o(1)\,,
    \]
    where $\Pb$ is taken over $(G,\Gs)\sim \mathbf G(n,p)$ and the internal randomness of $\mathcal A_\eta$.
\end{proposition}

 The boundary \( p_c = \sqrt{{\log n}/{n}} \) separating the sparse and dense regimes, as well as the appearance of \( \beta_c \) in the dense regime, may not be immediately apparent to readers.
Therefore, we present a heuristic derivation of Propositions~\ref{prop-sparse-PTAS} and \ref{prop-A-eta-is-ideal-online} to illustrate the key intuitions behind these results. To simplify the discussion, we focus on the greedy algorithm $\mathcal A_0$ and ignore the correlations across different iterations. 
 
For $1\le s\le n$, suppose that at the $s$-th step, $\pi^*(i), 1\leq i \leq s-1$ have been determined. We then choose $\pi^*(s)\in \mathsf{R}_{s-1}=[n]\setminus\{\pi^*(1),\dots,\pi^*(s-1)\}$ to maximize 
 \[
O_s(\ell)=\sum_{j<s}G_{j,s}\mathsf{G}_{\pi^*(j),\ell}=\sum_{j<s:G_{j,s}=1}\mathsf{G}_{\pi^*(j),\ell}\,.
 \]
 Conditioning on the realization of $G$ and denoting $N_s=\#\{j<s:G_{j,s}=1\}$, we note that $O_s(\ell)\sim \mathbf{B}(N_s,p)$ for each $\ell\in \mathsf{R}_{s-1}$, and these random variables are conditionally independent of each other. Consequently, we can view $O_s=\max_{\ell\in \mathsf R_{s-1}}O_s(\ell)$ as the maximum of $n-s+1$ independent Binomial variables that distribute as $\mathbf{B}(N_s,p)$.

Meanwhile, we have $N_s\sim \mathbf{B}(s-1,p)$, and thus typically we expect $N_s\approx sp$. The key difference between the sparse and dense regimes lies in the Poisson and Gaussian natures of $\mathbf{B}(sp,p)$:
 \begin{itemize}
     \item When $p\ll p_c$, using Poisson approximation we expect that $\mathbf{B}(N_s,p)\approx \mathbf B(sp,p)\approx \operatorname{Poi}(sp^2).$ Hence,  $O_k$ is close to the maximum of $n-s+1$ independent $\operatorname{Poi}(sp^2)$ variables, which equals with high probability $$(1+o(1))\cdot\frac{\log(n-s+1)}{\log(\log (n-s+1)/sp^2)}\,.$$
     \item When $p_c\ll p\ll 1$, using Gaussian approximation we expect $\mathbf B(N_s,p)\approx \mathbf{B}(sp,p)\approx \mathcal N(sp^2,sp^2(1-p)),$ and thus $O_k$ is close to the maximum of $n-s+1$ independent $\mathcal N(sp^2,sp^2(1-p))$ variables, which equals with high probability $$sp^2+(1+o(1))\cdot\sqrt{2sp^2\log(n-s+1)}\,.$$
 \end{itemize}

 Using the above heuristics, we conclude that with high probability, the output $\pi^*$ of $\mathcal A_0$ satisfies that, in the sparse regime $p=n^{-1/2+o(1)}, p\ll p_c$,
 \[
 \operatorname{O}(\pi^*)=O_1+\cdots+O_n=(1+o(1))\sum_{s=1}^n\frac{\log (n-s+1)}{\log(\log (n-s+1)/sp^2)}=(1+o(1))\frac{n
\log n}{\log(\log n/np^2)}\,,
 \]
 while in the dense regime $p_c\ll p\ll 1$, 
  \begin{equation*}
 \begin{split}
 \operatorname{O}(\pi^*) =&\ O_1+\cdots+O_n=
    \sum_{s=1}^{n}sp^2+(1+o(1))\sqrt{2sp^2\log (n-s+1)}\\
    =&\ \binom{n}{2}p^2+(1+o(1))\sqrt{p^2\log n} \sum_{s=1}^n \sqrt{2s}
    = \binom{n}{2}p^2 +(1+o(1))\beta_c\sqrt{n^3p^2\log n}\,,
 \end{split}
 \end{equation*}
 where the last equality follows from the fact that $$n^{-3/2}\sum_{s=1}^n \sqrt{2s}=(1+o(1))\int_0^1\sqrt{2x}\operatorname{d}\!x=(1+o(1))\beta_c\,.$$
 This completes our heuristic derivation.

Formalizing the above heuristics requires a significant amount of work, and the most challenging part is handling correlations within the iterations (which we have ignored when discussing the heuristics). The remainder of this subsection introduces the framework we will use to manage these correlations.

To begin, we introduce an equivalent form of $\mathcal A_\eta$ that decouples the correlation posed by the uniform sampling in the algorithm. We sample i.i.d. perturbations $X_{i,j}=X_{j,i}\sim \operatorname{U}(0,1/n^2)$ (the uniform distribution on $(0,1/n^2)$), and define $$\Gs^*_{i,j}=\Gs_{i,j}+X_{i,j},\forall 1\le i<j\le n\,.$$ We make the following observation: since the total sum of $X_{i,j}$'s is no more than $1$, by symmetry, uniformly sampling $\pi^*(s)$ from $\arg\max_{r\in \Rs_{s-1}}\sum_{j<s}G_{j,s}\Gs_{\pi^*(j),r}$ is equivalent to setting $\pi^*(s)=\arg\max_{r\in \Rs_{s-1}}\sum_{j<s}G_{j,s}\Gs^*_{\pi^*(j),r}$, which is almost surely uniquely determined. Henceforth we will work on this version of $\mathcal A_\eta$.

We proceed with the analysis of the algorithm and introduce some notations. First, let us keep in mind that throughout the analysis in this section, we always fix some realization of $G$ and $\Pb$ denotes the probability measure on the sole graph $\Gs\sim \mathbf G(n,p)$ as well as the random perturbations $X_{i,j},1\le i<j\le n$. We will work with appropriate realizations of $G$ which possess some nice properties incorporated in the event $\Gc_0$ (which has different meanings for the two regimes, see Definition~\ref{def-good-events-sparse} and Definition~\ref{def-good-event-dense}, respectively). Given $G\in \mathcal G_0$, let $\eta>0$ be a small constant, we write $a_\eta=\lfloor \eta n\rfloor$ and $b_\eta=\lfloor (1-\eta)n\rfloor$ for simplicity. For $a_\eta<s\le b_\eta+1$, we use $\mathcal{F}_{s-1}$ to denote the $\sigma$-field generated by the entire graph $G$, the matching indices $\pi^*(i)$ for $i<s$, the matched parts of $\Gs$ before the $s$-th step ($\mathsf G_{\pi^*(i),\pi^*(j)}$ for $i<j<s$), and the corresponding perturbations $X_{i,j}$ for $i<j<s$. Moreover, we define $\operatorname N_s=\{j<s:G_{j,s}=1\}$ for $a_\eta<s\le b_\eta$. %We use $\mathcal N_s$ to denote the $\sigma$-field generated by $\{\operatorname N_k:[\eta n]+1\leq k\leq s\}$. Note that the only information containing in $\mathcal N_s$ but not in $\mathcal M_{s-1}$ is the set $\operatorname{N}_s$.
	
Recall that $\Gs^*_{i,j}=\Gs_{i,j}+X_{i,j}$. For each $1\le s\le n$ we let 
 $$
 O_s=\sum_{j<s}G_{j,s}\Gs^*_{\pi^*(j),\pi^*(s)}=\sum_{j\in \operatorname{N}_s}\Gs^*_{\pi^*(j),\pi^*(s)}.
 $$
 Additionally, for $a_\eta<s\le b_\eta$ and $r\in \Rs_{s-1}$, let $$E_{r,s}=\sum_{j<s}G_{j,s}\mathsf G^*_{\pi^*(j),r}=\sum_{j\in\operatorname{N}_{s}}\mathsf G^*_{\pi^*(j),r}.$$ 
	Furthermore, conditioned on any realization of $\mathcal F_{s-1}$, we denote $\mathcal N_s$ as the $\sigma$-field generated by the random variables $\Gs_{i,j}$'s and $X_{i,j}$'s, where at least one of $i$ or $j$ belongs to $\Rs_{s-1}$.% and let $\mathcal N_s$ be the $\sigma$-field generated by variables in $\Es_s$. %We remark that despite the similarity between the notations $\mathcal N_s$ and $\widetilde{\mathcal N}_s$, they represent totally different information in the two graphs $G$ and $\Gs$.

 Now we fix an index $s$ with $a_\eta<s\le b_\eta+1$ and investigate the impact of the $s$-th step posed by conditioning on $\mathcal F_{s-1}$. For any $r\in \Rs_{s-1}$, it must hold that $E_{r,k}<O_k$ for any $a_\eta<k<s$, since $\pi^*(k)$ is chosen to maximize $E_{\pi^*(k),k}=O_k$. Define the events
 \[
 \mathcal D_{r,s}=\bigcap_{a_\eta<k<s}\{E_{r,k}<O_k\}\,,\forall r\in \Rs_{s-1}\,,\text{ and }\mathcal D_s=\bigcap_{r\in \Rs_{s-1}}\mathcal D_{r,s}\,.
 \]
The crucial observation is that given $\mathcal F_{s-1}$ (where the sets $\operatorname{N}_k,a_\eta<k<s$ and $\Rs_{s-1}$ are deterministic), for any event measurable with respect to $\mathcal N_s$, conditioning on $\mathcal F_{s-1}$ is equivalent to conditioning on $\mathcal D_s$. Moreover, each event $\mathcal D_{r,s}$ is measurable with respect to the variables ${\Gs^*_{r,\pi^*(j)}, j<s}$, and these events are conditionally independent for different indices $r\in \Rs_{s-1}$. Given $\mathcal F_{s-1}$, we define $F_s(\cdot)$ as the distribution function of $E_{r,s}$ conditioned on $\mathcal F_{s-1}$, i.e.,
\begin{equation}\label{eq-distribution-F-s}
    F_s(x)=\Pb[E_{r,s}\le x\mid \mathcal F_{s-1}]=\Pb[E_{r,s}\le x\mid \mathcal D_{s}]=\Pb[E_{r,s}\le x\mid \mathcal D_{r,s}]\,.
\end{equation}
The last equality holds because the random variable $E_{r,s}$ and the event $\mathcal D_{r,s}$ are both independent of the events $\mathcal D_{r',s}$ for $r' \in \operatorname{R}_{s-1}\setminus\{r\}$. Here and henceforth in this section whenever $\mathcal F_{s-1}$ is given, with a slight abuse of notation we use $\Pb$ to denote the product measure on $\mathbf B(1,p)$ variables $\Gs_{i,j}$'s and $\operatorname{U}[0,1/n^2]$ variables $X_{i,j}$'s with either of $i$ or $j$ in $\Rs_{s-1}$. We emphasize that under such conventions, the sets like $\operatorname{N}_k,a_\eta<k\le b_\eta$ and $\Rs_{s-1}$ are considered to be deterministic and known. 

It is evident that the distribution function $F_s$ does not depend on the specific choice of $r\in \Rs_{s-1}$. In the $s$-th step, the quantity $O_s$ is sampled as the maximum of $|\Rs_{s-1}|=n-s+1$ i.i.d random variables with the distribution function $F_s$.

Let $F_{s}^*(\cdot)$ denote the distribution function of $E_{r,s}$ without conditioning on $\mathcal D_{r,s}$, i.e. the distribution function of the sum of a binomial variable $\mathbf B(|\operatorname{N}_s|,p)$ and $|\operatorname{N}_s|$ independent $\operatorname{U}[0,1/n^2]$ variables. As hinted earlier in the heuristic derivation, our goal is to show that under typical realizations of $\mathcal F_{s-1}$, $F_s$ is approximately equal to $F_s^*$. In other words, conditioning on the event $\mathcal D_{r,s}$ does not significantly alter the distribution of $E_{r,s}$, indicating that $O_s$ behaves similarly to the maximum of $n-s+1$ i.i.d binomial variables $\mathbf{B}(|\operatorname{N}_s|,p)$.
 %which is near to $\log n/\log\big(\log n/np^2\big)$ in the sparse regime.
%\begin{proposition}\label{prop-good-each-step}
%    Denote $M_\eta=(1-\eta)\log n/\log \big(\log n/np^2\big)$, then for sufficiently small $\eta>0$, it holds that $
 %   \Pb[O_s\ge M_\eta,\forall s\in \operatorname{S}(\eta)]=1-o(1)$.
%\end{proposition}
%\begin{proof}[Proof of Proposition~\ref{prop-sparse-PTAS} assuming Proposition~\ref{prop-good-each-step}]
    
%\end{proof}

\subsection{Analysis of greedy algorithm: the sparse regime}\label{subsec-analysis-sparse}

Assume $p=n^{-1/2+o(1)}$ and $p\ll p_c$, we now turn to prove Proposition~\ref{prop-sparse-PTAS}, thereby completing the proof of Theorem~\ref{thm-PTAS}. 

We denote $$M_\eta=\frac{(1-\eta)\log n}{\log \big(\log n/np^2\big)}\,.$$ 
The assumption of $p$ implies that $1\ll M_\eta\ll \log n$. We start by defining several good events which we will work with. Recall that $$\operatorname{N}_s=\{j<s:G_{j,s}=1\},\forall a_\eta<s\le b_\eta\,.$$

%Our idea is to define certain good events $\mathcal G_0$ and $\mathcal G_{s-1},s\in \operatorname{S}(\eta)$, where $\mathcal G_0$ is measurable with respect to $G$ and $\mathcal G_{s-1}$ is measurable with respect to $\mathcal F_{s-1}$ for any $s\in \operatorname{S}(\eta)$. We require that on the one hand, all these events happens with high probability; on the other hand, for each $s\in \operatorname{S}(\eta)$, whenever $\mathcal F_{s-1}$ satisfies $\Gc_0\cap \Gc_{s-1}$, it holds that 
%\begin{equation}\label{eq-tail-probability>>1/n}
%    \Pb[E_{r,s}\ge M_\eta\mid \mathcal F_{s-1}]=1-F_s(M_\eta)\gg 1/n\,.
%\end{equation}
%Once this can be done, we get that 
%which is $o(1)$ by \eqref{eq-tail-probability>>1/n}, as desired.

%Now we present the definitions of good events $\mathcal G_0,\mathcal G_{s-1},s\in \operatorname{S}(\eta)$ as follow. 

	\begin{definition}[Good events in the sparse regime]
		\label{def-good-events-sparse}
		Denote $\mathcal G_0$ for the event that
		\begin{itemize}
		    \item 
			$(1-\eta)sp\leq |\operatorname N_s|\leq (1+\eta)sp,\forall a_\eta<s\le b_\eta$;%
   %\notag
			\item$|\operatorname N_r\cap\operatorname N_s|\leq (\log n)^3\,,\forall a_\eta<r<s\le b_\eta$.
   \end{itemize}
In addition, for each $a_\eta<s\le b_\eta$, we denote $\Gc_{s}$ as the event that \begin{equation}
    O_k\ge M_\eta,\forall a_\eta<k\le s\,.
\end{equation}
\end{definition}
Note that $\Gc_0$ is measurable with respect to $G$, and for any $a_\eta<s\le b_\eta$, $\Gc_{s}$ is measurable with respect to $\mathcal F_{s}$. We first show that $\Gc_0$ is a typical event. %Throughout the remaining part of this section, we assume that the information regarding $\{\operatorname{N}_k\vert\, a_\eta\leq k\leq s\}$ is fixed under the conditional probability $\Pb$. For simplicity, we denote this conditional law as $\Pb$ as well.
\begin{lemma}
\label{lem-G_0-sparse}
    For $G\sim G(n,p)$, $\Pb[G\in \Gc_0]=1-o(1)$.
\end{lemma}
 \begin{proof}
    Note that for each $a_\eta<s\le b_\eta$, $|\operatorname{N}_s|$ is distributed as a binomial variable $\mathbf B(s,p)$, and for any $a_\eta<r<s\le b_\eta$, the distribution of $|\operatorname{N}_r\cap \operatorname{N}_s|$ is given by $\mathbf B(r,p^2)$. Using the Chernoff bound \eqref{eq-chernoff-bound}, we can show that for any $a_\eta<s\le b_\eta$,
    \[
\Pb[||\operatorname{N}_s|-sp|\ge \eta sp]\le 2\exp\left(-\frac{(\eta sp)^2}{2(1+\eta)\eta sp}\right)=o(1/n),
    \]
    and (recall that $np^2\ll \log n$ since we are in the sparse regime)
    \[
\Pb[\mathbf B(r,p^2)\ge (\log n)^3]\le \exp\left(-\frac{(\log n)^6}{2\big((1+\eta)rp^2+(\log n)^3\big)}\right)=o\left(1/n^2\right),
    \]
    Thus the lemma follows from a simple union bound.
 \end{proof}
We claim that now it remains to show the following proposition.
\begin{proposition}\label{prop-tail-prob>logn/n}
    For each $a_\eta<s\le b_\eta$ and any realization of $\mathcal F_{s-1}$ that satisfies $\Gc_0\cap \Gc_{s-1}$, it holds that for any $r\in \Rs_{s-1}$
    \begin{equation}\label{eq-tail-prob>logn/n}
    \Pb[E_{r,s}\ge M_\eta\mid \mathcal F_{s-1},\Gc_0\cap\Gc_{s-1}]=1-F_s(M_\eta)\gg \log n/{n}\,.
    \end{equation}
\end{proposition}
\begin{proof}[Proof of Proposition~\ref{prop-sparse-PTAS}]
    Write $\Gc_{a_\eta}$ for the trivial event. Note that for each $a_\eta<s\le b_\eta$, we have $O_s=\max_{r\in \Rs_{s-1}}E_{r,s}$.  Thus
    \begin{equation*}
        \begin{aligned}
            \Pb[\Gc_{s-1}\cap \Gc_0]-\Pb[\Gc_s\cap \Gc_0]=&\ \Pb[O_s\le M_\eta,\Gc_0,\Gc_{s-1}]\le \Pb[O_s\le M_\eta\mid \Gc_0\cap \Gc_{s-1}]\\
            =&\ \mathbb{E}\big[E_{r,s}\le M_\eta,\forall r\in \Rs_{s-1}\mid \mathcal F_{s-1},\Gc_0\cap \Gc_{s-1}\big]\\
            =&\ \mathbb{E}\big[F_{s}(M_\eta)^{n-s+1}\mid \mathcal F_{s-1},\Gc_0\cap \Gc_{s-1}\big]\\
            =&\ o(1/n)\,,
        \end{aligned}
    \end{equation*}
where the second-to-last equality is due to conditional independence, and the last relation follows from \eqref{eq-tail-prob>logn/n} and the fact that $n-s+1\ge \eta n$ for any $s\le b_\eta$. This implies that $$\Pb[\Gs_{b_\eta}\cap \Gc_0]=\Pb[\Gc_0]+\sum_{a_\eta<s\le b_\eta}\big(\Pb[\Gc_{s}\cap \Gc_0]-\Pb[\Gc_{s-1}\cap \Gc_0]\big)\ge 1-o(1)\,,$$
and in particular, $\Pb[\mathcal G_{b_\eta}]=1-o(1)$. Finally we note that under this event, we have
\begin{equation*}
    \begin{aligned}
        \operatorname{O}(\pi^*)=&\ \sum_{s=1}^n O_s-\sum_{i<j}G_{i,j}X_{\pi^*(i),\pi^*(j)}
        \ge \sum_{s=a_\eta+1}^{ b_\eta}O_s-1\ge (1-2\eta)nM_\eta-1\,,%\ge (1-4\eta)S_{n,p}\,,
    \end{aligned}
\end{equation*}
which implies $\widetilde{\operatorname{O}}(\pi^*)\ge (1-4\eta)S_{n,p}$, and thus $\pi^*\in \operatorname{S}_{1-4\eta}(G,\Gs)$, completing the proof.
\end{proof}
The remaining of this section is devoted to the proof of Proposition~\ref{prop-tail-prob>logn/n}. 
We fix a realization of $G$ that satisfies $\Gc_0$ and then a realization of $\mathcal F_{s-1}$ that satisfy $\mathcal{G}_{s-1}$. %Thus the information regarding $\{\operatorname{N}_k, a_\eta<k<s\}$ and $\Rs_{s-1}$ is revealed, and with slight abuse of notation, we still use $\Pb$ to denote the product measure of $\Gs_{i,j}\sim \mathbf B(1,p),X_{i,j}\sim\operatorname{U}[0,1/n^2]$ with at least one of $i,j$ lies in $\Rs_{s-1}$ given the information in $\mathcal F_{s-1}$. 
Note that we may apply the following relaxation: $$E_{r,s}=\sum_{j\in \operatorname{N}_s}\Gs^*_{r,\pi^*(j)}\ge \sum_{j\in \operatorname{N}_s}\Gs_{r,\pi^*(j)}\stackrel{\text{def}}{=}E'_{r,s}\,,$$
and $E'_{r,s}$ is measurable with respect to $\Gs_{r,\pi^*(j)},j\in \operatorname{N}_s$. According to previous discussions, it holds that
	\begin{equation}
		\begin{split}
			 \Pb[E_{r,s}\geq M_\eta\,|\,\mathcal D_{r,s}]
			\geq&\ \Pb[M_\eta\le E'_{r,s}\le 6 M_\eta\mid \mathcal D_{r,s}]=\frac{\mathbb P[\{M_\eta\leq E'_{r,s}\leq 6M_\eta\}\cap \mathcal D_{r,s}]}{\mathbb P[\mathcal D_{r,s}]}\\
	=&\ \mathbb P[M_\eta\leq E'_{r,s}\leq 6M_\eta]\times\frac{\mathbb P[\mathcal D_{r,s}\,|\,M_\eta\leq E'_{r,s}\leq 6M_\eta]}{\mathbb P[\mathcal D_{r,s}]}\,.
			%=&\ (1-o(1))\widehat{\mathbb P}_s[E_{i,s}\geq F_{\varepsilon,s}]\times\frac{\widehat{\mathbb P}_s[E_{i,r}< F_{\varepsilon,r}\mbox{ for }r\in\mathcal A_{i,s}\,|\,F_{\varepsilon,s}\leq E_{i,s}\leq \widetilde F_n]}{\widehat{\mathbb P}_s[E_{i,r}< F_{\varepsilon,r}\mbox{ for }r\in\mathcal A_{i,s}]},
		\end{split}\notag
	\end{equation}
It is clear that $E'_{r,s}\sim \mathbf B(|\operatorname{N}_s|,p)$, so we can conclude from Lemma~\ref{lem:poisson_tail_approx} that $$\Pb[M_\eta\le E'_{r,s}\le 6M_\eta]=\Pb[E'_{r,s}\geq M_\eta]-\Pb[E'_{r,s}\ge 6M_\eta]\gg \log n/n\,,$$ hence it suffices to show that
    \begin{equation}
      \label{eq-cond.on.N}
\Pb[\mathcal D_{r,s}\mid M_\eta\le E'_{r,s}\le 6M_\eta]\geq [1-o(1)]\Pb[\mathcal D_{r,s}].
    \end{equation}

To obtain \eqref{eq-cond.on.N}, for each $r\in\Rs_{s-1}$ we define for $\ell \in \mathbb N$ the random set of indices
$$\operatorname{R}^\ell_{s}=\Big\{a_\eta <k< s:\sum_{j\in \operatorname{N}_k\cap \operatorname{N}_s}\Gs_{r,\pi^*(j)}=\ell\Big\}\,,$$
 namely, the set of step indices $k$ in which $r$ connect exactly $\ell$ edges in $\Gs$ to the set $\pi^*(\operatorname{N}_k)\cap \pi^*(\operatorname{N}_s)$.
 Let $\mathcal{R}_s$ denote the collection $(\operatorname{R}^\ell_{s},\ell\in \mathbb N)$. We refer to a collection $\mathcal{R}_s$ as typical if it satisfies the conditions $|\operatorname{R}^1_{s}|+|\operatorname{R}^2_s|<n^{0.6}$ and $\operatorname{R}^\ell_{s}=\emptyset$ for all $\ell\geq 3$. The following lemma justifies the use of the term ``typical" for such collections.
\begin{lemma}
\label{lem-typical-sets-are-typical}
Whenever $\mathcal F_{s-1}$ satisfies $\mathcal G_{s-1}\cap \mathcal G_0$, we have
    	\begin{equation}
	\label{eq:good.R}
		\mathbb P[\mathcal R_s\text{ is typical}\,|\,M_\eta\leq E'_{r,s}\leq 6 M_\eta]\geq 1-o(1).
	\end{equation}
\end{lemma} 
\begin{proof}
%According to our conditions in ``typical" $\mathcal R_s$, we divide our proof into three parts. 

We fix an arbitrary integer $K$ in the range $[M_\eta,6M_\eta]$. We observe that %$K\le E_{r,s}<K+1$ implies $\sum_{j\in \operatorname{N}_s}\Gs_{r,\pi^*(j)}=K$ since the perturbations satisfy $0\le X_{i,j}\leq 1/n^2$. Therefore, 
conditioned on the event $\{E'_{r,s}=K\}$, the set of indices $i\in \operatorname{N}_s$ such that $\Gs_{r,\pi^*(i)}=1$ forms a uniform subset of $\operatorname{N}_s$ with $K$ elements. Consequently, for each $K$, the quantity $
\sum_{j\in \operatorname{N}_k\cap \operatorname{N}_s}\Gs_{r,\pi^*(j)}
$  follows a hypergeometric distribution $\mathbf{HG}(|\operatorname{N}_s|,|\operatorname{N}_k\cap\operatorname{N}_s|,K)$, which represents the size of the intersection of $\operatorname{N}_k\cap\operatorname{N}_s$ with a uniform subset of $\operatorname{N}_s$ of cardinality $K$. Recall that under the good event $\Gc_0$, we have $(1-\eta)sp\le |\operatorname{N}_s|\le (1+\eta)sp$ and $|\operatorname{N}_k\cap \operatorname{N}_s|\le (\log n)^3$ for each $\lfloor\eta n\rfloor\le k\le s-1$. 

For any triple $(N,M,K)$ with $(1-\eta)sp\le N\le (1+\eta)sp,\ 0\le M\le (\log n)^3$ and $M_\eta\le K\le 6M_\eta\ll \log n$, we have
 \begin{equation*}
	         \begin{aligned}	\Pb\big[\mathbf{HG}(N,M,K)\ge1\big]&=\sum_{k\ge1}\frac{\binom{M}{k}\binom{N-M}{K-k}}{\binom{N}{K}}\leq\sum_{k\ge 1} \binom{M}{k}\cdot\frac{N^{K-k}}{(K-k)!}\cdot\frac{K!}{(N-K)^K}\\
            &\le\sum_{k\ge 1}M^k\left(1-\frac K N\right)^{-K+k}\cdot\frac{K^k}{k!(N-K)^k}\\
            &\le \exp\left(\frac{K^2}{N}\right)\sum_{k\ge 1}\frac{1}{k!}\left(\frac{KM}{N-K}\right)^k\le \frac{eKM}{N-K}\le \frac{(\log n)^4}{np}\,.
	         \end{aligned}
	         \end{equation*}
    As a result, 
\begin{equation}\label{eq-expectation-of-R-1-sparse}
              \mathbb E\Big[\sum_{\ell\ge 1}|\operatorname{R}^\ell_s|\mid M_\eta\le E'_{r,s}\le 6M_\eta\Big]\le C(\log n)^4/p\ll n^{0.6}\,,
          \end{equation}
          where the last inequality follows since we assume $p=n^{-1/2+o(1)}$. 
          Similarly we have
          \begin{equation*}
          \begin{aligned}
          \Pb[\mathbf{HG}(N,M,K)\ge 3]=&\  \sum_{k\ge3}\frac{\binom{M}{k}\binom{N-M}{K-k}}{\binom{N}{K}}\le \sum_{k\ge3}\binom{M}{k}\left(1-\frac{K}{N}\right)^{-K+k}\left(\frac K {N-K}\right)^k\\
          \le&\ \left(1-\frac K N\right)^{-K}\sum_{k\ge3}\frac 1{k!}\left(\frac{MK}{N-K}\right)^k\le \frac{(KM)^3}{(N-K)^3}\le \frac{(\log n)^{12}}{(np)^3}\,,
          \end{aligned}
          \end{equation*}
and hence it follows
\begin{equation}\label{eq-expectation-of-R-2-sparse}
            \mathbb E\Big[\sum_{\ell\ge 3}|\operatorname{R}^\ell_s|\mid M_\eta\le E'_{r,s}\le 6M_\eta\Big]=o(1)\,.
          \end{equation}
          Therefore, we conclude from Markov inequality that $|\operatorname{R}^1_s|+|\operatorname{R}^2_s|\le n^{0.6}$ and $\operatorname{R}^\ell_s=\emptyset$ for any $\ell\ge 3$ with high probability, which completes the proof of the lemma.
\end{proof}
Recall that $$\mathcal D_{r,s}=\bigcap_{a_\eta<k<s}\{E_{r,k}\leq O_k\}=\bigcap_{a_\eta<k<s}\{\sum_{j\in \operatorname{N}_k}\Gs^*_{r,\pi^*(j)}\leq O_k\}\,.$$
By introducing the random sets $\operatorname{R}^\ell_s,\ell\in \mathbb N$, we obtain that $\Pb[\mathcal D_{r,s}\mid M_\eta\le E'_{r,s}\le 6M_\eta]$ equals to
\begin{equation}
%\begin{split}
\mathbb P\Big[E_{r,k}-\!\sum_{j\in \operatorname{N}_k\cap\operatorname{N}_s}\Gs_{r,\pi^*(j)}\leq O_{k}-\ell,\forall \ell\ge 0\text, k\in\operatorname{R}^\ell_{s}\,|\,M_\eta\leq E'_{r,s}\leq  6M_\eta\Big]\,.
%\end{split}
\notag
\end{equation}
From the total probability formula,  we can express $\Pb[\mathcal D_{r,s}\mid M_\eta\le E'_{r,s}\le 6M_\eta]$ as
\begin{equation*}
    \begin{aligned}
&\sum_{(R^0_s,R^1_s,\dots)}\Pb[\mathcal R_s=(R^0_s,R^1_s,\dots)\mid M_\eta\le E'_{r,s}\le 6M_\eta]\times \\
&\Pb\Big[E_{r,k}-\!\sum_{j\in \operatorname{N}_k\cap\operatorname{N}_s}\Gs_{r,\pi^*(j)}\leq O_{k}-\ell,\forall \ell\ge 0\text, r\in R^\ell_{s}\mid M_\eta\leq E'_{r,s}\leq  6M_\eta,\mathcal R_s=(R^0_s,R^1_s,\dots)\Big]\,,
    \end{aligned}
\end{equation*}
where the summation is taken over all possible realizations of $\mathcal R=(\operatorname{ R^0_s,\operatorname{R}^1_s,\dots)}$. 
Note that for any $r\in \Rs_{s-1}$, 
$$E_{r,k}-\sum_{j\in \operatorname{N}_k\cap \operatorname{N}_s}\Gs_{r,\pi^*(j)}=\sum_{j\in\operatorname{N}_k\setminus\operatorname{N}_{s}}\mathsf G_{r,\pi^*(j)}+\sum_{j\in \operatorname{N}_k}X_{r,\pi^*(j)}\,,$$
which is independent with the random variables $\Gs_{r,\pi^*(j)},j\in \operatorname{N}_s$. Conversely, 
the event $\{M_\eta\le E'_{r,s}\le 6M_\eta\}$ together with the random sets $\operatorname{R}^\ell_s,\ell\in \mathbb N$ are all measurable with respect to these variables, suggesting that the conditioning in the second probability term can be removed.
By combining this with Lemma~\ref{lem-typical-sets-are-typical}, it suffices to show that for any realization of $\mathcal R_s=(\operatorname{R}^0_s,\operatorname{R}^1_s,\dots)$ that is typical (denoted as $(R^0_s,R^1_s,\dots)$) the following inequality holds:

\[
\Pb\Big[E_{r,k}-\sum_{j\in\operatorname{N}_k\cap\operatorname{N}_{s}}\mathsf G_{r,\pi^*(j)}\leq O_{k}-\ell,\forall \ell\ge 0, k\in{R}^\ell_{s}\Big]\ge [1-o(1)]\Pb[\mathcal D_{r,s}]\,.
\]
To tackle this relation, a trivial lower bound simplifies our goal to show
   \begin{equation*}
   \label{eq:finalgoal_sparse}
       \Pb\Big[E_{r,k}\leq O_{k}-\ell,\forall \ell\ge 0, k\in R^{\ell}_{s}\,\mid\mathcal D_{r,s}\Big]\geq 1-o(1)\,,
   \end{equation*}
   or equivalently (recall that $R^\ell_s=\emptyset$ for $\ell\ge 3$ since $(R^0_s,R^1_s,\dots)$ is typical),
    \begin{equation}
\label{eq:finalfinal_goal_sparse}
        \Pb\Big[\exists \ell\in\{1,2\},k\in R^\ell_{s}\mbox{ s.t.}E_{r,k}> O_{k}-\ell\mid\mathcal D_{r,s}\Big]\leq o(1).
    \end{equation}
   Since the event in the probability is increasing and $\mathcal D_{r,s}$ is decreasing, the FKG inequality and a simple union bound imply that the left hand side of \eqref{eq:finalfinal_goal_sparse} is bounded by
	\begin{equation*}
 \begin{aligned}
		&\ \Pb\Big[\exists \ell\in\{1,2\},k\in R^\ell_{s}\mbox{ s.t. }E_{r,k}> O_{k}-\ell\Big]\\
  \leq&\  \big(|R^1_s|+|R^2_s|\big)\max_{a_\eta<k<s}\Pb[{E'_{r,k}}>O_k-3]\\
  \le&\ n^{0.6}\times \Pb[\mathbf B(2np,p)\ge M_\eta-3]\,,
  \end{aligned}
	\end{equation*}
 where in the second inequality we used the facts that $|R^1_s|+|R^2_s|\le n^{0.6}$ by the assumption that $(R^0_s,R^1_s,\dots)$ is typical, $|\operatorname{N}_k|\le 2np$ by $\mathcal G_0$ and $O_k\ge M_\eta$ by $\Gc_{s-1}$. From Lemma~\ref{lem:poisson_tail_approx} we see the probability term is $n^{-1+\eta+o(1)}$, thus \eqref{eq:finalfinal_goal_sparse} holds for sufficiently small $\eta$. This concludes Proposition~\ref{prop-tail-prob>logn/n} and completes the analysis for $p$ in the sparse regime.    

 %\section{Algorithmic lower bounds in the dense regime}\label{sec-dense-algo}
%This section serves for proving
%Theorem~\ref{thm-algo}, where we divide the proof of (i) and (ii) into separated parts.

\subsection{Analysis of greedy algorithm: the dense regime}\label{subsec-algo-analysis-dense}
Now we focus on the regime $p_c\ll p\le 1/(\log n)^4$ and prove Proposition~\ref{prop-A-eta-is-ideal-online}, thereby concluding the first item in Theorem~\ref{thm-algo}.

%\noindent Note that $\mathcal{A}_\eta$ is an online algorithm that runs in $O(n^3)$ time for any $\eta>0$. Therefore, Proposition~\ref{prop-A-eta-is-ideal-online} implies statement (i) in Theorem~\ref{thm-algo} (e.g. by choosing $\mathcal{A}^*=\mathcal{A}_{\varepsilon/20}$). 

To this end, we will continue to work with the notations and conventions developed in Section~\ref{subsec-algo-framework}.
%Write $a_\eta=\lfloor \eta n\rfloor$ and $b_\eta=\lfloor (1-\eta) n\rfloor$ for simplicity. 
Note that 
\[
\operatorname{O}(\pi^*)=\sum_{s=1}^nO_s-\sum_{i<j}G_{i,j}X_{\pi^*(i),\pi^*(j)}\ge\sum_{s=1}^n O_s-1\,,
\]
We claim that it is sufficient to establish the following lower bounds with high probability.
\begin{proposition}\label{prop-lower-bounds-dense}
    Denote $\Gc_{\operatorname{first}}$ for the event that 
    \begin{equation}\label{eq-first-steps-good}
    \sum_{s=1}^{a_\eta}O_s\ge \sum_{1\le s\le a_\eta}sp^2-\eta D_{n,p}\,,
    \end{equation}
    $\Gc_{\operatorname{middle}}$ for the event that
    \begin{equation}\label{eq-middle-steps-good}
       O_s\ge sp^2+\sqrt{2(1-2\eta)sp^2\log n},\forall a_\eta<s\le b_\eta\,,
    \end{equation}
    and $\Gc_{\operatorname{last}}$ for the event that %$G$ satisfies $\sum_{b_\eta<s\le n}|\operatorname{N}_s|\ge \sum_{b_\eta<s\le n}sp-\sqrt{n^3\log n}$ and
    \begin{equation}\label{eq-last-steps-good}
        O_s\ge sp^2-\sqrt{10np^2\log n},\forall b_\eta<s\le n\,.
    \end{equation}
    Then $\Pb[\Gc_{\operatorname{first}}\cap \Gc_{\operatorname{middle}}\cap \Gc_{\operatorname{last}}]=1-o(1)$.
\end{proposition}
We claim that these lower bounds are adequate for proving Proposition~\ref{prop-A-eta-is-ideal-online}.
\begin{proof}[Proof of Proposition~\ref{prop-A-eta-is-ideal-online} assuming Proposition~\ref{prop-lower-bounds-dense}] 
Recall that $D_{n,p}=\sqrt{n^3p^2\log n}$.
    Under $\Gc_{\operatorname{first}}\cap \Gc_{\operatorname{middle}}\cap \Gc_{\operatorname{last}}$, we have $\operatorname{O}(\pi^*)+1$ is bounded below by
\begin{equation*}
    \begin{aligned}
   &\ \sum_{s\le a_\eta}sp^2-\eta D_{n,p}+\sum_{a_\eta<s\le b_\eta}\big(sp^2+\sqrt{2(1-2\eta)sp^2\log n}\big)+\sum_{b_\eta<s\le n}\big(sp^2-\sqrt{10np^2\log n}\big)\\
    	\geq&\ \sum_{s\le n}sp^2-\eta D_{n,p} +(\beta_c-2\eta)D_{n,p}-\eta n\cdot \sqrt{10np^2\log n}
   	\ge \binom{n}{2}p^2+(\beta_c-20\eta) D_{n,p}+1\,,
    \end{aligned}
\end{equation*}
where the first inequality follows from the fact that
\begin{equation}
   \sum_{a_\eta<s\le b_\eta}\sqrt{2s}\ge \int_{a_\eta }^{b_\eta}\sqrt{2x}\operatorname{d} x=\sqrt{\frac{8x^3}{9}}\ \bigg|^{b_\eta}_{a_\eta}\ge (\beta_c-\eta)n\sqrt n\,.\notag
 \end{equation}
This suggests that under $\Gc_{\operatorname{first}}\cap \Gc_{\operatorname{middle}}\cap \Gc_{\operatorname{last}}$ it holds $\pi^*\in \operatorname{S}_{\beta_c-20\eta}(G,\Gs)$. Since $\Gc_{\operatorname{first}}\cap \Gc_{\operatorname{middle}}\cap \Gc_{\operatorname{last}}$ occurs with high probability, the proof is completed.
\end{proof}

The remainder of this section is dedicated to proving Proposition~\ref{prop-lower-bounds-dense}. It is straightforward to deal with the first event, since
$$
\sum_{1\le s\le a_\eta}O_s\ge \sum_{1\le i<j\le a_\eta}G_{i,j}\Gs_{i,j}\,,
$$
which is a binomial variable $\mathbf{B}\big(\binom{a_\eta}{2},p^2\big)$ and $\Pb[\mathcal{G}_{\text{first}}]=1-o(1)$ follows from \eqref{eq-chernoff-bound-lower}.

The proof demonstrating that both $\Gc_{\text{middle}}$ and $\Gc_{\text{last}}$ occur with high probability is much more involved. We provide the detailed explanation in the following two sub-subsections, respectively. However, before delving into the details of these two events, we introduce the good event $\Gc_0$ in the dense regime concerning $G$, which will serve as the basis for our analysis.
\begin{definition}[Good event in the dense regime]\label{def-good-event-dense}
        Denote $\Gc_0$ for the event that $G$ satisfies 
        \begin{itemize}
            \item $\big||\operatorname{N}_s|-sp\big|\le \sqrt{4sp\log n},\forall a_\eta<s\le b_\eta$;
            \item  $|\operatorname N_r\cap \operatorname N_s|\le 2np^2,\forall a_\eta<r< s\le b_\eta$.
        \end{itemize}
\end{definition}
Similar to Lemma~\ref{lem-G_0-sparse}, we can show, using a simple union bound, that when $p$ falls within the dense regime, it holds $\Pb[G\in \Gc_0] = 1 - o(1)$. Hence, it suffices to prove that under any realization of $G$ that satisfies $\Gc_0$, both $\Gc_{\operatorname{middle}}$ and $\Gc_{\operatorname{last}}$ are typical events. For the remainder of this section, we will fix $G \in \Gc_0$ and treat it as a deterministic graph. %In particular, the sets $\operatorname{N}_k$, $a_\eta < k \leq b_\eta$, are given. To simplify the notation, we will continue to use $\Pb$ to denote the product measure of $\Gs_{i,j}\sim \mathbf B(1,p)$ and $X_{i,j}\sim \operatorname{U}[0,1/n^2]$.

%Therefore in the following we should keep in mind that $\Pb$ is actually a conditional law.

\subsubsection{Controlling the last steps}\label{subsubsec-last}

%and we need to introduce several new notations and definitions. 

For the convenience of analysis, we use another way to encode the random variables $\{O_s\}_{a_\eta<s\le b_\eta}$. 
Recall that by definition, for any $a_\eta<s\le b_\eta$ and any $r\in \Rs_{s-1}$,
\[
F_s(x)=\Pb[E_{r,s}\le x\mid \mathcal F_{s-1}]=\Pb[E_{r,s}\le x\mid \mathcal D_{r,s}]=\Pb[E_{r,s}\le x\mid E_{r,k}\le O_k,\forall a_\eta<k<s]\,.
\]
We define a sequence of random variables $\{c_s\}_{a_\eta<s\le b_\eta}$ such that
\begin{equation}\label{eq-def-c-s}
    1-F_s(O_s)=\frac{c_s}{n-s+1},\forall a_\eta<s\le b_\eta\,.
\end{equation}
It is evident that for any $a_\eta<s\le b_{\eta}+1$, there is a one-to-one correspondence between the two sequences $\{O_k\}_{a_\eta<k<s}$ and $\{c_k\}_{a_\eta<k<s}$. Hence all the $c_k$'s with $k<s$ are measurable with respect to $\mathcal{F}_{s-1}$. We provide further properties of the sequence $\{c_s\}_{a_\eta<s\le b_\eta}$ in the following lemma. %Recall that for any $a_\eta<s\le b_\eta+1$ and $r\in \Rs_{s-1}$,$\mathcal {D}_{r,s}$ is the event that $E_{r,k}\leq O_k$ for all $a_\eta < k < s$.
\begin{lemma}\label{lem-c-s}
For any $a_\eta<s\le b_\eta+1$, given the sequence $\{O_k\}_{a_\eta<k<s}$ (and hence also $\{c_k\}_{a_\eta<k<s}$), the following identity holds for any $r\in \Rs_{s-1}$,
\begin{equation}\label{eq-prob-of-D-r-s}
    \Pb[\mathcal D_{r,s}]=\prod_{a_\eta<k<s}\left(1-\frac{c_k}{n-k+1}\right)\,.
\end{equation}
Furthermore, given any realization of $\{O_k\}_{a_\eta<k<s},\{c_k\}_{a_\eta<k<s}$, the conditional distribution of $c_s$ is stochastically dominated by an exponential variable with rate $1$.
\end{lemma}
\begin{proof}
    By the multiplicative rule and the definition of $F_k,a_\eta<k\le b_\eta$ we have
    \begin{equation}
        \begin{aligned}
            \Pb[\mathcal D_{r,s}]=&\ \Pb[E_{r,k}\le O_k,\forall a_\eta<k<s]\\
            =&\ \prod_{a_\eta<k<s}\Pb[E_{r,k}\le O_k\mid E_{r,l}\le O_l.\forall a_\eta<l<k]\\
            =&\ \prod_{a_\eta<k<s}F_k(O_k)=\prod_{a_\eta <k<s}\left(1-\frac{c_k}{n-k+1}\right)\,,
        \end{aligned}
        \notag
    \end{equation}
    concluding \eqref{eq-prob-of-D-r-s}. Furthermore, given any realization of $\mathcal F_{s-1}$ (including the realization of $\{O_k\}_{a_\eta<k<s},\{c_k\}_{a_\eta<k<s}$), we see for any $x\ge 0$, it holds that
    \begin{equation*}
    \begin{aligned}
        \Pb[c_s\ge x\mid \mathcal F_{s-1}]=&\ \Pb\left[1-F_{s}(O_s)\ge \frac{x}{n-s+1}\mid \mathcal F_{s-1}\right]\\
        =&\ \Pb\left[F_s\left(\max_{r\in \Rs_{s-1}}E_{r,s}\right)\le 1-\frac{x}{n-s+1}\mid \mathcal F_{s-1}\right]
        \\=&\ \Pb\left[E_{r,s}\le F_s^{-1}\left(1-\frac{x}{n-s+1}\right),\forall r\in \Rs_{s-1}\mid\mathcal F_{s-1} \right]\\
        =&\ \prod_{r\in \Rs_{s-1}}\Pb\left[E_{r,s}\le F_s^{-1}\left(1-\frac{x}{n-s+1}\right)\mid \mathcal F_{s-1}\right]\\
        =&\ \left(1-\frac{x}{n-s+1}\right)^{n-s+1}\le e^{-x}\,,
    \end{aligned}
\end{equation*}
where the second-to-last equality follows from conditional independence, and the last equality arises from the definition of $F_s$. This indicates that for any $a_\eta < s \leq b_\eta$, the conditional distribution of $c_s$ given $\{c_k\}_{a_\eta<k<s}$ is always dominated by an exponential variable with a rate of 1, concluding the lemma.
\end{proof}
With Lemma~\ref{lem-c-s} in hand, we now prove $\Pb[\mathcal G_{\operatorname{last}}]=1-o(1)$. %It is easy to show that 
%\[
%\sum_{s=b_\eta+1}^n |\operatorname{N}_s|\ge \sum_{s=b_\eta+1}^n sp-\eta D_{n,p}
%\]
%holds with high probability. For the remaining, 
Define $\Gc^*$ as the event that $$\sum_{s=a_{\eta}+1}^{b_\eta}\frac{c_s}{n-s+1}\le {2}/{\eta}\,.$$
Note that $\Gc^*$ is measurable with respect to $\mathcal F_{b_\eta}$. We will show that
$\Pb[\Gc^*]=1-o(1)$, and for any realization of $\mathcal F_{b_\eta}$ that satisfies $\Gc^*$, it holds for any $b_\eta+1\le s\le n$, 
\begin{equation}\label{eq-each-last-step-is-not-bad}
    \Pb[O_s\le sp^2-\sqrt{10np^2\log n}\mid \mathcal F_{s-1},\Gc^*]=o(1/n)\,.
\end{equation}
Once this is established, by applying a union bound we obtain
\begin{align*}
&\ \Pb\big[\exists b_\eta<s\le n,O_s\le |\operatorname{N}_s|p-\sqrt{10np^2\log n}\big]\\
\le&\  \Pb[(\Gc^*)^c]+\sum_{s=b_\eta+1}^n\Pb\big[O_s\le |\operatorname{N}_s|p-\sqrt{10np^2\log n}\mid \Gc^*\big]=o(1)\,,
\end{align*}
and the main result follows.

From Lemma~\ref{lem-c-s}, we observe that the sequence $\{c_s\}_{a_\eta<s\le b_\eta}$ is stochastically dominated by i.i.d. exponential variables with mean $1$. Therefore, by the law of large numbers, it follows that with high probability,
\[
\sum_{s=a_\eta+1}^{b_\eta}\frac{c_s}{n-s+1}\le \frac{1}{\eta n}\sum_{s=a_\eta+1}^{b_\eta}c_s\le 2/\eta\,.
\]
This implies that $\Pb[\Gc^*]=1-o(1)$. The main purpose of $\Gc^*$ is as follows: whenever $\mathcal F_{b_\eta}$ satisfies $\Gc^*$, we get from equation \eqref{eq-prob-of-D-r-s} that for any $r\in \Rs_{b_\eta}$,
\begin{equation}\label{eq-lower-bound-of-Prob-D-s}
    \Pb[\mathcal D_{r,b_\eta}]=\prod_{s=a_\eta+1}^{b_\eta}\left(1-\frac{c_s}{n-s+1}\right)\ge \exp\Bigg(-\sum_{s=a_\eta+1}^{b_\eta}\frac{2c_s}{n-s+1}\Bigg)\ge \exp(-4/\eta)\,.
\end{equation}

Now we fix a realization of $\mathcal F_{b_\eta}$ that satisfies $\Gc^*$. Recall that in the last $n-b_\eta$ steps, $\mathcal A_\eta$ completes $\pi^*$ to a permutation such that $\pi^*(b_\eta+1)<\cdots<\pi^*(n)$ and hence the values of $\pi^*(b_\eta+1),\cdots,\pi^*(n)$ are measurable with respect to $\mathcal{F}_{b_\eta}$. Conditioned on $\mathcal F_{b_\eta}$,  for each $b_\eta+1\le s\le n$, we have $O_s=\sum_{j<s}G_{j,s}\Gs^*_{\pi^*(j),\pi^*(s)}$ which is measurable with respect to $\Gs^*_{j,\pi^*(s)},1\le j\le n$. Therefore, $O_s$ is conditionally independent with the events $\mathcal D_{r,b_\eta}$ for $r\in \Rs_{b_\eta}\setminus \{\pi^*(s)\}$. By combining this observation with \eqref{eq-lower-bound-of-Prob-D-s} (let us denote the event $\{O_s \leq sp^2 - \sqrt{10np^2\log n}\}$ as $\mathcal{A}_s$), we can conclude that
    \begin{align*}\label{eq-prob-upper-A-s}
        &\ \Pb[\mathcal A_s\mid \mathcal F_{b_\eta}]=\Pb[\mathcal A_s\mid \mathcal D_{b_\eta}]=\Pb[\mathcal A_s \mid \mathcal D_{\pi^*(s),b_{\eta}}]\le \frac{\Pb[\mathcal A_s]}{\Pb[\mathcal D_{\pi^*(s),b_\eta}]}\le \exp(4/\eta)\Pb[\mathcal A_s]\,.
        %\\
        %\stackrel{\eqref{eq-lower-bound-of-Prob-D-s}}{\le}&\ \exp(4/\eta)\cdot\Pb[\mathcal A_s]\le \exp(4/\eta)\cdot \Pb\Big[\sum_{j<s}G_{j,s}\Gs_{\pi^*(j),\pi^*(s)}\le sp^2-\sqrt{10np^2\log n}\Big]\\
        %{\le}&\ \exp(4/\eta-10\log n)=o(1/n)\,,
    \end{align*}
    
Finally, Note that $$O_s\ge \sum_{j<s}G_{j,s}\Gs_{\pi^*(j),\pi^*(s)}=\sum_{j\in \operatorname{N}_s}\Gs_{\pi^*(j),\pi^*(s)}\,,$$
which is a binomial variable $\mathbf B(|\operatorname{N}_s|,p)$. And from the definition of $\Gc_0$ we see $$sp^2-\sqrt{10np^2\log n}\le |\operatorname{N}_s|p-\sqrt{4|\operatorname{N}_s|p\log n}\,.$$ Therefore, by applying \eqref{eq-chernoff-bound-lower} we get
$$\Pb[\mathcal A_s]\le \Pb\big[\mathbf B(|\operatorname{N}_s|,p)\le |\operatorname{N}_s|p-\sqrt{4|\operatorname{N}_s|p\log n}\big]{\le} \exp(-2\log n)\,.$$ %by \eqref{eq-chernoff-bound-lower}. 
As a result, $\Pb[\mathcal A_{s}\mid \mathcal F_{b_\eta}]\le \exp(4/\eta-2\log n)=o(1/n)$, which establishes \eqref{eq-each-last-step-is-not-bad} and thus completes the proof of $\Pb[\Gc_{\operatorname{last}}]=1-o(1)$.

\subsubsection{Controlling the middle steps}

Finally, we prove that $\Pb[\Gc_{\operatorname{middle}}] = 1 - o(1)$. Similar as before, we begin by defining some good events $\Gc_s,a_\eta<s\le b_\eta$ in the dense regime.% and this is where the technical assumption $p\le 1/(\log n)^4$ is used. Similar as before, we present several good events that we will work on. Recall that $\operatorname{N}_s=\{j<s:G_{j,s}=1\}$ for any $s\in\operatorname{S}(\eta)$.
\begin{definition}
   % Denote $\Gc_0$ for the event that $G$ satisfies
    %\begin{itemize}
     %   \item $(1-\eta)sp\le |\operatorname{N}_s|\le (1+\eta)sp,\forall s\in \operatorname{S}(\eta)$; 
      %  \item $|N_r\cap N_s|\le 2np^2,\forall r\neq s\in \operatorname{S}(\eta)$.
    %\end{itemize}
    For each $a_\eta<s\le b_\eta$, define $\Gc_s$ as the event that 
    \begin{itemize}
        \item $c_k\le 2\log n,\forall a_\eta<k\le s$ and $\sum_{a_\eta<k\le s}{c_k}\le 2n$;
        \item $\sqrt{2(1-\eta)|\operatorname{N}_k|p\log n}\le O_k-|\operatorname{N}_k|p\le \sqrt{10np^2\log n}, \forall a_\eta\leq k\leq s$;
        \item For each $a_\eta<k\le s$, it holds that
    \begin{equation}\label{eq-comparison-to-F*}
            1-F_k(O_k)\ge \big(1-o(1)\big)\big(1-F_k^*(O_k)\big)\,,%=1-\frac{c_k+o(c_k)}{n-k+1}\,,
        \end{equation}
        (recall that $F^*_k$ is the distribution function of $E_{r,k}$ without conditioning on $\mathcal D_{r,k}$) where the $o(1)$ terms are uniform for each $a_\eta<k\le s$.
    \end{itemize}
\end{definition}
%\begin{proposition}\label{prop-good-event-high-prob-dense}
%    $\Pb[\Gc_0]=1-o(1)$ and  $\Pb[\Gc_{b_\eta}\mid \Gc_0]=1-o(1/n)$.
%\end{proposition}
\noindent Note that by the first item in $\Gc_0$ (Definition~\ref{def-good-event-dense}) it holds
\[
|\operatorname{N}_s|p+\sqrt{2(1-\eta)|\operatorname{N}_s|p\log n}\ge sp^2+\sqrt{2(1-2\eta)sp^2\log n}\,,
\]
so $\mathcal G_{b_\eta}\subset \mathcal G_{\operatorname{middle}}$. Now it remains to show that under any realization of $G\in \Gc_0$, $\Gc_{b_\eta}$ holds with high probability.
%Proposition~\ref{prop-good-event-high-prob-dense} implies $\Pb[\Gc_{\operatorname{middle}}]\ge \Pb[\operatorname{G}_{b_\eta}]\ge \Pb[\operatorname{G}_0\cap \Gc_{b_\eta}]\ge 1-o(1)$.

%Now we prove this final proposition. It is easy to show $\Pb[\Gc_0]=1-o(1)$ by a simple union bound as in Lemma~\ref{lem-G_0-sparse} and we omit the details. To show the remaining we fix a realization of $G$ that satisfies $\Gc_0$ and henceforth we assume $G$ is deterministic. 
Write $\Gc_{a_\eta}$ for the trivial event. Our goal is to show that for any $a_\eta<s\le b_\eta$, it holds the following inequality
$$\Pb[\Gc_{s-1}]-\Pb[\Gc_{s}]\le \Pb[\Gc_s^c\cap \Gc_{s-1}]=o(1/n)\,.$$
Therefore, for the rest part we fix $a_\eta < s \leq b_\eta$ and focus on establishing this relation.

It can be deduced from Lemma~\ref{lem-c-s} that the events $c_s\le 2\log n$ and $\sum_{a_\eta<s\le s}c_s\le 2n$ fail with probability $o(1/n)$.
Additionally, utilizing the FKG inequality, we can show that (note that $\{O_s\ge |\operatorname{N}_s|p+\sqrt{10|\operatorname{N}_s|p\log n}\big\}$ is an increasing event while $\mathcal D_{s}$ is a decreasing event)
\begin{align*}
&\ \Pb\big[O_s\ge |\operatorname{N}_s|p+\sqrt{10|\operatorname{N}_s|p\log n}\mid \mathcal F_{s-1}\big]=\Pb\big[O_s\ge |\operatorname{N}_s|p+\sqrt{10|\operatorname{N}_s|p\log n}\mid \mathcal D_{s}\big]\\
\le&\ \Pb\big[O_s\ge |\operatorname{N}_s|p+\sqrt{10|\operatorname{N}_s|p\log n}\big]
\le n\Pb\big[\mathbf B(|\operatorname{N}_s|,p)\ge |\operatorname{N}_s|p+\sqrt{10|\operatorname{N}_s|p\log n}-1\big]\,, 
\end{align*}
which is $o(1/n)$ by \eqref{eq-chernoff-bound}. Now we claim that the events $\{c_s\le 2\log n\}$, $\{\sum_{a_\eta<k\le s}c_k\le 2n\}$ together with $\Gc_{s-1}$ indeed imply that $O_s\ge |\operatorname{N}_s|p+\sqrt{2(1-\eta)|\operatorname{N}_s|p\log n}$ as well as \eqref{eq-comparison-to-F*} holding for $k=s$. This would complete the proof. 

To verify this claim, we take any $r^*\in \Rs_{s-1}$ and define $\Omega_s=\Omega_s^1\times \Omega_s^2=\{0,1\}^{|\operatorname{N}_s|}\times [0,1/n^2]^{|\operatorname{N}_s|}$, denoted by
\[
\big\{(g_{j},\dots,x_{j},\dots)_{j\in \operatorname{N}_s}:g_{j}\in \{0,1\}, 0\le x_{j}\le 1/n^2,\forall j\in \operatorname{N}_s\big\}\,,
\]
as the set of realizations of $(\Gs_{r^*,\pi^*(j)})_{j\in \operatorname{N}_s}$ and $(X_{r^*,\pi^*(j)})_{j\in \operatorname{N}_s}$. We see $\Pb$ restricts as a product measure of $|\operatorname{N}_s|$ Bernoulli variables with parameter $p$ and $|\operatorname{N}_s|$ uniform variables in $[0,1/n^2]$. Denote 
\begin{equation}\label{eq-HG-upper-bound-control}
    T_k=
\begin{cases}
    |\operatorname{N}_k\cap \operatorname{N}_s|p+\sqrt{10|\operatorname{N}_k\cap\operatorname{N}_s|p\log n}\,,&p\le 1/(\log n)^4\text{ and }np^3\ge (\log n)^3\,,\\
    2(\log n)^3\,,&n^{-0.1}\leq np^3<(\log n)^3\,,\\
    20\,,&p\gg p_c\text{ and }np^3< n^{-0.1}\,.
\end{cases}
\end{equation}
we define $\Omega_s^*\subset \Omega_s^1$ as the set of coordinates $\omega=(g_j,\dots)_{j\in \operatorname{N}_s}$ that satisfies
\begin{equation}\label{eq-good-omega}
    \sum_{j\in \operatorname{N}_k\cap \operatorname{N}_s}g_{j}\le T_k,\forall a_\eta<k<s\,.
\end{equation}
Furthermore, denote $I_s$ for the interval $$\Big[|\operatorname{N}_s|p+\sqrt{2(1-\eta)|\operatorname{N}_s|p\log n},|\operatorname{N}_s|p+\sqrt{10|\operatorname{N}_s|p\log n}\Big]$$ and for any $x\in I_s$ we define $\Pb_{s,x}$ as the conditional distribution on $\Omega_s$ given that 
\[
E_{r^*,s}=\sum_{j\in \operatorname{N}_s}\Gs^*_{r^*,\pi^*(j)}=\sum_{j\in \operatorname{N}_s}(\Gs_{r^*,\pi^*(j)}+X_{r^*,\pi^*(j)})\ge x\,.
\]
We will utilize the following two lemmas, the proofs of which have been postponed to the appendix.
\begin{lemma}\label{lem-HG-control}
    For any $x\in I_s$ it holds $\Pb_{s,x}[\omega_1\in \Omega_s^*]=1-o(1/n)$.
\end{lemma}
\begin{lemma}\label{lem-asmptotic-of-tail}
    Denote $\delta(n)=2/\sqrt{\log n}$, then when $p_c\ll p\le 1/(\log n)^4$, for each $a_\eta<k<s$ it holds 
    \begin{equation}\label{eq-asymptotic-tail}
            \Pb\Big[\sum_{j\in \operatorname{N}_k\setminus \operatorname{N}_s}G_{r^*,\pi^*(j)}\le 
            O_k-T_k-1\Big]\ge F_k^*(O_k)-\frac{2\delta(n)(c_k+1)}{n-k+1}\,.
    \end{equation}
\end{lemma}

We will first show that \eqref{eq-comparison-to-F*} holds when replacing $O_s$ with any $x \in I_s$. For any $x \in I_s$, by definition we have
\[
1-F_s(x)=\Pb[E_{r^*,s}\ge x\mid \mathcal D_{r^*,s}]=\frac{\Pb[E_{r^*,s}\ge x,\mathcal D_{r^*,s}]}{\Pb[\mathcal D_{r^*,s}]}\,.
\]
%The denominator equals to $\prod_{a_\eta<k<s}\big(1-\frac{c_k}{n-k+1}\big)$ by \eqref{eq-prob-of-D-r-s} in Lemma~\ref{lem-c-s}. 
For the numerator, it can be rewritten as
\[
\Pb[E_{r^*,s}\ge x,\mathcal D_{r^*,s}]=\Pb[E_{r^*,s}\ge x]\cdot \Pb_{s,x}[\mathcal D_{r^*,s}]=\big(1-F_s^*(x)\big)\Pb_{s,x}[\mathcal D_{r^*,s}]\,.
\]
We expand $\Pb_{s,x}[\mathcal D_{r^*,s}]$ according to the realization of $\omega_1=(g_j,\dots)_{j\in \operatorname{N}_s}$:
\begin{equation}\label{eq-expansion-omega-1}
        \Pb_{s,x}[\mathcal D_{r^*,s}]=\sum_{\omega_1\in \{0,1\}^{|\operatorname{N}|_s}}\Pb_{s,x}[\omega_1]\cdot \Pb_{s,x}[\mathcal D_{r^*,s}\mid \omega_1]
        \ge\sum_{\omega_1\in \Omega_s^*}\Pb_{s,x}[\omega_1]\cdot\Pb_{s,x}[\mathcal D_{r^*,s}\mid \omega_1]\,.
\end{equation}
For any realization $\omega_1\in \Omega_s^*$, we have
\begin{equation*}
    \begin{aligned}
        &\ \Pb_{s,x}[\mathcal D_{r^*,s}\mid \omega_1]=\Pb[E_{r^*,k}\le O_k,\forall a_\eta<k<s\mid \omega_1,E_{r^*,s}\ge x]\\
        \ge &\ \Pb\Big[\sum_{j\in \operatorname{N}_k}\Gs_{r^*,\pi^*(j)}\le O_k-1,\forall a_\eta<k<s\,\big\vert\, \omega_1,E_{r^*,s}\ge x\Big]\\
        =&\ \Pb\Big[\sum_{j\in \operatorname{N}_k\setminus \operatorname{N}_s}\Gs_{r^*,\pi^*(j)}\le O_k-1-\sum_{j\in \operatorname{N}_s\cap \operatorname{N}_k}g_j\,\big\vert\, \omega_1,\sum_{j\in \operatorname{N}_s}X_{r^*,\pi^*(j)}\ge x-\sum_{j\in \operatorname{N}_s}g_j\Big]\\
        {=}&\ \Pb\Big[\sum_{j\in \operatorname{N}_k\setminus \operatorname{N}_s}\Gs_{r^*,\pi^*(j)}\le O_k-1-\sum_{j\in \operatorname{N}_k\cap \operatorname{N}_s}g_j,\forall a_\eta<k<s\Big]\,,
        \end{aligned}
    \end{equation*}
    where the last equality follows from independence. Note that $\omega_1\in \Omega_s^*$ means $$\sum_{j\in \operatorname{N}_k\cap \operatorname{N}_s}g_j\le T_k,\forall a_\eta<k<s\,,$$ Hence, the above expression is lower-bounded by
    \begin{equation*}
    \begin{aligned}
        \Pb\Big[\sum_{j\in \operatorname{N}_k\setminus \operatorname{N}_s}\Gs_{r^*,\pi^*(j)}\le O_k-T_k-1,\forall a_\eta<k<s\Big]\,,
        \end{aligned}
        \end{equation*}
        which is bounded below by
        \[\prod_{a_\eta<k<s}\Pb\Big[\sum_{j\in \operatorname{N}_k\setminus \operatorname{N}_s}\Gs_{r^*,\pi^*(j)}\le O_k-T_k-1\Big]
        \]
        from FKG inequality since all these events are decreasing.
        From Lemma~\ref{lem-asmptotic-of-tail} and the third item in $\Gc_{s-1}$, we see for any $a_\eta<k<s$, 
        \[
        \Pb\Big[\sum_{j\in \operatorname{N}_k\setminus \operatorname{N}_s}\Gs_{r^*,\pi^*(j)}\le O_k-T_k-1\Big]\ge F_k^*(O_k)-\frac{2\delta(n)(c_k+1)}{(n-k+1)}\ge 1-\frac{c_k+o(c_k+1)}{n-k+1}\,,
        \]
        where the $o(\cdot)$ terms are uniform for all $a_\eta<k<s$. Plugging this into the above estimation, we see for any $\omega_1\in \Omega_s^*$, $\Pb_{s,x}[\mathcal D_{r^*,s}\mid \omega_1]$ is bounded from below by
        \begin{equation*}
        \begin{aligned}
        \prod_{a_\eta<k<s}\left(1-\frac{c_k+o(c_k+1)}{n-k+1}\right)
        \ge\big(1-o(1)\big)\prod_{a_\eta<k<s}\left(1-\frac{c_k}{n-k+1}\right)=\big(1-o(1)\big)\Pb[\mathcal D_{r^*,s}]\,,
    \end{aligned}
\end{equation*}
where in the first inequality we used the fact that $\sum_{a_\eta<s\le b_\eta}c_s\le 2n$ and the second equality follows from \eqref{eq-prob-of-D-r-s}. 
As a result, we see from \eqref{eq-expansion-omega-1} that 
\[
\Pb_{s,x}[\mathcal D_{r^*,s}]\ge \Pb_{s,x}[\omega_1\in \Omega_s^*]\times \big(1-o(1)\big)\Pb[\mathcal D_{r^*,s}]\ge \big(1-o(1)\big)\Pb[\mathcal D_{r,s}]\,,
\]
where the last inequality follows from Lemma~\ref{lem-HG-control}. Therefore, we get for any $x\in I_s$,
\begin{equation}\label{eq-comparison-of-tail}
    1-F_s(x)\ge \big(1-o(1)\big)(1-F_s^*(x))\,,
\end{equation}
which is equivalent to \eqref{eq-comparison-to-F*} with $O_k$ replaced by $x\in I_s$. It is also worth noting that the $o(1)$ term can be taken uniformly for $a_\eta < s \leq b_\eta$.

Finally, it can be deduced from Lemma~\ref{lem:Upper bound for binomial tail} that $$1-F_s^*\left(|\operatorname{N}_s|p+\sqrt{2(1-\eta)|\operatorname{N}_s|p\log n}\right )=n^{-1+\eta+o(1)}\gg \frac{\log n}{n-s+1}\,,$$
and thus $c_s\le 2\log n$ together with $\eqref{eq-comparison-of-tail}$ implies $O_s\ge |\operatorname{N}_s|p+\sqrt{2(1-\eta)|\operatorname{N}_s|p\log n}$. In particular, we conclude that %from \eqref{eq-comparison-of-tail} that
%\[
%F_s^*(O_s)=1-\big(1-o(1)\big)\big(1-F_s(O_s)\big)=1-\frac{c_s+o(c_s)}{n-s+1}\,,
%\]
%which verifies that 
\eqref{eq-comparison-to-F*} holds for $k=s$, and this completes the proof.

\begin{remark}\label{rmk-p=o(1)}
As mentioned in Remark~\ref{rmk-thm-algo}, we anticipate that the assumption $p\le 1/(\log n)^4$ is primarily for technical reasons and can potentially be removed. The proof presented above relies on the property that under this assumption, one has $$1-F_s(O_s)\ge\big(1-o(1)\big)(1-F_s^*(O_s))$$ with overwhelming probability. Therefore, a simple application of the FKG inequality is sufficient for our purposes. However, we believe that for any $p=o(1)$, it should hold that 
    $$
    1-F_s(O_s)\ge n^{-o(1)}(1-F_s^*(O_s))\,,
    $$
    which would also be enough to establish $O_s\ge sp^2+\sqrt{2(1-\eta)sp^2\log n}$. Unfortunately, deriving such estimates seems challenging, as the FKG inequality is not particularly tight when $p$ exceeds some poly-log factor threshold. Detailed calculations involving multiple binomial variables are required to obtain precise results.
\end{remark}

\section{Hardness results via the overlap-gap property}\label{sec-2-OGP}
   
In this section, we present evidence supporting the existence of statistical-computational gap in the dense regime based on various types of the overlap-gap property. Recall the property of admissibility defined as in Definition~\ref{def-admissible}. We fix an arbitrary $p$ in the dense regime, and in most part of this section, we will work with some specific realization of $\Gs$ which is $p$-admissible.
\subsection{Failure of stable algorithms via 2-OGP}\label{subsec-two-OGP}
This subsection is dedicated to proving the failure of stable algorithms in finding near-optimal solutions.
To formally state Theorem~\ref{thm-OGP}, we introduce several definitions and notations. Specifically, we use $\operatorname{d}(\cdot)$ to denote the Hamming distance on the symmetric group. We also define stable algorithms in $\operatorname{GAA}$ as follows.
   
   	\begin{definition}[stable algorithms]       
		We say an algorithm $\mathcal{A}\in\operatorname{GAA}$ is \emph{$(\Delta,\kappa)-$stable}, if\begin{equation}\label{eq-kappa-stable}
			\operatorname{d}(\mathcal{A}(G,\mathsf G),\mathcal{A}(G',\mathsf G))\leq \kappa
		\end{equation}
		holds for $G,G',\Gs\in \mathfrak{G}_n$ whenever $G$ and $G'$ differ at most $\Delta$ edges.% $\left((G,\Gs),(G',\Gs)\right)$  with $$ \sum_{i<j}|G_{i,j}-G'_{i,j}|\leq \Delta\,,$$ 
  
Furthermore, we say $\mathcal A\in \operatorname{GAA}$ is \emph{almost $(\Delta,\kappa)$-stable}, if \eqref{eq-kappa-stable} holds with probability $1-o(1/n^2)$ for $G,G',\Gs\in \mathfrak{G}_n$, where $(G,\Gs)$ is sampled from $\mathbf G(n,p)^{\otimes 2}$ and $G'$ is any graph that differs from $G$ at most $\Delta$ edges.
	\end{definition}
Write $\rho_0=1/3$. Fix any $\beta_0\in (\sqrt{25/27},1)$ and take $\eta>0$ such that
\begin{equation}\label{eq-choice-of-eta}
    2-\rho_0+\eta-\frac{2\beta_0^2}{1+(\rho_0+\eta)^2}<0\,.
\end{equation}
Note that this can be accomplished since \eqref{eq-choice-of-eta} holds when $\eta=0$. The following theorem presents our main evidence of algorithmic barriers in the dense regime.

\begin{thm}\label{thm-stable-hardness}
    When $p$ is in the dense regime, for any almost $(1,\eta n)$-stable algorithm $\mathcal A\in \operatorname{GAA}$, it holds
    \begin{equation}\label{eq-stable-cannot-succeed-with-high-prob}
        \Pb[\mathcal A(G,\Gs)\in \operatorname{S}_{\beta_0}(G,\Gs)]\le 1-1/n^2\,,
    \end{equation}
    where the probability is taken over $(G,\Gs)\sim \mathbf G(n,p)^{\otimes 2}$.
\end{thm}
   We offer several comments on Theorem~\ref{thm-stable-hardness}. Our result excludes the possibility of achieving near-optimal solutions with high enough probability for a broad class of algorithms. Moreover, although Theorem~\ref{thm-stable-hardness} states only for deterministic algorithms, similar hardness result can be obtained for randomized stable algorithms with verbal changes. While an upper bound of $1-1/n^2$ for the success probability may seem less than ideal, we still consider it as an indication of computational barriers for the following reasons:
\\\noindent (i) Efficient algorithms that can find near-optimal solutions (like the algorithms constructed for Theorem~\ref{thm-PTAS}) typically achieve this with extremely high probability, often super-polynomially close to 1 (though this is not always the case; some counterexamples were recently found in \cite{LS24}).
\\\noindent
(ii) Additionally, through further analysis, we can demonstrate that for sufficiently stable algorithms (e.g., $(\Delta,\kappa)$-stable algorithms with $\kappa^2\ll \Delta$), none of them can find near-optimal solutions with any non-vanishing probability.
\\\noindent
Based on these considerations, Theorem~\ref{thm-stable-hardness} provides convincing evidence for the computational hardness in the dense regime.

To prove Theorem~\ref{thm-stable-hardness}, we will make use of the overlap-gap property (2-OGP). Denote $N=\binom{n}{2}$. Recall that $\operatorname{U}$ represents the set of unordered pairs. We label the elements in $\operatorname{U}$ as $e_1,\dots,e_N$ according to lexicographical order. For illustration, $e_1=(1,2)$, $e_2=(1,3)$, $e_3=(2,3)$, and so on, up to $e_N=(n-1,n)$. This labeling possesses a specific property: for any $\alpha\in [0,1]$, the set of pairs $\operatorname{E}_\alpha=\{e_k:k\le \alpha N\}$ is contained within the first $\big(\sqrt{2\alpha}+o(1)\big)n$ integers. 

We begin by defining the $(2,\alpha)$-correlated instance.

    \begin{definition}[$(2,\alpha)$-correlated instance]
For any $\alpha\in [0,1]$, we call a pair of graphs $(G^1,G^2)$ is a $(2,\alpha)$-correlated instance, if 
  \begin{equation}
      \begin{split}
G^{1}_{e_k}=G^{2}_{e_k},1\le k\le \lfloor\alpha N\rfloor, 
      G^{i}_{e_k}, \lfloor\alpha N\rfloor+1\le k\le N, i=1,2\,,
      \end{split}\notag
  \end{equation}
  are independent $\mathbf B(1,p)$ variables.%and $\Gs$ is sampled independently from $\mathbf G(n,p)$.
\end{definition}
Recall the property of admissibility in Definition~\ref{def-admissible}. For any admissible graph $\Gs$, we provide a useful estimate for $(2,\alpha)$-correlated instances analogous to Proposition~\ref{prop-two-point-est}.
\begin{proposition}\label{prop-two-point-est-correlated}
 For any admissible graph $\Gs$, any $\alpha,\rho\in [0,1]$ and $\pi_1,\pi_2\in \operatorname{S}_n$ with $\operatorname{overlap}(\pi_1,\pi_2)=\rho n$, it holds that
\begin{align}\label{eq-two-point-est-correlated}
\Pb\left[\widetilde{\operatorname{O}}(\pi_1)\geq \beta_0D_{n,p}\,,\widetilde{\operatorname{O}}(\pi_2)\geq \beta_0 D_{n,p}\right]\leq\exp\left(-\frac{2\beta_0^2n\log n}{1+\rho^2\wedge \alpha}+o(n\log n)\right)\,,
\end{align}
where $\widetilde{\operatorname{O}}(\pi_i)$ is the centered overlap of $(G^i,\Gs)$ through $\pi_i$, $i=1,2$, and $\Pb$ is taken over $(2,\alpha)$-correlated instance $(G^1,G^2)$.
\end{proposition}
\begin{proof}
The proof of Proposition~\ref{prop-two-point-est-correlated} is very similar to that of Proposition~\ref{prop-two-point-est} so we only give a sketch. Since $\Gs$ is a deterministic admissible graph, we have
\[
{\operatorname{O}}(\pi_1)=\sum_{(i,j)\in \operatorname{U}}G^1_{i,j}\Gs_{\pi_1(i),\pi_1(j)}=S_0+S_1\,,\quad {\operatorname{O}}(\pi_2)=\sum_{(i,j)\in \operatorname{U}}G^2_{i,j}\Gs_{\pi_2(i),\pi_2(j)}=S_0+S_2\,,
\]
where $S_0,S_1,S_2$ are independent binomial variables given by 
\[
S_0=\sum_{(i,j)\in \operatorname{E}_\alpha}G^1_{i,j}\mathbf{1}_{\Gs_{\pi_1(i),\pi_1(j)}=\Gs_{\pi_2(i),\pi_2(j)}=1}\,,\quad S_1={\operatorname{O}}(\pi_1)-S_0\,,\quad S_2={\operatorname{O}}(\pi_2)-S_0.
\]

As in the proof of Proposition~\ref{prop-two-point-est}, the core lies in controlling the variance of $S_0$. On the one hand, we have
\begin{align*}
\operatorname{Var}(S_0)=&\ p(1-p)\sum_{(i,j)\in \mathbb{E}_\alpha}\mathbf{1}_{\Gs_{\pi_1(i),\pi_1(j)}=\Gs_{\pi_2(i),\pi_2(j)}=1}\\
\le&\ p\sum_{(i,j)\in \operatorname{U}}\mathbf{1}_{\Gs_{\pi_1(i),\pi_1(j)}=\Gs_{\pi_2(i),\pi_2(j)}=1}=p|\operatorname{OL}(\Gs,\pi_1\circ \pi_2^{-1})|\,,
\end{align*}
which is bounded by {$\big(\rho^2+o(1)\big)Np^2$} from the third item in admissibility. On the other hand, considering the specific property enjoyed by the labelling, we have
\[
\operatorname{Var}(S_0)\le p\sum_{(i,j)\in \operatorname{E}_\alpha}\mathbf{1}_{\Gs_{\pi_1(i),\pi_1(j)}=1}\le \sum_{i<j<(\sqrt{2\alpha}+o(1))n}\Gs_{\pi_1(i),\pi_1(j)}\,,
\]
which is bounded by $\big(\alpha+o(1)\big)Np^2$ from the second item of admissibility. As a result, $\operatorname{Var}(S_0)$ is bounded by $\big(\rho^2\wedge \alpha+o(1)\big)Np^2$ and this gives rise to the factor $\rho^2\wedge \alpha$ in the expression. We omit the detailed calculations.
\end{proof}

The main input for proving Theorem~\ref{thm-stable-hardness} is the following geometric property of $\beta_0$-optimal solution spaces concerning the $(2,\alpha)$-correlated instance. This property suggests that certain forbidden structures are highly unlikely to exist.
\begin{proposition}\label{prop-OGP}
    For a couple of graph pairs $(G^1,\Gs),(G^2,\Gs)$, denote $$\mathcal E_{\operatorname{FS}}=\mathcal E_{\operatorname{FS}}\big((G^1,\Gs),(G^2,\Gs)\big)$$ 
    as the event that there exists $\pi_i\in \operatorname{S}_{\beta_0}(G^i,\Gs),i=1,2$ such that \begin{equation}\label{eq-overlap-in-forbidden-region}
        \operatorname{overlap}(\pi_1,\pi_2)\in \big((\rho_0-\eta)n,(\rho_0+\eta)n\big)\,.
    \end{equation}
    Then for any $\alpha\in [0,1]$, it holds that  \begin{equation}\label{eq-OGP-high-prob}
        \Pb[\mathcal E_{\operatorname{FS}}\text{ happens for }(G^1,\Gs),(G^2,\Gs)]=o(1/n^2)\,,
    \end{equation}
    where the probability is taken over a $(2,\alpha)$-correlated instance $(G^1,G^2)$ together with another independent $\Gs\sim \mathbf G(n,p)$.
\end{proposition}
\begin{proof}
    Since $\Pb[\Gs\text{ is admissible}]= 1 - o(1/n^2)$ by Lemma~\ref{lem-good-event-on-G}, we can fix a realization of $\Gs$ which is admissible and work with this specific graph. The next step is to prove the result by applying a union bound. 
    
    On the one hand, we consider the enumeration of permutations pairs $(\pi_1, \pi_2)$ that satisfy equation \eqref{eq-overlap-in-forbidden-region}. It is easy to see that there are at most $$\sum_{\rho\in (\rho-\eta,\rho+\eta),\rho n\in \mathbb N}n!\times \binom{n}{\rho n}((1-\rho)n)!\le \exp\left((2-\rho_0+\eta)n\log n+O(n)\right)$$
many of such pairs. On the other hand, according to Proposition~\ref{prop-two-point-est-correlated}, for each of these pairs and any $\alpha \in [0,1]$, 
    \[
\Pb\left[\pi_1\in \operatorname{S}_{\beta_0}(G^1,\Gs),\pi_2\in \operatorname{S}_{\beta_0}(G^2,\Gs)\right]\le \exp\left(-\frac{2\beta_0^2n\log n}{1+(\rho_0+\eta)^2}+o(n\log n)\right)\,.
    \]
    As a result, the probability of $\mathcal E_{\operatorname{FS}}$ conditioned on $\Gs$ is bounded by 
    \[
    \exp\left(\Big[2-\rho_0+\eta-\frac{2\beta_0^2}{1+(\rho_0+\eta)^2}+o(1)\Big]n\log n \right)\stackrel{\eqref{eq-choice-of-eta}}{=}\exp\big(-\Omega(n\log n)\big)=o(1/n^2)\,,
    \]
    completing the proof.
\end{proof}

Now we proceed with the proof of Theorem~\ref{thm-stable-hardness}. We begin by sampling three graphs $G$, $G'$ and $\Gs$, independently from $\mathbf{G}(n,p)$. We will construct an interpolation path between $G$ and $G'$ as follows: for each $0 \leq k \leq N$, let $G^k$ be the graph defined by
    \[
    G^k_{e_i}=G_{e_i}\text{ for } 1\le i\le k\text{ and }G_{e_i}^k=G'_{e_i} \text{ for }i=k+1,\dots,N.
    \]
Then it is straightforward to verify that $(G^0,G^k)$ forms a $(2, k/N)$-correlated instance.
Define
    \begin{equation}\label{eq-E-OGP}
    \mathcal E_{\operatorname{OGP}}=\bigcap_{k=0}^N\mathcal E_{\operatorname{FS}}^c\big((G^0,\Gs),(G^k,\Gs)\big)
    \end{equation}
    as the event that the forbidden structure does not exist on the entire interpolation path. By combining Proposition~\ref{prop-OGP} with a union bound, we can conclude that $\Pb[\mathcal{E}_{\operatorname{OGP}}] = 1 - o(1)$.

    Now take any almost $(1,\eta n)$-stable algorithm $\mathcal A\in \operatorname{GAA}$, we define the following three events
    \begin{equation}\label{eq-E-suc}
        \mathcal E_{\operatorname{suc}}=\{\mathcal A(G^k,\Gs)\in \operatorname{S}_{\beta_0}(G^k,\Gs),\forall 0\le k\le N\}\,,
    \end{equation}
    \begin{equation}
        \mathcal E_{\operatorname{stable}}=\{\operatorname{d}\big(\mathcal A(G^k,\Gs),\mathcal A(G^{k+1},\Gs)\big)\le \eta n, \forall 0\le k\le N-1\}\,,
    \end{equation}
    \begin{equation}\label{eq-E-end}
    \mathcal E_{\operatorname{ends}}=\{\operatorname{overlap}\big(\mathcal A(G^0,\Gs),\mathcal A(G^N,\Gs)\big)\le (\rho-\eta)n\}\,.
    \end{equation}
\begin{proposition}\label{prop-intersection-is-empty}
    The intersection of $\mathcal E_{\operatorname{OGP}},\mathcal E_{\operatorname{suc}},\mathcal E_{\operatorname{stable}}$ and $\mathcal E_{\operatorname{ends}}$ is empty.
\end{proposition}
\begin{proof}
    We denote $O_k = \operatorname{overlap}\big(\mathcal{A}(G^0, \Gs), \mathcal{A}(G^{k}, \Gs)\big)$ for $0 \leq k \leq N$. Based on $\mathcal{E}_{\operatorname{stable}}$ and the triangle inequality, we obtain that $|O_{k+1} - O_k| \leq 2\eta n$ holds for any $0 \leq k \leq N-1$. Considering that $O_0 = n$ and $O_N \leq (\rho - \eta)n$ based on $\mathcal{E}_{\operatorname{ends}}$, we can conclude that there exists $1 \leq k \leq N-1$ such that $O_k \in \big((\rho - \eta)n, (\rho + \eta)n\big)$. Assuming that $\mathcal{E}_{\operatorname{suc}}$ also holds, we have $\mathcal{A}(G^0, \Gs) \in \operatorname{S}_{\beta_0}(G^0, \Gs)$ and $\mathcal{A}(G^k, \Gs) \in \operatorname{S}_{\beta_0}(G^k, \Gs)$. However, this contradicts with $\mathcal{E}_{\operatorname{OGP}}$, which means the intersection of the four events is empty.
\end{proof}

We are now ready to prove the main result of this section.

\begin{proof}[Proof of Theorem~\ref{thm-stable-hardness}]
   
    We will argue by contradiction. Assume that for some $\mathcal A$ it holds $$\Pb[\mathcal{A}(G, \Gs) \in \operatorname{S}_{\beta_0}(G, \Gs)] \geq 1 - 1/n^2\,.$$ {Then it holds $\Pb[\mathcal{E}_{\operatorname{suc}}] \geq 1/2$} from a union bound. Similarly, from Proposition~\ref{prop-OGP} and the definition of almost $(1, \eta n)$-stable algorithms, we can conclude that $\Pb[\mathcal{E}_{\operatorname{OGP}}] = 1 - o(1)$ and $\Pb[\mathcal{E}_{\operatorname{stable}}] = 1 - o(1)$.
    
Furthermore, we note that $(G^0, \Gs)$ and $(G^N, \Gs)$ form a pair of $(2, 0)$-correlated instance. Therefore, by utilizing Lemma~\ref{lem-good-event-on-G}, Proposition~\ref{prop-two-point-est-correlated}, and a similar argument as in the proof of Proposition~\ref{prop-OGP}, we get the probability $\Pb[\mathcal{E}_{\operatorname{ends}}^c]$ is bounded by
    \begin{align*}
    &\ \Pb[\Gs\text{ is not admissible}]+\Pb[\mathcal E_{\operatorname{ends}}^c\mid \Gs\text{ is admissible}]\\
    \le&\ o(1)+\exp\left([2-\rho_0+\eta-2\beta_0^2+o(1)]n\log n\right)=o(1)\,.
    \end{align*}
    This shows $\Pb[\mathcal E_{\operatorname{ends}}]=1-o(1)$ and thus 
    \[
    \Pb[\mathcal E_{\operatorname{OGP}}\cap \mathcal E_{\operatorname{suc}}\cap \mathcal E_{\operatorname{stable}}\cap\mathcal E_{\operatorname{ends}}]\ge 1/2-o(1)-o(1)-o(1)>0\,,
    \]
    contradicting Proposition~\ref{prop-intersection-is-empty}. This concludes the desired result.
\end{proof}

\subsection{Failure of online algorithms via branching-OGP}\label{subsec-branching-OGP}

In this subsection, we establish Theorem~\ref{thm-algo}-(ii), the hardness result for online algorithms. The proof relies on a technique known as the branching-OGP and is carried out in a similar spirit to the previous subsection. At a high level, we first argue that a particular forbidden structure is highly unlikely to occur. Then, assuming the existence of an online algorithm capable of finding near-optimal solutions, we leverage this algorithm to construct the forbidden structure, leading to a contradiction. This establishes that any online algorithm must be powerless.

\subsubsection{The branching overlap-gap property}

We begin with a brief high-level overview of the key ideas in the branching-OGP framework. First, and most importantly, while the $2$-OGP applies to a \emph{pair} of correlated instances, the branching-OGP generalizes this setup by considering a set of correlated instances indexed by the leaves of a large \emph{branching tree}, and they way the correlated to each other is captured by the positions of their corresponding leaves (see Definition~\ref{def-correlated-instance} below for details). Additionally, akin to the $2$-OGP, the branching-OGP forbidden structure comprises near-optimal solutions on these correlated instances under certain overlap constraints, while in our setting the constraints are tailored to address online algorithms. Finally, using a Jensen’s inequality trick developed in \cite{HS21}, the branching-OGP provides a strong upper bound on the probability that online algorithms can find near-optimal solutions.
% This is because given an online algorithm capable for finding near-optimal solutions, one can obtain a product lower bound for the probability of the existence of the forbidden structure, instead of a trivial lower bound given by the union bound.  }

%which generalizes the $2$-OGP discussed in the previous subsection. The key generalization is that the branching-OGP incorporates more intricate forbidden structures. On the one hand, while the $2$-OGP applies to a \emph{pair} of correlated instances (see Definition~\ref{def-correlated-instance}), the branching-OGP considers a set of correlated instances indexed by the leaves of a large \emph{ultrametric tree}. The correlation between any two instances is determined by the locations of their corresponding leaves in the ultrametric tree. On the other hand, akin to the $2$-OGP, the branching-OGP forbidden structure consists of near-optimal solutions over the correlated instances with certain overlap constraints. However, the overlap constraints are modified to address online algorithms. Crucially, this modification enables precise estimation of the enumeration of the combinatorial structure involved in the forbidden structure, allowing us to design a structure that remains forbidden down to the critical threshold $\beta_c = \sqrt{8/9}$.}

Let us elaborate further on the ideas discussed above through the relatively simple one-layer branching-OGP, which also coincides with the multi-OGP found in the literature. Fix a large integer $m$ and a parameter $\alpha\in (0,1)$. We define a set of $(m,\alpha)$-correlated instances $G^1,\dots,G^m$ as follows: the entries $G_{i,j}^k,1\le i<j\le n,1\le k\le m$ satisfy
\[
G_{i,j}^1=\cdots=G_{i,j}^{m},j\le \lfloor\alpha n\rfloor\,, G_{i,j}^1,\dots,G_{i,j}^m,j>\lfloor \alpha n\rfloor\,,
\]
and are i.i.d. $\mathbf{B}(1,p)$ indicators. For an admissible graph $\mathsf G$ and $\beta\in (0,1)$, we define the forbidden structure as follows: there exist $\pi_1,\dots,\pi_m\in \operatorname{S}_n$ such that\\
\noindent (i) $\pi_1(i)=\cdots=\pi_m(i)$ for all $1\le i\le \lfloor \alpha n\rfloor$;\\
\noindent (ii) $\pi_k\in \mathcal S_\beta(G^k,\mathsf G)$ for all $1\le k\le m$.

We assume that for some appropriate choices of $m = O(1)$, $\alpha$, and $\beta$, the probability of the above forbidden structure occurring is at most $p_{\operatorname{forbid}} = o(1)$ for any admissible $\mathsf{G}$. We now explain how this implies the failure of online algorithms. 

Suppose that for some admissible $\Gs$, there exists an online algorithm $\mathcal{A}$ such that 
$$
\mathbb{P}[\mathcal{A}(G, \mathsf{G}) \in \mathcal{S}_{\beta}(G, \mathsf{G})] \geq p_{\operatorname{suc}}\,,
$$
where the probability is taken over $G \sim \mathbf{G}(n, p)$ and the internal randomness of $\mathcal{A}$.
We generate $(m, \alpha)$-correlated instances $G^1, \dots, G^m$ and run $\mathcal{A}$ on $(G^1, \mathsf{G}), \dots, (G^m, \mathsf{G})$ simultaneously, using the same source of internal randomness for $\mathcal{A}$. Let $\pi_1, \dots, \pi_m$ represent the outputs of $\mathcal{A}(G^1, \mathsf{G}), \dots, \mathcal{A}(G^m, \mathsf{G})$. Since $G^1, \dots, G^m$ agree on the first $\lfloor \alpha n \rfloor$ vertices, the rules of the online algorithm imply that $\pi_1(i) = \cdots = \pi_m(i)$ for all $1 \leq i \leq \lfloor \alpha n \rfloor$. Hence, $(\pi_1, \dots, \pi_m)$ always satisfies item (i) of the forbidden structure.

Meanwhile, we claim that $\pi_k\in \mathcal S_\beta(G^k,\Gs)$ happens simultaneously for $1\le k\le m$ with probability at least $p_{\operatorname{suc}}^m$. To see this, denoting $\mathcal F$ as the $\sigma$-field generated by $G_{i,j}^1=\cdots=G_{i,j}^m,1\le i<j\le \lfloor \alpha n\rfloor$, we have 
\begin{align*}
    \mathbb{P}[\pi_k\in \mathcal S_{\beta}(G^k,\mathsf G),\forall 1\le k\le m]=&\ \mathbb{E}\big[\mathbb{P}[\pi_k\in \mathcal S_{\beta}(G^k,\mathsf G),\forall 1\le k\le m\mid \mathcal F]\big]\\
    (\text{by independence and symmetry})=&\ \mathbb{E}\big[\mathbb{P}[\pi_1\in \mathcal S_\beta(G^1,\mathsf G)\mid \mathcal F]^m\big]\\
    (\text{by Jensen's inequality})\ge &\ \big(\mathbb{E}\big[\mathbb{P}[\pi_1\in \mathcal S_\beta(G^1,\mathsf G)\mid \mathcal F]\big]\big)^m\\
    (\text{by the iterative expectation theorem})=&\ \big(\mathbb{P}[\pi_1\in \mathcal S_\beta(G^1,\mathsf G)]\big)^m\ge p_{\operatorname{suc}}^m\,,
\end{align*}
as claimed.
Consequently, we obtain 
\[
p_{\operatorname{suc}}^m\le \mathbb{P}[\text{forbidden structure exists}]\le p_{\operatorname{forbid}}\Rightarrow p_{\operatorname{suc}}\le p_{\operatorname{forbid}}^{1/m}=o(1)\,.
\]
This completes the derivation of online hardness from the one-layer branching OGP.

In what follows, we will extend the above strategy to the multi-layer tree case. By leveraging the multi-layer branching trees, we will be able to show that for any $\beta>\beta_c$, over typical instances of $\mathsf G$, a certain forbidden structure of scale $m=O(1)$ regarding $\beta$-optimal solutions exists with probability at most $p_{\operatorname{forbid}}=\exp(-\Theta(n\log n))$. Combined with an iterative application of Jensen's inequality, we conclude as above that for typical $\Gs$, any online algorithm fails to find $\beta$-optimal solutions with probability at least $p_{\operatorname{forbid}}^{1/m}=\exp(-\Theta(n\log n))$. The details appear in the following two subsections.

\subsubsection{The forbidden structure}

Let $\varepsilon>0$ be fixed. The objective of this subsection is to design a specific forbidden structure and demonstrate its highly improbable existence. By the definition of the Riemann integral, we have that for a sequence $0=\alpha_0<\alpha_1<\cdots<\alpha_N=1$ with $\Delta=\max_{1\le i\le n}(\alpha_i-\alpha_{i-1})$, 
\[
\lim_{\Delta\to 0}\sum_{i=1}^N(\alpha_i-\alpha_{i-1})\sqrt{\alpha_i+\alpha_{i-1}}=\int_{0}^1\sqrt{2x}\operatorname{d} x=\frac{2\sqrt{2x^3}}{3}\Big|_0^1=\beta_c.
\]
Therefore, we can pick some integer $N$ together with $\overrightarrow\alpha=(\alpha_0,\dots,\alpha_N)$, $0=\alpha_0<\alpha_1<\cdots<\alpha_N=1$ such that
\begin{equation}\label{eq-choice-of-alpha}
\sum_{i=1}^N(\alpha_i-\alpha_{i-1})\sqrt{\alpha_i+\alpha_{i-1}}<\beta_c+\varepsilon/3\,.
\end{equation}
Denote $\frac{\beta_c+2\varepsilon/3}{\beta_c+\varepsilon/3}$ as $1+\delta$. Fix a large integer $D$ satisfying
\begin{equation}\label{eq-def-D}
D>\max_{1\le k\le N}\frac{\alpha_{k-1}}{2\delta(\alpha_k-\alpha_{k-1})}\,.
\end{equation}

Consider a $D$-regular tree $\mathcal{T}$ with $(N+1)$ generations. We denote the set of leaves in $\mathcal{T}$ by $\mathcal{L}$, which has size $L = |\mathcal{L}| = D^N$. For a vertex $v$ of $\mathcal{T}$, we use $|v|$ to denote its depth in $\mathcal{T}$. If $v \in \mathcal{L}$, we use $v(k)$ to denote the ancestor of $v$ with depth $k$. For two vertices $v, w \in \mathcal{L}$, we define $v \wedge w$ as the common ancestor of $v$ and $w$ with the largest depth.

We now define the $(\mathcal{T},\overrightarrow{\alpha})$-correlated instance as follows:
\begin{definition}[$(\mathcal{T},\overrightarrow{\alpha})$-correlated instance]
\label{def-correlated-instance}
 For each vertex $v$ of $\mathcal T$, we assign it with independent $\mathbf B(1,p)$ variables $\{E^v_{i,j}\}_{1\le i< j\le \alpha_{|v|}n}$.  A $(\mathcal{T},\overrightarrow{\alpha})$-\emph{correlated instance} is a set of graphs $\{G^v\}_{v\in \mathcal L}$ indexed by $\mathcal L$, where for each $v\in \mathcal L$, the adjacency matrix of $G^v$ is defined by
 \[
 G_{i,j}^v=E_{i,j}^{v(k)}\,,\text{ if }i<j\text{ and }\alpha_{k-1}n<j\le \alpha_k n\,.
 \]
\end{definition}
Now we are able to state the forbidden structure concerning $(\mathcal{T},\overrightarrow{\alpha})$-correlated instance.
\begin{definition}[Forbidden structure for correlated instance]
 For a $(\mathcal{T},\overrightarrow{\alpha})$-correlated instance $\{G^v\}_{v\in \mathcal L}$ as well as a graph $\Gs\in \mathfrak{G}_n$, we say a set of permutations $\{\pi_v\}_{v\in \mathcal L}$ has the forbidden structure, if
 \begin{enumerate}
  \item [(i)] For each $v\in \mathcal L$, $\pi_v\in \operatorname{S}_{\beta_c+\varepsilon}(G^v,\Gs)$.
  \item [(ii)] For each pair of $u,v\in \mathcal L$, it holds that $\pi_u(i)=\pi_v(i)$ for any $1\le i\le \alpha_{|u\wedge v|}n$.
 \end{enumerate}
\end{definition}
    In the next proposition, we will show that for any admissible graph $\Gs$, the forbidden structure exists with a probability no greater than $\exp\big(-\Omega(n\log n)\big)$. This result will serve as a crucial ingredient for deriving the hardness results for online algorithms.
\begin{proposition}\label{prop-forbidden-structure}
There exists $c>0$ such that for any admissible graph $\Gs$,
 \begin{equation}
  \Pb[\text{the forbidden structure exists}]\le \exp(-cn\log n)\,,
 \end{equation}
where the probability is taken over $(\mathcal{T},\overrightarrow{\alpha})$-correlated instance.
\end{proposition}
\begin{proof}
	Keep in mind that $\Gs$ is a deterministic graph which is admissible. Let us assume the existence of a forbidden structure $\{\pi_v\}_{v\in \mathcal L}$. We define a set of embeddings $\{\sigma_v\}_{v\in \mathcal T}$ indexed by the vertices of $\mathcal T$ as follows: for each vertex $v$ of $\mathcal T$, $\sigma_v$ is an embedding from $\{i:\alpha_{|v|-1}n<i\le \alpha_{|v|}n\}$ to $[n]$ with $\sigma_v(i)=\pi_{v'}(i)$ for each $i$, where $v'$ is any descendant of $v$ in $\mathcal L$. Note that condition (ii) guarantees the well-definedness of $\{\sigma_v\}_{v\in \mathcal{T}}$, and the mapping $\{\pi_v\}_{v\in \mathcal L}\mapsto \{\sigma_v\}_{v\in \mathcal T}$ is clearly injective. 
	
	For each $v\in \mathcal T$, denote $\gamma_v^*$ as the random variable such that
	\begin{equation}\label{eq-def-gamma}
	\sum_{\substack{i<j,\\1\le i\le \alpha_{|v|}n,\\\alpha_{|v|-1}n<j\le\alpha_{|v|}n}}G^v_{i,j}\Gs_{\sigma_v(i),\sigma_v(j)}=	\sum_{\substack{i<j,\\1\le i\le \alpha_{|v|}n,\\\alpha_{|v|-1}n<j\le \alpha_{|v|}n}}\Gs_{\sigma_v(i),\sigma_v(j)}p+\gamma^*_v D_{n,p}\,.
	\end{equation}
    For any $v\in \mathcal L$, summing \eqref{eq-def-gamma} along the path connecting from the root to $v$, we get
    \begin{equation}\label{eq-sum-gamma}
        \sum_{i<j}G^v_{i,j}\Gs_{\pi_v(i),\pi_v(j)}= |E(\Gs)|p+\sum_{k=0}^N\gamma^*_{v(k)}D_{n,p}\,.
    \end{equation}
    Combining with condition (i), we have that the left hand side of \eqref{eq-sum-gamma} is lower bounded by $\binom{n}{2}p^2+(\beta_c+\varepsilon)D_{n,p}$, which implies
	\[
\sum_{k=0}^N\gamma^*_{v(k)}\ge\beta_c+\varepsilon-\frac{|\binom{n}{2}p-|E(\Gs)||p}{D_{n,p}}\stackrel{\eqref{eq-fixed-graph-edge-concentration}}{\ge}  \beta_c+2\varepsilon/3,\ \forall\ v\in \mathcal L\,.
	\]
 Summing over the above inequality for all $v\in \mathcal L$, we get
	\[
	\sum_{k=0}^N\frac{1}{D^k}\sum_{|v|=k}\gamma^*_v\ge \beta_c+2\varepsilon/3\,.
	\]
	Combining with \eqref{eq-choice-of-alpha}, we see there must exists some $0\le k\le N$ such that
	\begin{equation}\label{eq-the-bad-event-E-k}
	\frac{1}{D^k}\sum_{|v|=k}\gamma^*_v\ge \frac{\beta_c+2\varepsilon/3}{\beta_c+\varepsilon/3}\cdot(\alpha_k-\alpha_{k-1})\sqrt{\alpha_k+\alpha_{k-1}}=(1+\delta)(\alpha_k-\alpha_{k-1})\sqrt{\alpha_k+\alpha_{k-1}}\,.
	\end{equation}
	We denote the corresponding events as $\mathcal E_k,k=0,1,\dots,N$, and thus conclude that
	\[
	\Pb[\text{forbidden structure exists}]\le \sum_{k=0}^N\Pb[\mathcal E_k]\,.
	\]
	%\paragraph{Upper bounds on $\mathcal E_k$}
	Now, we need to bound the probability of $\mathcal{E}_k$, and we will do this using the first moment method. Note that $\mathcal{E}_k$ only depends on those $\sigma_v$ with $|v|\leq k$. Hence, we can bound $\Pb[\mathcal{E}_k]$ by considering the number of possible realizations of $\{\sigma_v\}_{|v|\leq k}$ as well as the probability that \eqref{eq-the-bad-event-E-k} occurs for each fixed $\{\sigma\}_{|v|\leq k}$.
 It is not hard to see that the number of possible realizations of such $\{\sigma_v\}_{|v|\le k}$ is bounded by
	\[
 \exp\left(\sum_{i=1}^{k}D^{i-1}(\alpha_i-\alpha_{i-1})n\log n+O(n)\right)\,.
	\]
 In order to control the probability term, for each realization $\{\sigma_v\}_{|v|\leq k}$,
we note that by our constructions, the family of random variables 
	\[
X_v^\sigma\stackrel{\operatorname{def}}{=}\sum_{\substack{i<j,\\1\le i\le \alpha_{|v|}n,\\\alpha_{|v|-1}n<j\le \alpha_{|v|}n}}G_{i,j}^v\Gs_{\sigma_v(i),\sigma_v(j)},v\in \mathcal T
	\]
are mutually independent with each other (and so do the variables $\gamma_v^*,v\in \mathcal T$), while each $X_v^\sigma$ has the distribution of $\mathbf{B}(N_v^\sigma,p)$, where
	\[
N_v^\sigma\stackrel{\operatorname{def}}{=}\sum_{\substack{i<j,\\1\le i\le \alpha_{|v|}n,\\\alpha_{|v|-1}n<j\le\alpha_{|v|}n}}\Gs_{\sigma_v(i),\sigma_v(j)}\,.
	\]
	From \eqref{eq-control-of-E(H)} in admissibility we have for any $v\in \mathcal T$
	\[
N_v^\sigma=\frac{\alpha_{|v|}^2-\alpha_{|v|-1}^2}{2}n^2p+o(n^2p)\,.
	\]
	Hence, by \eqref{eq-chernoff-bound} we see for each vertex
     $v\in \mathcal T$ and any $\gamma$ such that $|\gamma|D_{n,p}\ll n^2p^2$, the probability $\Pb[\gamma_v^*\ge \gamma]=\Pb[X_v^\sigma\ge N_v^\sigma p+\gamma D_{n,p}]$ is bounded by
	\begin{align*}
\exp\left(-\frac{(\gamma\vee 0)^2 n^3p^2\log n}{2(N_v^\sigma p+(\gamma\vee 0)D_{n,p})}\right)
= \exp\left(-\frac{(\gamma\vee 0)^2}{\alpha_{|v|}^2-\alpha_{|v|-1}^2}n\log n+o(n\log n)\right)\,.
	\end{align*}

Write $M=\sqrt{D^{k-1}(1+\delta)^2(\alpha_k-\alpha_{k-1})}$, and we denote 
 \[
 \Gamma_k=\left\{(\gamma_v)_{|v|\le k}:\sum_{|v|=k}\gamma_v\ge (1+\delta)(\alpha_k-\alpha_{k-1})\sqrt{\alpha_k+\alpha_{k-1}},\gamma_v\le M,\forall |v|\le k\right\}\,.
 \]
Firstly, based on the previous estimation, we obtain that the probability that $X_v^\sigma \geq N_v^\sigma + MD_{n,p}$ for some $v\in \mathcal{T}$ is at most $\exp\big(-M^2n\log n+o(n\log n)\big)$. Therefore, by considering the two cases that whether there exists $v$ with $|v|\le k$ such that $\gamma_v^*\ge M$, we conclude that the probability term can be bounded by $\exp\big(-M^2n\log n+o(n\log n)\big)$ plus
$$\exp\big(O(\log n)\big)\times\sup_{(\gamma_v)_{|v|\le k}\in \Gamma_k}\Pb[\gamma_v^*\ge \gamma_v,\forall v \text{ s.t. } |v|\le k]\,,$$
 where the first factor represents the number of possible realizations of $\gamma^*$ in $\Gamma_k$. As a result of independence, the supremum of the probability term given above is bounded by
	\begin{align*}
	&\ \sup_{(\gamma_v)_{|v|\le k}\in \Gamma_k}\exp\left(-\sum_{i=1}^{k}\sum_{|v|=i}\frac{(\gamma_v\vee 0)^2}{\alpha_i^2-\alpha_{i-1}^2}n\log n+o(n\log n)\right)\\
 \le&\ \sup_{(\gamma_v)_{|v|\le k}\in \Gamma_k}\exp\left(-\sum_{i=1}^kD^{i-1}\frac{(\sum_{|v|=i}\gamma_v\vee 0)^2}{\alpha_i^2-\alpha_{i-1}^2}n\log n+o(n\log n)\right)\\
 \le &\sup_{(\gamma_v)_{|v|\le k}\in \Gamma_k}\exp\left(-D^{k-1}\frac{(\sum_{|v|=k}\gamma_v\vee 0)^2}{\alpha_k^2-\alpha_{k-1}^2}n\log n+o(n\log n)\right)\\
 \le &\ \exp\left(-D^{k-1}(1+\delta)^2(\alpha_k-\alpha_{k-1})n\log n+o(n\log n)\right)\,.
	\end{align*}
	Here, the first inequality follows from Cauchy-Schwartz inequality, and the last inequality follows from the definition of $\Gamma_k$. With these estimates in hand, we obtain that $\Pb[\mathcal{E}_k]$ is bounded by
	\begin{equation}
		\begin{aligned}		&\exp\left(\sum_{i=1}^kD^{i-1}(\alpha_i-\alpha_{i-1})n\log n+o(n\log n)\right)\times \\
  &\ \left[\exp\left(-M^2n\log n+o(n\log n)\right)+\exp\left(-D^{k-1}(1+\delta)^2(\alpha_k-\alpha_{k-1})n\log n+o(n\log n\right)\right]\\
					\le&\ 2\exp\left([D^{k-2}\alpha_{k-1}-2\delta D^{k-1}(\alpha_k-\alpha_{k-1})]n\log n+o(n\log n)\right)=\exp(-cn\log n)\,,
		\end{aligned}
  \notag
	\end{equation}
 for some constant $c>0$ (which follows from the choice of $D$ in \eqref{eq-def-D}), as desired. 
\end{proof}

\subsubsection{Computational hardness of online algorithms}

Let $\mathcal A$ be an arbitrary online algorithm. According to Definition~\ref{def-online-algo}, we can assume that there exist probability spaces $\Omega_1, \dots, \Omega_n$ and deterministic functions $$f_k:\{G_{i,j}\}_{1\le i<j\le k}\times \{\Gs_{i,j}\}_{1\le i<j\le n}\times\Omega_k\to [n]\,,$$ such that $\mathcal A$ operates as follows: for each $1\le k\le n$, assuming that $\pi^*(1),\dots,\pi^*(k-1)$ are determined, $\mathcal A$ samples $\omega_k\in \Omega_k$ according to some probability measure $\Pb_k$ determined by $\{G_{ij}\}_{1\le i<j\le k},\{\Gs_{ij}\}_{1\le i<j\le n}$ and $\pi^*(1),\dots,\pi^*(k-1)$, and sets $$\pi^*(k)=f_k(\{G_{i,j}\}_{1\le i<j\le k},\{\Gs_{i,j}\}_{1\le i<j\le n},\omega_k)\,.$$ 
The mechanism should also ensure that the final output $\pi^*$ is always a permutation in $\operatorname{S}_n$.

We fix an arbitrary admissible graph $\Gs$. For a $(\mathcal{T},\overrightarrow{\alpha})$-correlated instance of graphs $\{G^v\}_{v\in \mathcal L}$ as defined in Definition~\ref{def-correlated-instance}, we simultaneously run $\Ac$  on the pairs  $(G^v,\Gs)$ for all $v\in \mathcal L$ with outputs $\{\pi_v^*\}_{v\in \mathcal L}$. The sampling procedure follows the following rule: for any two vertices $u,v\in \mathcal L$, the internal randomness $\omega_k$ used for determining $\pi^*_u(k)$ and $\pi_v^*(k)$ are sampled identically if $k\le \alpha_{|u\wedge v|}n$, while they are independently sampled if $k>\alpha_{|u\wedge v|}n$. It is important to note that due to the construction of $(\mathcal{T},\overrightarrow{\alpha})$-correlated instance and the definition of online algorithms, one can prove by induction that the law $\Pb_k$ of $\pi_u^*(k)$ and $\pi_v^*(k)$ are the same as long as $1\le k\le \alpha_{|u\wedge v|}n$. This demonstrates that the sampling rule mentioned above is self-consistent, and the running procedure is valid.

We have the following lower bound for the probability of $\pi_v^*\in\operatorname{S}_{\beta_c+\varepsilon}(G,\Gs)$ for all $v\in \mathcal L$.

\begin{proposition}
	Fix an admissible graph $\Gs$. Assume that the output $\pi^*$ of $\Ac$ satisfies $$\Pb[\pi^*\in \operatorname{S}_{\beta_c+\varepsilon}(G,\Gs)]=\mathbb{E}_{G\sim \mathbf G(n,p)}\mathbb Q\left[\pi^*\in \operatorname{S}_{\beta_c+\varepsilon}(G,\Gs)\right]\ge p_{\operatorname{suc}}\,,$$ where $\mathbb Q$ is the internal randomness of $\mathcal A$. Then through the aforementioned procedure, we have 
	\begin{equation}\label{eq-jensen-lower-bound}
		\Pb[\pi_v^*\in \operatorname{S}_{\beta_c+\varepsilon}(G^v,\Gs),\forall v\in \mathcal L]\ge p_{\operatorname{suc}}^L\,,
	\end{equation}
\end{proposition}
\begin{proof}
	Let $S_v$ denote the event that $\pi_v^*\in \operatorname{S}_{\beta_c+\varepsilon}(G^v,\Gs)$. For $i=0,1,\dots,N$, we define $\mathcal F_i$ as the $\sigma$-field generated by $\{G^v_{i,j}\}_{1\le i<j\le \alpha_in},\forall v\in \mathcal L$ and $\pi^*(k),1\le k\le \alpha_in$. For each $v\in \mathcal T$, we use $\mathcal L(v)$ to denote the offspring of $v$ in $\mathcal L$, and $\mathcal{S}(v)$ represents the event that $\pi_u^*\in \operatorname{S}_{\beta_c+\varepsilon}(G^u,\Gs)$ for each $u\in \mathcal L(v)$. Note that we are interested in the event that $\Pb[\mathcal S(v_o)]$ for the root $v_o$, and we will show by induction that  
	\begin{equation}\label{eq-iterative-jensen}
	\Pb[\mathcal S(v_o)]\ge \prod_{|v|=i}\Pb[\mathcal S(v)]\,,\forall 0\le i\le N\,.
	\end{equation}
 
The case $i=0$ trivially holds. Assuming that \eqref{eq-iterative-jensen} holds for some $0\le i\le N-1$. For each $v$ with $|v|=i$, we denote its $D$ descendants as $v_1,\dots,v_D$, then it follows that $\mathcal{S}(v)=\bigcap_{k=1}^D\mathcal{S}(v_k)$, and the events $\mathcal{S}(v_k),\ k=1,2,\dots,D$ are conditionally independent given $\mathcal{F}_{i}$. By conditioning on $\mathcal{F}_{i}$ and utilizing the iterated expectation formula, we deduce that
\begin{align*}
\Pb[\mathcal S(v)]=&\ \mathbb{E}\Big[\prod_{k=1}^D\Pb[\mathcal S(v_k)\mid \mathcal F_i]\Big]=\mathbb{E}\Big[\Pb[\mathcal S(v_1)\mid \mathcal F_i]^D\Big]\\
\ge&\ \big(\mathbb{E}\big[\Pb[\mathcal S(v_1)\mid \mathcal F_i]\big]\Big)^D=\big(\Pb[\mathcal S(v_1)]\big)^D=\prod_{k=1}^D\Pb[\mathcal S(v_k)]\,,
\end{align*}
where we employed the symmetry between $v_1,\dots,v_D$ twice, and the inequality follows directly from Jensen's inequality. Consequently, by considering the product of all $v$ with $|v|=i$, we observe that \eqref{eq-iterative-jensen} holds for $i+1$. This establishes the validity of \eqref{eq-iterative-jensen} for each $0\le i\le N$. In particular, for $i=N$, we can deduce that
\[
\Pb[\mathcal S(v_o)]\ge \prod_{|v|=N}\Pb[\mathcal S(v)]=(\Pb[\pi^*\in \operatorname{S}_{\beta+c}(G,\Gs)]\big)^{L}\ge p_{\operatorname{suc}}^L\,,
\]
completing the proof.
\end{proof}
Note that if $\pi_v^*\in \operatorname{S}_{\beta_c+\varepsilon}(G^v,\Gs)$ for each $v\in \mathcal L$, then $(\pi_v^*)_{v\in \mathcal L}$ constitutes the forbidden structure. By combining this observation with Proposition~\ref{prop-forbidden-structure}, we conclude that whenever $\Gs$ is admissible, the following holds for any $\mathcal A\in \operatorname{OGAA}$:
\[
\Pb[\mathcal A(G,\Gs)\in \operatorname{S}_{\beta_c+\varepsilon}(G,\Gs)]\le \exp(-cn\log n/L)=\exp(-c_0n\log n)\,.
\]
So, by the Markov inequality, we see 
\[
\Pb\big[G:\mathbb{Q}[\mathcal A(G,\Gs)\in \operatorname{S}_{\beta_c+\varepsilon}(G,\Gs)\ge \exp(-c_0n\log n/2)]\big]\le \exp(-c_0n\log n/2)\,.
\]
Note that the above statement holds for any admissible graph $\Gs$ and any online algorithm $\mathcal A\in \operatorname{OGAA}$. Since $\Gs\sim \mathbf G(n,p)$ is admissible with high probability by Lemma~\ref{lem-good-event-on-G}, the proof of the second item in Theorem~\ref{thm-algo} is now completed.

We conclude this section by highlighting that our method for establishing the hardness result of online algorithms can be extended to demonstrate limitations of a more general class of algorithms, which we refer as the \emph{greedy local iterative algorithms}, as defined below.
\begin{definition}[(Greedy) Local Iterative Algorithms]\label{def-GLIA}
We say an algorithm $\mathcal A\in \operatorname{GAA}$ is a \emph{local iterative algorithm}, if $\mathcal A$ operates as follows: 
\begin{itemize}
	\item Initially, $\mathcal A$ sets $\operatorname{M}_0=\pi^*(\operatorname{M}_0)=\emptyset$.
	\item For any $s\ge 1$, whenever $\operatorname{M}_{s-1}\neq [n]$, there is a probability measure $\Pb_s$ determined by the random variables $\{G_{i,j},\Gs_{i,j}\}_{(i,j)\in \operatorname{U}}$ as well as the values $\pi^*(i),i\in \operatorname{M}_{s-1}$ (via some given rule), from which $\mathcal A$ samples a random set $\Delta_s\subset [n]\setminus \operatorname{M}_{s-1}$ as well as values $\pi^*(i)\in [n]\setminus \pi^*(\operatorname{M}_{s-1})$, $i\in \Delta_s$. 
	\item It should also satisfies that almost surely, $|\Delta_s|=o(n)$ uniformly for all $s$, and $\pi^*(i),i\in \Delta_s$ are distinct. 
	\item $\mathcal A$ sets $\operatorname{M}_s=\operatorname{M}_{s-1}\cup \Delta_s$ and $\pi^*(\operatorname{M}_s)=\pi^*(\operatorname{M}_{s-1})\cup\{\pi^*(i):i\in \Delta_s\}$, then proceeds with the next $s$.
	\item Until $\operatorname{M}_s=[n]$ for some $s$, $\mathcal A$ outputs $\pi^*=(\pi^*(1),\dots,\pi^*(n))$ and terminates.
\end{itemize}%set  For any $s\ge 1$, if $\operatorname{M}_{s-1}=[n]$, then $\mathcal A$ outputs $\pi^*=(\pi^*(1),\dots,\pi^*(n))$. Else, there is a probability measure $\Pb_s$ determined by the random variables $\{G_{i,j},\Gs_{i,j}\}_{(i,j)\in \operatorname{U}}$ as well as the values $\pi^*(i),i\in \operatorname{M}_0$ (via some given rule), and $\mathcal A$ samples according to $\Pb_s$ a random set $\Delta_s\subset [n]\setminus \operatorname{M}_s$ with $|\Delta_s|=o(n)$, as well as distinct values $\pi^*(i)\in [n]\setminus \pi^*(\operatorname{M}_0)$, $i\in \Delta_s$, then set $$\operatorname{M}_s=\operatorname{M}_{s-1}\cup \Delta_s,\pi^*(\operatorname{M}_s)=\pi^*(\operatorname{M}_{s-1})\cup\{\pi^*(i):i\in \Delta_s\}\,.$$
Furthermore, we say a local iterative algorithm $\mathcal A$ is \emph{greedy}, if it satisfies the following property: for independent \ER graphs $G,\Gs\sim \mathbf G(n,p)$ and any $s\ge 1$, conditioned on any realization of $\Delta_t,1\le t\le s-1$ as well as the values $\pi^*(i),i\in \operatorname{M}_{s-1}$, the unmatched parts of these two graphs are stochastically dominated by two independent \ER graphs on the corresponding parts. In other words, under any such conditioning, it holds that the joint distribution of the random variables
$$
\{G_{i,j}\}_{i,j\text{ not all in }\operatorname{M}_{s-1}},\{\Gs_{i',j'}\}_{i',j'\text{ not all in }\pi^*(\operatorname{M}_{s-1})}
$$
is dominated by the product measure of $\mathbf B(1,p)$ variables. 
\end{definition}
Note that the algorithms presented in \cite{DDG22} are all greedy local iterative algorithms (see Lemma~3.11 therein). One can show that any greedy local iterative algorithm is unlikely to find solutions beyond the threshold $\beta_c$ under typical instances of $(G,\Gs)\sim \mathbf G(n,p)^{\otimes 2}$, as stated in the next proposition.
\begin{proposition}\label{prop-extended-hardness}
	For any $\varepsilon>0$, there exists $c=c(\varepsilon)>0$, such that for any greedy local iterative algorithm $\mathcal A\in \operatorname{GAA}$, it holds
	\[
	\Pb\big[(G,\Gs):\mathbb Q[\mathcal A(G,\Gs)\in \operatorname{S}_{\beta_c+\varepsilon}(G,\Gs)]\ge \exp(-cn\log n)\big]=o(1)\,,
	\]
	where $\Pb$ is taken over $(G,\Gs)\sim \mathbf G(n,p)^{\otimes 2}$ and $\mathbb Q$ is the internal randomness of $\mathcal A$.
\end{proposition}
The result follows in an essentially same manner as the proof of Theorem~\ref{thm-algo} (ii), where the greedy property in these algorithms is dedicated to an analogue of Proposition~\ref{prop-forbidden-structure}. We omit the details.

\begin{appendix}
\section{Tail estimates}
In this section, we provide various tail estimates for binomial variables that are useful for the proof. These estimates are categorized into two types based on the range of the parameter $p$, and the two regimes correspond to Poisson approximation and normal approximation, respectively. For simplicity, we denote the binomial distribution with parameters $N$ and $p$ as $\mathbf B(N,p)$, which also represents a random variable following this distribution. %In the subsequent lemmas, we will frequently use the following easy-checked fact about binomial numbers:
	%\begin{equation}\label{eq:est_binom_number}
%		\frac{n^k}{k!}\exp\left(-\frac{k^2}{n}\right)\leq\binom{n}{k}\leq \frac{n^k}{k!}\,.
%	\end{equation}
%\begin{lemma}[Estimate for $\binom{n}{k}$]
%\label{lem:est_binom_number}
%	We have the following one point estimate for $\binom{n}{k}$,
%	\begin{equation}
%		\frac{n^k}{k!}\exp\left(-\frac{k^2}{n}\right)\leq\binom{n}{k}\leq \frac{n^k}{k!}.\notag
%	\end{equation}
%\end{lemma}
\begin{lemma}
\label{lem:poisson_tail_approx}
When $p=n^{-1/2+o(1)}$ and $p\ll p_c$, write $M_0=\log n/\log\big(\log n/np^2\big)$, then for any integer $N\asymp np$ and any constant $\alpha>0$,
	\begin{equation}
\Pb\big[\mathbf B(N,p)\geq\alpha M_0\big]= n^{-\alpha+o(1)}.\notag
	\end{equation}
\end{lemma} 
\begin{proof}%[Proof of lemma \ref{lem:poisson_tail_approx}]
	For the upper bound, by applying \eqref{eq-chernoff-bound-upper}, we have the deviation probability is bounded by
 \[
 \exp\left(-Np\left[\frac{\alpha M_0}{Np}\log\left(\frac{\alpha M_0}{Np}\right)-\frac{\alpha M_0}{Np}+1\right]\right)\,.
 \]
 Now by the assumption we have $M_0\log\left(\alpha M_0/Np\right)=\log n+o(\log n)$ and $M_0=o(\log n)$, thus the expression above reduces to $\exp\big(-\alpha \log n+o(\log n)\big)=n^{-\alpha+o(1)}$, as desired.
 
	For the lower bound, it suffices to derive a one-point estimation. Write $k=\lceil \alpha M_0\rceil$ for short. Note that by our assumption $M_0\gg 1$ so $k=\big(\alpha+o(1)\big) M_0$. As a result, we have
	\begin{equation}
		\begin{split}
  \Pb\left[\mathbf B(N,p)=k\right]&=\binom{N}{k}p^k(1-p)^{N-k}\ge \frac{(Np)^k}{k!}\exp\left(-\frac{k^2}{N}\right)(1-p)^{N}\\
			&=\big(1+o(1)\big)\frac{(Np)^k}{k!}\exp(-Np)\\
			&=\big(1+o(1)\big)\exp\left(k\log(Np)-k\log k-k+\log(\sqrt{2\pi k})-Np\right)\\
			&=\big(1+o(1)\big)\exp\left(-k\log\left(\frac{k}{Np}\right)+o(\log n)\right)\\
			&=\big(1+o(1)\big)\exp\big(-\alpha\log n+o(\log n)\big)=n^{-\alpha+o(1)}\,,
		\end{split}\notag
	\end{equation}
	which completes the proof.
\end{proof}

\begin{lemma}
\label{lem:Upper bound for binomial tail}
When $p$ is in the dense regime,  for any integer $N=\Omega(np)$ and any constant $\alpha>0$, it holds that
 \begin{equation}
		\Pb\left[\mathbf B(N,p)\ge Np+\sqrt{2\alpha Np\log n}\right]=n^{-\alpha +o(1)}\,.\notag
	\end{equation}
 %hold under the assumption that either $p\gg p_c, N\ge \eta np/2$ or $np^3\ge (\log n)^2, N\ge \eta np^2/2$. 
\end{lemma}
\begin{proof}%[Proof of lemma \ref{lem:Upper bound for binomial tail}]

Under the given assumptions, we have $Np\gg \sqrt{Np\log n}$, which allows us to obtain the upper bound using \eqref{eq-chernoff-bound}. On the other hand, we recall a well-known result presented in \cite[Theorem 2.1]{Slud77}. This theorem states that for a binomial variable $\mathbf B(N,p)$ with $p\le 1/4$ and any $k$ with $Np\le k\le N(1-p)$, it always holds that
\[
\Pb[\mathbf B(N,p)\ge k]\ge 1-\Phi\left(\frac{k-Np}{\sqrt{Np(1-p)}}\right)\,,
\]
where $\Phi$ is the distribution function of a standard normal variable.
Since $p=o(1)$ based on our assumption, the lower bound can be derived from this result and classic estimations for normal tails readily.
	%By Theorem $1$ in \cite{AG89} we get
	%\begin{equation}
	%\label{eq:entropy binom}
	%\Pb\left[\mathbf B(N,p)\geq aN\right]\leq \exp\left(-N\mathcal H(a\,\vert\, p)\right).
	%\end{equation}
	%where $\mathcal H(a\,\vert\, p)=a\log(a/p)+(1-a)\log\left((1-a)/(1-p)\right)$. Therefore we only need to estimate the value of $\mathcal H\left(p+\sqrt{2\alpha p(1-p)\log n/N}\,\vert\,p\right)$. By Taylor expansion, this equals to
	%\begin{equation*}
	%\label{eq:est_entropy}
	%	\begin{split}
	%		&\ \left(p+\sqrt{2\alpha \frac{p(1-p)\log n}{N}}\right)\log\left(1+\sqrt{2\alpha \frac{(1-p)\log n}{Np}}\right)\\
	%		&\qquad+\left(1-p-\sqrt{2\alpha \frac{p(1-p)\log n}{N}}\right)\log\left(1-\sqrt{2\alpha\frac{p\log n}{N(1-p)}}\right)\\
	%		=&\ \left(p+\sqrt{2\alpha \frac{p(1-p)\log n}{N}}\right)\left(\sqrt{\frac{2\alpha (1-p)\log n}{Np}}-\frac{\alpha(1-p)\log n}{Np}+o\left(\frac{\log n}{Np}\right)\right)\\
	%		&\qquad +\left(1-p-\sqrt{2\alpha \frac{p(1-p)\log n}{N}}\right)\left(-\sqrt{\frac{2\alpha p\log n}{N(1-p)}}-\frac{\alpha p\log n}{N(1-p)}+o\left(\frac{p\log n}{N(1-p)}\right)\right)\\
	%		=&\ \frac{\alpha\log n}{N}\,,
	%	\end{split}
	%\end{equation*}
	%which proves the upper tail bounds. The proof for lower tail bound is similar.
\end{proof}

\section{Complimentary proof}
\subsection{Proof of Lemma~\ref{lem-good-event-on-G}}
\begin{proof}
It suffices to prove that all the items are true with probability $1-o(1/n^2)$. For the first one, we have $|E(G)|\sim \mathbf{B}(\binom{n}{2},p)$ and thus by \eqref{eq-chernoff-bound} we get
	\[
	\Pb\left[\Big||E(G)|-\binom{n}{2}p\Big|\ge 2\sqrt{n^2p\log n}\right]\le 2\exp\left(-\frac{4n^2p\log n}{n^2p+4\sqrt{n^2p\log n}}\right)=o(1/n^2)\,.
	\]
%This shows that \eqref{eq-fixed-graph-edge-concentration} happens with high probability.

To establish the second item, which requires \eqref{eq-control-of-E(H)} to be true for every graph induced subgraph $H$, we observe that for each induced subgraph $H$ with $k$ vertices, $|E(H)|\sim \mathbf{B}(\binom{k}{2}, p)$. By applying \eqref{eq-chernoff-bound} again we obtain that for each induced subgraph $H$,
	\[
	\Pb\left[\Big||E(H)|-\binom{k}{2}p\Big|\ge \frac{n^2p}{{\log n}^{1/4}}\right]\le2 \exp\left(-\frac{n^4p^2/\log n^{1/2}}{k^2p\vee (2n^2p/\log n^{1/4})}\right)\le\exp\left(-\frac{n^{3/2}}{\log n^{1/4}}\right)\,,
	\]
	where the last inequality follows from $p\gg p_c=\sqrt{\log n/n}$. Therefore, the result follows by applying a union bound since there are $2^n=\exp(O(n))$ induced subgraphs ${H}$ in total.

For the third item, we recall that $\operatorname{U}$ represents the set of unordered pairs. Given a fixed permutation $\pi\in \operatorname{S}_n$, we define the mapping $\Pi: \operatorname{U} \to \operatorname{U}$ as $\Pi(i,j) = (\pi(i),\pi(j))$. It is evident that $\Pi$ is a bijection on $\operatorname{U}$. Now consider a pair $(i,j)\in \operatorname{U}$ and let $X_{i,j}=G_{i,j}G_{\pi(i),\pi(j)}$. We will make use of the following observations: for each $(i,j)\in \operatorname{U}$,
\begin{enumerate}
\item[(i)] $X_{i,j}\sim \mathbf{B}(1,p)$ if $\Pi(i,j)=(i,j)$, and $X_{i,j}\sim\mathbf{B}(1,p^2)$ otherwise.
\item[(ii)] $X_{i,j}$ is independent with the variables $X_{k,l},{(k,l)\in \operatorname{U}\setminus \{\Pi(i,j),\Pi^{-1}(i,j)}\}$.
\end{enumerate}
Motivated by these observations, we can partition $\operatorname{U}$ into the disjoint union of four sets: $\operatorname{F}\sqcup \operatorname{A}\sqcup \operatorname{B}\sqcup \operatorname{C}$. Here, $\operatorname{F}$ represents the set of fixed points of $\Pi$, while $\operatorname{A}$, $\operatorname{B}$, and $\operatorname{C}$ are chosen in such a way that no pair $(i,j)\in \operatorname{U}$ exists where both $(i,j)$ and $\Pi(i,j)$ belong to the same set. The existence of such partitioning can be proven by partitioning each orbit in the cycle decomposition of $\Pi$ into three parts, ensuring that no two edges within the same part are connected.
%	We draw a directed graph (denoted as $G_{\pi}$) with those elements in \,$U$ as vertices,
%	in which each $(i,j)$ has a unique directed edge towards $f_{\pi }(i,j)$.
%	Therefore $G_{\pi}$ is a combination of some points and directed circles. We denote $X$ as the set of independent points in $G_{\pi}$.
%	Then, $$|X|=\binom{F(\pi)}{2}+T(\pi).$$ For the other vertex of $G_{\pi}$, we can assign each vertex one of the colors red, blue and yellow. We can prove that there is a coloring with the following property : for every $(i,j)$ , the color of it is different from the color of $f_{\pi}(i,j)$. Such coloring is called \emph{good coloring}.
%	In fact, we only need to give an \emph{good coloring} for each directed circle. 
%	If it's an even circle of length $2k$ ($k$ a positive integer), we give each vertex a label: $v_1,v_2,\cdots, v_{2k}$ that $v_{i+1}=f_{\pi}(v_i)$ for $1\le i \le 2k$ (here $v_{2k+1}=v_1$). Then we color $v_{2i-1}$ red and $v_{2i}$ blue for $1\le i \le k$. 
%	If it's an odd circle of length $2k+1$ ($k$ a positive integer), we give each vertex a label $v_{i}(1\le i\le 2k+1)$ like above. Then we color $v_{2i-1}$ red and $v_{2i}$ blue for $1\le i \le k$, and color $v_{2k+1}$ yellow.

Denote 
\[
\sigma(\star)=\sum_{(i,j)\in \star}X_{i,j},\ \forall \star\in \{\operatorname{F},\operatorname{A},\operatorname{B},\operatorname{C}\}\,,
	\]
	then $\operatorname{OL}(G,\pi)=\sum_{\star\in\{\operatorname{F},\operatorname{A},\operatorname{B},\operatorname{C}\}}\sigma(\star)$. Additionally, by combining the previous observations with the choices we made for $\operatorname{F}$, $\operatorname{A}$, $\operatorname{B}$, and $\operatorname{C}$, we obtain the following result:
 \[
\sigma(\operatorname{F})\sim\mathbf{B}(|\operatorname{F}|,p),\textbf{ and }\sigma(\star)\sim\mathbf{B}(|\star|,p^2),\forall\ \star\in \{\operatorname{A},\operatorname{B},\operatorname{C}\}\,.
 \]
%	With the definition of $X,R,B,Y$, we notice that $\{G_{ij}G_{\pi(i)\pi(j)}\}_{(i,j\in S)}$ is a sequence of independent random variables for all $S\in\{X,R,B,Y\}$. Thus we have
%	\[
%	\sigma(X)\sim\mathbf{B}(|X|,p), \, \sigma(R) \sim\mathbf{B}(|R|,p^2)
%	\]\\
%	\[
%    \sigma(B)\sim\mathbf{B}%(|B|,p^2), \,\sigma(Y)\sim\mathbf{B}(|Y|,p^2)
%	\]\\
%Note that $4\sqrt{n^3p\log n}\ge 2\sqrt{|\operatorname{F}|np\log n}+3\sqrt{2n^3p^2\log n}$. 
As a result, the probability that $|\operatorname{OL}(G,\pi)|$ deviates from its expectation by at least $2\sqrt{|\operatorname{F}|np\log n}+3\sqrt{2n^3p^2\log n}$ is bounded by
\begin{equation*}
    \begin{aligned}
    &\  \Pb\left[|\sigma(\operatorname{F})-\mathbb{E}\sigma(\operatorname{F})|\ge 2\sqrt{|\operatorname{F}|np\log n}\right]+\sum_{\star\in\{\operatorname{A},\operatorname{B},\operatorname{C}\}}\Pb\Big[|\sigma(\star)-\mathbb{E}\sigma(\star)|\ge \sqrt{2n^3p^2\log n}\Big]\\
\stackrel{\eqref{eq-chernoff-bound}}{\le}&\ \exp\left(-\frac{4|\operatorname{F}|np\log n}{2|\operatorname{F}|p+4\sqrt{|\operatorname{F}|np\log n}}\right)+\sum_{\star\in\{\operatorname{A},\operatorname{B},\operatorname{C}\}}\exp\left(-\frac{2n^3p^2\log n}{2|\star|p^2+2\sqrt{2n^3p^2\log n}}\right)\,.
    \end{aligned}
\end{equation*}
Note that it trivially holds $|\star|\le n^2/2$ for all $\star\in{\operatorname{A},\operatorname{B},\operatorname{C}}$. Consequently, the expression mentioned above is much smaller than $\exp(-n\log n)$. As a result, \eqref{eq-OL-concentration} holds for all permutation $\pi\in \operatorname{S}_n$ with probability of $1-o(1/n^2)$, as indicated by a union bound. 

Additionally, since $\mathbb E|\operatorname{OL}(G,\pi)|=|\operatorname{F}|(p-p^2)+\binom{n}{2}p^2\ge (|\operatorname{F}|p\vee n^2p^2)/3$, and it can be easily checked that under the assumption $p\gg p_c$, $\sqrt{n^3p^2\log n}\ll n^2p^2$ and $\sqrt{|\operatorname{F}|np\log n}\ll |\operatorname{F}|p\vee n^2p^2$. Therefore, \eqref{eq-OL-concentration} implies $|\operatorname{OL}(G,\pi)|/\mathbb{E}|\operatorname{OL}(G,\pi)|=1+o(1)$, which concludes the final statement in the third item. The proof is completed.
\end{proof}

\subsection{Proof of Proposition~\ref{prop-concentration-of-max-O-pi}}
 Here in this subsection we provide a complete proof of Proposition~\ref{prop-concentration-of-max-O-pi}. Let us begin by recalling some definitions and notations concerning to Talagrand's concentration inequality.
%related to the application of Talagrand's inequality as stated in Section～7.7 in \cite{AlonProbabilistic} :
\begin{definition}
	Let $\Omega = \prod_{i=1}^{N} \Omega_i$ be a product space. %where each $\Omega_i$ is a probability space associated to a $\mathbf{B}(1,p)$ random variable and $\Omega$ has the product measure. 
 For a function $h : \Omega \to \mathbb{R}$, we call $h$ to be Lipschitz, if $|h(x)- h(y)|\leq 1 $ whenever $x,y$ differ in at most one coordinate.
	Furthermore, let $f: \mathbb{R}_{\ge 0} \to \mathbb{R}_{\ge 0}$ be an increasing function, we say that $h$ is $f$-certifiable if, whenever $h(x) \geq s$, there exists $I \subset \{ 1,...,\binom{n}{2} \}$ with $|I| \leq f(s) $ so that all $y \in \Omega$  that agree with $x$ on the coordinates in $I$ have $h(y) \geq s$.
\end{definition}
A special case of Talagrand's concentration inequality can be stated as follow.

\begin{proposition}\label{prop-talagrand-concentration-ineq}
	For any product probability measure $\Pb$ on $\Omega$ and any Lipschitz function $h$ that is $f$-certifiable, let $X=h(\omega)$, we have for any $b,t>0$:
	\[
	\Pb [X \leq b - t\sqrt{f(b)}]\cdot\Pb\left[ X \geq b \right] \leq e^{-{t^2}/{4}}\,. 
    \]
\end{proposition}
The proof of this powerful inequality can be found in, e.g. \cite[Section~7.7]{Alon}. Now we are ready to give the proof of Proposition~\ref{prop-concentration-of-max-O-pi}. Recall the definition of admissibility in Definition~\ref{def-admissible}.
\begin{proof}[Proof of Proposition~\ref{prop-concentration-of-max-O-pi}]
Fix an admissible graph $G$. For our purposes, we consider $\Omega=\{0,1\}^{\binom{n}{2}}$ and $\Pb$ be the product measure of $\binom{n}{2}$ Bernoulli variables with parameter $p$. Denote $$X = h(\Gs)\triangleq \max\limits_{\pi\in \operatorname{S}_n}{\operatorname{O}}(\pi)\,,$$ where we view $\Gs$ as a vector in $\{0,1\}^{\binom{n}{2}}$. It is trivial to verify that $h$ is Lipschitz.

Furthermore, we claim that the function $h$ defined above is $f$-certifiable, where $f(s) = \lceil s \rceil$. To see this, suppose $h(x) \geq s$ for some $x = (\Gs_{i,j})_{(i,j)\in \operatorname{U}} \in \Omega$. Then, there exists a permutation $\pi^*$ and $\lceil s \rceil$ unordered pairs $(i_k,j_k)\in \operatorname{U}$, $1 \leq k \leq \lceil s \rceil$, such that $G_{i_k, j_k} = 1$ and $\Gs_{\pi^*(i_k),\pi^*(j_k)} = 1$ for each $k$. We can select $I$ as the set ${(\pi^*(i_k), \pi^*(j_k)), 1 \leq k \leq \lceil s \rceil}$, which contains $\lceil s \rceil$ edges. For any $\Gs'$ that includes all the edges in $I$, it holds that $h(\Gs') \geq \operatorname{O}(\pi^*) \geq s$.

As a result, Proposition~\ref{prop-talagrand-concentration-ineq} yields that for any $b, t\ge 0$, 
\begin{equation}\label{eq-talagrand-ineq}
\Pb[X\le b-t\sqrt{\lceil b\rceil}]\cdot\Pb[X\ge b]\le e^{-t^2/4}\,.
\end{equation}
%Then, we recall the Talagrand inequality, stated as Theorem 7.7.1 in \cite{AlonProbabilistic}:
%Thus we can apply the above theorem in our case with $b \triangleq Median(X)$ and $f(s)\triangleq s$, which yields:
%$$ \Pb \left(X \leq Median(X) - t\sqrt{Median(X)} \right)\leq 2e^{-\frac{t^2}{4}} , $$
%switching $t$ to $\frac{t}{\sqrt{Median(X)}}$, we have:
Taking $b=\operatorname{Median}(X)$ in \eqref{eq-talagrand-ineq} and by an easy transformation, we get for any $t\ge 0$,
$$ \Pb \left[X - \operatorname{Median}(X) \leq - t \right]\leq 2\exp\left(-\frac{t^2}{4 \lceil \operatorname{Median}(X)\rceil} \right)\,. $$
In addition, for any $r\ge 0$, taking $b=\operatorname{Median}(X)+r, t={r}/{\sqrt{b}}$ in \eqref{eq-talagrand-ineq} yields:
\begin{align*}
\Pb \left[X - \operatorname{Median}(X) \geq r \right]\leq&\ 2\exp\left(-\frac{r^2}{4 (\lceil \operatorname{Median}(X)+r\rceil)}\right)\\
\le&\ 2\exp\left(-\frac{r^2}{8\lceil \operatorname{Median}(X)\rceil}\right)+2\exp\left(-\frac{r^2}{8\lceil r\rceil}\right)\,. 
\end{align*}

From \eqref{eq-fixed-graph-edge-concentration} in admissibility and the proof of the upper bounds in Section~\ref{subsec-info-upper}, we can deduce that for any admissible graph $G$, the quantity $\operatorname{Median}(X)$ is asymptotically given by $\mathbb{E} \operatorname{O}(\pi) \sim n^2p^2/2$. In particular, we have $\lceil \operatorname{Median}(X) \rceil \leq n^2p^2$. Now, we claim that $|\mathbb{E}X - \operatorname{Median}(X)| = o(D_{n,p})$. To see this, we can use the aforementioned tail estimates, which imply that
\begin{equation}
\begin{aligned}
    &\ |\Eb X - \operatorname{Median}(X)|  \leq \Eb [| X -\operatorname{Median}(X)| ]\\
    =&\ \int_{0}^{+\infty} \Pb \left( | X - \operatorname{Median}(X)| \geq t \right) \operatorname{d} t\\
     \leq&\   \int_{0}^{+\infty}\left[2\exp\left(-\frac{t^2}{4n^2p^2}\right)+2\exp\left(-\frac{t^2}{8n^2p^2}\right)+2\exp\left(-\frac{t^2}{8\lceil t\rceil}\right)\right]\operatorname{d} t \\
     =&\ O(np)=o(D_{n,p})\,.
\end{aligned}
\end{equation}
As a result, we see for $n$ large enough,
\begin{equation*}
\begin{aligned}
&\ \Pb\left[|X-\Eb X|\geq \varepsilon D_{n,p} \right]\leq \Pb\left[|X-\operatorname{Median}(X)|\geq \varepsilon D_{n,p}/2 \right] \\
\le&\ 2\exp\left(-\frac{\varepsilon^2 D_{n,p}^2}{16n^2p^2}\right)+2\exp\left(-\frac{\varepsilon^2D_{n,p}^2}{32n^2p^2)}\right)+2\exp\left(-\frac{\varepsilon D_{n,p}}{16}\right)\\
=&\ \exp\big(-\Omega(n\log n)\big)\,.
\end{aligned}
\end{equation*}
Note that $X-\Eb X=\max_{\pi\in \operatorname{S}_n}\widetilde{\operatorname{O}}(\pi)-\mathbb E\max_{\pi\in \operatorname{S}_n}\widetilde{\operatorname{O}}(\pi)$, the desired result follows.
\end{proof}
\subsection{Proof of Lemma~\ref{lem-HG-control}}
\begin{proof}%[Proof of Lemma~\ref{lem-HG-control}]
Recall that 
\[
I_s=\Big[|\operatorname{N}_s|p+\sqrt{2(1-\eta)|\operatorname{N}_s|p\log n},|\operatorname{N}_s|p+\sqrt{10|\operatorname{N}_s|p\log n}\Big]\,.
\]
We denote %$x^*=|\operatorname{N}_s|p+\sqrt{20|\operatorname{N}_s|p\log n}$, 
$p^*=\sqrt{100p\log n/|\operatorname{N}_s|}\le \sqrt{100/\eta}\cdot p_c\ll p$, and $\Pb^*$ as the  product measure of $|\operatorname{N}_s|$ independent Bernoulli variables $\mathbf B(1,p+p^*)$ on $\Omega_1$. 
 For any $x\in I_s$, we claim that there exists a coupling $\{(\omega,\omega'):\omega\sim \Pb_{s,x},\omega'\sim\Pb^*\}$ such that with probability $1-o(1/n)$, it holds $\omega_j\le \omega_j'$ for all $j\in \operatorname{N}_s$.

To prove the existence of such a coupling, we start with the following observation: 
for $\omega=(g_j,\dots)_{j\in \operatorname{N}_s}$ sampled from $\Pb_{s,x}$ (resp. $\Pb^*$), given that $\sum_{j\in \operatorname{N}_s}g_j=K$, the conditional distribution of $\omega$ can be characterized by uniformly sampling a subset $A$ of $\operatorname{N}_s$ with $|A|=K$ and setting $g_j=\mathbf{1}_{j\in A}$. It is easy to see that for any $K_1\le K_2\le |\operatorname{N}_s|$, there exists a coupling between two uniformly sampled subsets $A_1,A_2\subset \operatorname{N}_s$ with $|A_i|=K_i$ for $i=1,2$, such that $A_1\subset A_2$ holds with probability one.

Based on these observations, it suffices to show the existence of a coupling $(\omega,\omega')$ between $\Pb_{s,x}$ and $\Pb^*$ such that with probability $1-o(1/n)$, $\sum_{j\in \operatorname{N}_s}\omega_j\le \sum_{j\in \operatorname{N}_s}\omega_j'$. This can be verified as follows: with $y=|\operatorname{N}_s|p+\sqrt{16|\operatorname{N}_s|p\log n}$, it holds
\[
\Pb_{s,x}[E_{r,s}\ge y]\le \frac{\Pb[E_{r,s}\ge y]}{\Pb[E_{r,s}\ge x]}\le \frac{\exp(-(8+o(1))\log n)}{\exp(-(5+o(1))\log n}=o(1/n)\,,
\]
and (since $\mathbb{E}_{\Pb^*} [E_{r,s}]\ge |\operatorname{N}_s|p+10\sqrt{|\operatorname{N}_s|p\log n}\ge y+6\sqrt{|\operatorname{N}_s|p\log n}$) 
\[
\Pb^*[E_{r,s}\le y]\le \exp(-(18+o(1)\log n)=o(1/n)\,.
\]
Hence, the desired coupling exists.

Note that the event $\{\omega_1\notin \Omega_s^*\}$ is an increasing event. Therefore, with the aforementioned coupling, it suffices to show that $\Pb^*[\omega_1\notin \Omega_s^*]=o(1/n)$. It is clear that under $\Pb^*$,  $\sum_{j\in \operatorname{N}_k\cap \operatorname{N}_s}g_j\sim \mathbf B(|\operatorname{N}_k\cap \operatorname{N}_s|,p+p^*)$. Recall that under $\Gc_0$, we have $|\operatorname{N}_k\cap \operatorname{N}_s|\le 2np^2$ for all $a_\eta<k<s$. The remaining proof is just simple applications of binomial tail estimates and then taking union bounds, which we omit the detailed calculations.
\end{proof}
\subsection{Proof of Lemma~\ref{lem-asmptotic-of-tail}}
\begin{proof}%[Proof of Lemma~\ref{lem-asmptotic-of-tail}]
    Since by definition, 
    $$F_k^*(O_k)=\Pb\Big[\sum_{j\in \operatorname{N}_k}G^*_{r^*,\pi^*(j)}\le O_k\Big]\le \Pb\Big[\sum_{j\in \operatorname{N}_k}G_{r^*,\pi^*(j)}\le O_k-1\Big]\,,$$ it suffices to show 
    \[
    \Pb\Big[\sum_{j\in \operatorname{N}_k\setminus \operatorname{N}_s}G_{r^*,\pi^*(j)}\leq
            O_k-T_k-1\Big]
            \ge  \Pb\Big[\sum_{j\in \operatorname{N}_k}G_{r^*,\pi^*(j)}\le O_k-1\Big]-\frac{2\delta(n)(c_k+1)}{n-k+1}\,.
    \]
    %which is then lower bounded by $F_k^*(O_k)-o(1/n)$ by definition.
    It is clear that $$\sum_{j\in \operatorname{N}_k\setminus\operatorname{N}_s}G_{r^*,\pi^*(j)}\stackrel{\text{def}}{=}S_{k,s}\sim\mathbf{B}(|\operatorname{N}_k\setminus\operatorname{N}_s|,p
    )\,,\quad \sum_{j\in \operatorname{N}_k}\Gs_{r^*,\pi^*(j)}\stackrel{\text{def}}{=}P_k\sim \mathbf B(|\operatorname{N}_k|,p)\,. $$
    We introduce an auxiliary random variable $Y\sim\mathbf{B}(|\operatorname{N}_k\cap\operatorname{N}_s|,p)$ which is independent of everything before, then
    \begin{equation}\label{eq-OkTk-approx}
    %\begin{split}
         \Pb\Big[\sum_{j\in \operatorname{N}_k\setminus \operatorname{N}_s}G_{r^*,\pi^*(j)}\leq
            O_k-T_k-1\Big]=\Pb\left[S_{k,s}+Y\leq O_k-T_k-1+Y\right]\,.%\\
        %=&\Pb[\mathbf{B}(|\operatorname{N}_k\setminus\operatorname{N}_s|,p)+Y\leq O_k-T_k+Y]\,.\\
    %\end{split}
    \end{equation}
    The analysis of \eqref{eq-OkTk-approx} varies in details for $p$ in different regimes. %  The key is to control the magnitude of $Y$ and to compare the tail probability. 
    We take $S_k$ to be
    \begin{equation}\label{eq-Sk}
        S_k=
        \begin{cases}
            |\operatorname{N}_k\cap\operatorname{N}_s|p-\sqrt{10np^3\log n}\,,& np^3\geq (\log n)^3\mbox{ and }p\ll1/(\log n)^4\,,\\
            0\,,&np^3\ll (\log n)^3\text{ and }p\gg p_c\,.
            \end{cases}
    \end{equation}
    From \eqref{eq-chernoff-bound-lower} and the fact that $|\operatorname{N}_k\cap \operatorname{N}_s|\le 2np^2$ we see $\Pb[Y\geq S_k]=1-o(1/n^2)$. Noting that $S_{k,s}+Y\sim \mathbf{B}(|\operatorname{N}_k|,p)$, thus the left hand side of \eqref{eq-OkTk-approx} is lower bounded by 
    \begin{equation}\label{eq-add-Y}
    %\begin{split}
        %&\Pb[\mathbf{B}(|\operatorname{N}_k\setminus\operatorname{N}_s|,p)+Y\leq O_k-T_k+Y,Y\geq S_k]\\
        \Pb[\mathbf{B}(|\operatorname{N}_k|,p)\leq O_k-T_k-1+S_k]-o(1/n^2)\,.%\\
        %\geq &\Pb[\mathbf{B}(|\operatorname{N}_k|,p)\leq O_k-T_k+S_k]-o(1/n)\,.
    %\end{split}
    \end{equation}
    Recall that $\delta(n)=2/\sqrt{\log n}$, in order to prove the claim, it suffices to show that 
    \begin{equation}
        \label{eq-compare-tail-SkTk}
        \frac{\Pb\left[O_k-T_k-1+S_k\leq \mathbf{B}(|\operatorname{N}_k|,p)\leq O_k-2\right]}{\Pb\left[\mathbf{B}(|\operatorname{N}_k|,p)\geq O_k-1\right]}\leq \delta(n)\,,
    \end{equation}
    %for some $\alpha(n)\to 0$ which depends on the regime.
    since this together with \eqref{eq-add-Y} would imply that
    \begin{align*}
        &\Pb[S_{k,s}\le O_k-T_k-1]-\Pb[P_k\le O_k-1]
        \\
        \ge&-\Pb[O_k-T_k-1+S_k\le \mathbf B(|\operatorname{N}_k|,p)\le O_k-2]-o(1/n^2)\\
        \ge&-\delta(n)\Pb[\mathbf B(|\operatorname{N}_k|,p)\ge O_k-1]-o(1/n^2)\\
        \ge&-\frac{2\delta(n)(c_k+1)}{n-k+1}\,,
    \end{align*}
    leading to the desired result.
    Here the last inequality exploits the fact that from the assumption in $\mathcal G_{s-1}$, it holds $$
    \Pb[P_k\ge O_k-1]\ge\Pb[E_{r,k}\ge O_k]\ge \frac{\big(1+o(1)\big)c_k}{n-k+1}\,.
    $$
    %Also note that  $\mathcal{G}_{s-1}$ implies $O_k=|\operatorname{N}_k|p+M_k$ for some  %and that $$\Pb[\mathbf{B}(|\operatorname{N}_k|,p)\geq O_k]=(1+o(1))\frac{c_k}{n-k+1}\,.$$ 
    To prove \eqref{eq-compare-tail-SkTk}, write $O_k=|\operatorname{N}_k|p+M_k$, we compute the binomial one-point density for each $t=|\operatorname{N}_k|p+M_k+r_k$ for $|r_k|\leq R_k$, where
    \begin{equation}
        R_k=
        \begin{cases}
            \sqrt{np^{2.5}\log n}\,,&np^3\geq (\log n)^3\text{ and }p\ll1/(\log n)^4\,,\\
            (\log n)^4\,,&n^{-0.1}\leq np^3\leq (\log n)^3\,,\\
            \sqrt{\log n}\,,&p\gg p_c\text{ and }np^3\leq n^{-0.1}\,.
        \end{cases}
        \notag
    \end{equation}
    Note that $\mathcal G_{s-1}$ implies 
    $\sqrt{2(1-\eta)|\operatorname{N}_k|p\log n}\leq M_k\leq \sqrt{10np^2\log n}$.
    Write $N=|\operatorname{N}_k|$, by Stirling's formula, we have $\Pb[\mathbf B(N,p)=t]=\binom{N}{t}p^t(1-p)^{N-t}$ equals to $1+o(1)$ times
    \begin{equation*}
        \begin{aligned}
            &\ \frac{\sqrt{2\pi N}N^Ne^{-N}}{\sqrt{2\pi t}t^te^{-t}\sqrt{2\pi(N-t)}(N-t)^{N-t}e^{-(N-t)}}p^t(1-p)^{N-t}\\
            =&\ \frac{1}{\sqrt{2\pi}}\exp\left\{N\log N-t\log t-(N-t)\log(N-t)-\frac{1}{2}\log t\right\}p^t(1-p)^{N-t}\\
             =&\ \frac{1+o(1)}{\sqrt{2\pi Np}}\exp\left\{N\log N-t\log t-(N-t)\log(N-t)\right\}p^t(1-p)^{N-t}\\
            =&\ \frac{1+o(1)}{\sqrt{2\pi Np}}\exp\left\{-t\log\left(\frac{t}{Np}\right)-(N-t)\log\left(\frac{N-t}{N(1-p)}\right)\right\}\\
            =&\ \frac{1+o(1)}{\sqrt{2\pi Np}}\exp\left\{-\left(Np+M_k\right)\log\left(1+\frac{M_k+r_k}{Np}\right)-(N(1-p)-M_k)\log\left(1-\frac{M_k+r_k}{N(1-p)}\right)\right\}\\
            =&\ \frac{1+o(1)}{\sqrt{2\pi Np}}\exp\left\{-\left(Np+M_k\right)\log\left(1+\frac{M_k}{Np}\right)-(N(1-p)-M_k)\log\left(1-\frac{M_k}{N(1-p)}\right)\right\}\,,
        \end{aligned}
    \end{equation*}
    where the approximation in the last line is uniform for $r_k$ with $|r_k|\leq R_k$. Let $$\mathcal{E}(k,p)=\left(Np+M_k\right)\log\left(1+\frac{M_k}{Np}\right)+(N(1-p)-M_k)\log\left(1-\frac{M_k}{N(1-p)}\right)\,,$$
    with the aforementioned one-point estimates, we can conclude by only considering the first $R_k$ terms in $\Pb[\mathbf B(N,p)\ge O_k-1]=\sum_{t=-1}^{\infty}\Pb[\mathbf B(N,p)=O_k+t]$ that the denominator in \eqref{eq-compare-tail-SkTk} is lower bounded by $1+o(1)$ times $$\frac{R_k}{\sqrt{2\pi Np}}\exp\left(-\mathcal{E}(k,p)\right)\,.$$ 
    In addition, from the definition of $T_k,S_k$ in \eqref{eq-HG-upper-bound-control} and \eqref{eq-Sk}, we see that for any $t\in [O_k-T_k-1+S_k,O_k-2]$, it holds $|t-O_k|\le R_k$. Therefore, the one-point estimate above can be applied to each term in $\Pb[O_k-T_k-1+S_k\le \mathbf B(N,p)\le O_k-2]=\sum_{t=-T_k-1+S_k}^{-2}\Pb[\mathbf B(N,p)=O_k+t]$,
    which yields that the numerator in \eqref{eq-compare-tail-SkTk} is upper bounded by $1+o(1)$ times $$\frac{T_k-S_k}{\sqrt{2\pi Np}}\exp\left(-\mathcal{E}(k,p)\right)\,.$$  Therefore, uniformly in the three regimes (where we essentially make use of the assumption that $p\le 1/(\log n)^4$), the upper bound of \eqref{eq-compare-tail-SkTk} is given by: $$\big(1+o(1)\big)\cdot\frac{T_k-S_k}{R_k}\leq \frac{1+o(1)}{\sqrt{\log n}}\le \delta(n)\,,$$
    completing the proof. %uniformly for the three regimes (this is where the assumption $p\le 1/(\log n)^4$ is used). This implies $$\Pb[\mathbf{B}(|\operatorname{N}_k|,p)\leq O_k-T_k+S_k]\geq 1-\alpha(n)\frac{c_k}{n-k+1}\,.$$
    %Combining with \eqref{eq-OkTk-approx} and \eqref{eq-add-Y}, we finish the proof.
\end{proof}
\end{appendix}


\bibliographystyle{plain}

\end{document}